\documentclass[a4paper,12pt]{article}
\usepackage[margin=1in]{geometry}

\usepackage{amsthm,amsmath,amsfonts,amssymb} 
\usepackage[authoryear]{natbib} 
\usepackage[colorlinks=true,allcolors=blue]{hyperref} 
\usepackage{graphicx} 

\usepackage[T1]{fontenc} 
\usepackage[utf8]{inputenc} 
\usepackage[english]{babel} 

\usepackage{dsfont} 
\usepackage{braket} 
\usepackage{mathtools} 
\usepackage{mathrsfs} 
\usepackage{bm,bbm} 

\usepackage{microtype} 
\usepackage{multirow,array} 
\usepackage{enumitem} 
\usepackage{caption}
\usepackage{xcolor}
\usepackage{subcaption} 
\usepackage[ruled,noend]{algorithm2e} 
\usepackage[toc,page]{appendix}	
\usepackage{authblk} 

\newtheorem{theorem}{Theorem}[section]
\newtheorem{lemma}[theorem]{Lemma}
\newtheorem{proposition}[theorem]{Proposition}

\theoremstyle{definition}
\newtheorem{definition}[theorem]{Definition}
\newtheorem{example}[theorem]{Example}
\newtheorem{remark}[theorem]{Remark}

\usepackage{setspace}
\def\eqd{\stackrel{\mbox{\scriptsize{d}}}{=}}
\def\simiid{\stackrel{\mbox{\scriptsize{\rm iid}}}{\sim}}

\providecommand{\ddr}{\mathrm{d}}
\providecommand{\crm}{\tilde \mu}
\providecommand{\CRM}{\textup{CRM}}
\providecommand{\id}{\textup{ID}}
\providecommand{\E}{\mathbb{E}}

\providecommand{\R}{\mathbb{R}}
\providecommand{\X}{\mathbb{X}}
\providecommand{\Prob}{\mathbb{P}}
\providecommand{\var}{\textup{Var}} 
\providecommand{\cov}{\textup{Cov}} 
\providecommand{\corr}{\textup{Corr}}
\newcommand{\cut}{\vspace{0pt}}

\begin{document}
	
\title{Hierarchical Random Measures without Tables}
\author[1]{Marta Catalano}
\author[2]{Claudio Del Sole}
\affil[1]{\emph{Luiss University, Rome, Italy}, \ mcatalano@luiss.it}
\affil[2]{\emph{University of Milano-Bicocca, Milan, Italy}, \ claudio.delsole@unimib.it}
\date{}
\renewcommand\Affilfont{\small}
\renewcommand\Authands{ and }

\maketitle
\vspace{-\baselineskip}
	
\begin{abstract}
	\noindent
	The hierarchical Dirichlet process is the cornerstone of Bayesian nonparametric multilevel models. Its generative model can be described through a set of latent variables, commonly referred to as \emph{tables} within the popular restaurant franchise metaphor. The latent tables simplify the expression of the posterior and allow for the implementation of Gibbs sampling algorithms to approximately draw posterior samples. However, managing their assignments can become computationally expensive, especially as the size of the dataset and the number of levels increase. In this work, we identify a prior for the concentration parameter of the hierarchical Dirichlet process that (i) induces a quasi-conjugate posterior distribution, and (ii) removes the need for tables, leading to more interpretable expressions for the posterior, with both a scalable and an exact algorithm to sample from it. Remarkably, this construction extends beyond the Dirichlet process, leading to a new framework for defining normalized hierarchical random measures and a new class of algorithms to sample from their posteriors. The key analytical tool is the independence of multivariate increments, that is, their representation as \emph{completely random vectors}.
	
\end{abstract}

\section{Introduction}

Historically, Bayesian nonparametric and hierarchical models have addressed data complexity in different ways. Nonparametric models ensure full flexibility to the marginal distribution of the observations by increasing the dimensionality of the parameter space, while hierarchical models focus on the interactions between the observations, grouping them and modeling their dependencies through shared parameters or latent features. The interaction between parameters is recursively modeled in a similar fashion, defining a hierarchical structure that enables pooling of information across different groups while preserving their distinct characteristics -- all within the principled framework of Bayesian inference. The hierarchical Dirichlet process \citep{Teh2006} represented a significant breakthrough, demonstrating the advantages of combining the two approaches through the sharing of infinite-dimensional parameters. Since then, it has proved effective in an impressive number of contexts, including natural language processing \citep{Teh2006, Zavitsanos2011}, genomics \citep{Sohn2009, Elliott2019, Liu2024}, computer vision \citep{Sudderth2008, Haines2011}, music segmentation and speaker diarization \citep{Ren2008, Fox2011}, cognitive science \citep{Griffiths2007}, robotics \citep{Nakamura2011, Taniguchi2018}, network analysis \citep{Durante2025}.

The computational feasibility of the hierarchical Dirichlet process in \citet{Teh2006} is  strictly linked to a compelling posterior representation via latent variables, often referred to as \emph{tables} in the restaurant franchise metaphor. The need for tables arises from the nature of the infinite-dimensional parameter, which is an almost surely discrete random probability $\tilde P = \sum_{i \ge 1} J_i \delta_{\theta_i}$, characterized by a countably infinite number of jumps $J_i$ and atoms $\theta_i$. In a Bayesian nonparametric setting, a single group of observations is often modeled as conditionally independent and identically distributed from $\tilde P$, where $\tilde P$ is e.g.~a Dirichlet process \citep{Ferguson1973} with diffuse mean measure $P_0$. In this case, since the atoms $\theta_i$ are independent and identically distributed from $P_0$, two observations coincide if and only if they share the same atom $\theta_i$. Conversely, in the hierarchical Dirichlet process, the distribution $P_0$ of the atoms is itself a latent parameter, which is shared across different groups and likewise modeled as an almost surely discrete random probability. Hence, contrary to the exchangeable case, the atoms display ties with positive probability, and two observations coincide either if they share the same atom or if their respective atoms have identical values. This complicates the posterior representation, unless one keeps track of the atom associated to each observation, which is precisely the role of the latent tables \citep{Sankhya2024}. For $n$ observations, the tables induce a distribution on the space of partitions of $n$ elements; the dimensionality of this space increases dramatically with the number of observations, and one typically needs a Gibbs sampling algorithm to sample from it.

The most popular implementation for the hierarchical Dirichlet process is the Gibbs sampler based on the restaurant franchise metaphor \citep{Teh2006}, often addressed as \emph{marginal} or \emph{collapsed} Gibbs sampler. Despite the convenient expressions of its full conditionals and its remarkable flexibility, it is not an exact algorithm, since it samples from a Markov chain and thus approximates the posterior distribution only asymptotically. In addition, there are other well-known drawbacks: it involves a considerable amount of bookkeeping, and scales poorly as the number of observations $n$ increases, as it relies on a sequential updating scheme for the $n$ latent tables. Moreover, each table allocation depends on all the other allocations, inducing high autocorrelation in the Markov chain, slow mixing, and preventing
parallelization. Such limitations have been recognized by several works; see e.g.~\cite{Teh2006, Teh2010, Williamson2013, Rigon2020, Das2025}.
Accordingly, a plethora of sampling-based strategies have been proposed to reduce computational time and improve mixing and scalability of the standard implementation. These are typically Gibbs samplers adopting a \emph{conditional} or \emph{blocked} approach, that is, instantiating the random measures with some finite-dimensional approximation of the posterior. Direct assignment schemes instantiate the jumps of the common random probability, often addressed as \emph{global weights}, through the stick-breaking construction \citep{Teh2006} or a finite-dimensional Dirichlet approximation \citep{Fox2011}, while variational methods \citep{Teh2007, Wang2011, Bryant2012} construct mean field approximations of stick-breaking ratios at both levels of the hierarchy. Nevertheless, their full conditional distributions still depend on the nested partition induced by the tables, and thus do not completely avoid some of its inherent drawbacks.
Interestingly, the conditional Gibbs samplers proposed in \cite{Rigon2020} and \cite{Das2025} obviate the need for tables by considering a finite-dimensional approximation of the model, i.e.~by truncating the common random probability a priori.

In this work we pursue a different strategy that eliminates the need for tables while preserving the infinite-dimensionality of the model, thanks to a specific gamma hyperprior for the shared concentration parameter of the Dirichlet processes. Intuitively, the hierarchical Dirichlet process defines a vector of dependent random probabilities with independent jumps and common atoms with ties. Our hyperprior allows for an alternative representation as dependent random probabilities with dependent jumps and common atoms without ties, thereby making the tables superfluous. A fundamental result of this work shows that this representation is the normalization of a specific \emph{completely random vector} \citep{Catalano2021}, that is, a vector of dependent random measures with jointly independent increments. This is the natural multivariate extension of a completely random measure \citep{Kingman1967}, whose normalization is very popular in the Bayesian nonparametric literature \citep{Regazzini2003, James2006, James2009}. In particular, \cite{James2009} derive an almost conjugate representation of the posterior, conditionally on a real-valued latent variable $U$ amenable to standard approximate sampling schemes, such as random walk Metropolis-Hastings \citep{Barrios2013}. In this work, we extend the posterior representation to any normalized completely random vector, which in principle could be applied to many other well-established models, including the normalization of GM-dependent measures \citep{lijoi2014bayesian}, L\'evy copulas \citep{EpifaniLijoi2010}, compound random measures \citep{GriffinLeisen2017}, and thinned random measures \citep{LauCripps2022}. 

The characterization as a normalized completely random vector allows to identify a novel posterior representation for the hierarchical Dirichlet process with hyperprior that does not require the latent tables. 
The price to pay for casting off the tables is the introduction of a latent vector and $k$ jump vectors, where $k$ is the total number of distinct observations, typically much smaller than the number of observations $n$. In principle, the dimension of the latent and jump vectors coincides with the number of groups $d$, which could be potentially large, thus slowing down the posterior inference algorithm. However, through an in-depth analysis of their distributions, we reduce the non-standard sampling steps to sampling a random vector supported on $[0,\infty)^{2k+1}$. This approach offers three fundamental advantages: (i) the random vector lives in a standard space, and thus can be approximately sampled with standard Markov Chain Monte Carlo (MCMC) techniques, which show good mixing and can be analyzed with a well-established set of diagnostics; (ii) its dimension does not increase with the number of observations $n$ nor with the number of groups $d$, which leads to scalable algorithms whenever $n$ or $d$ increase but $k$ remains only moderately large; (iii) when $n$ is moderately large, we devise an exact i.i.d.~sampling algorithm, thus avoiding the approximation error of Gibbs sampling procedures. In summary, our novel posterior representation allows for the development of a scalable algorithm when $n$ or $d$ are large compared to $k$, and of an exact i.i.d.~sampling algorithm when $n$ is only moderately large. Note that $k$ consistently smaller than $n$ is the most common setting where the model should be applied, since, e.g., for the hierarchical Dirichlet process, $k$ is of the same order as $\log(\log(n))$.

The hierarchical Dirichlet process with hyperprior is only a particular case of the more general class of nonparametric hierarchical models discussed in this work. These models arise as the normalization of a vector $(\crm_1,\dots, \crm_d)$ of conditionally independent completely random measures, given another completely random measure $\crm_0$. 
We show that these dependent random measures are completely random vectors and, for this reason, we term them  \emph{hierarchical completely random vectors}. Such hierarchical structures have been used to model dependent hazards \citep{Camerlenghi2021, DelSole2026} 
and beta processes \citep{ThibauxJordan2007, Masoero2018, James2024}, but have never been considered for normalization. This 
novel specification merges two compelling properties for the first time: the naturalness of the hierarchical construction and the analytical tractability of independent multivariate increments. The former leads to simple representations and draws interesting parallels with popular hierarchical models in the literature 
\citep{Teh2006, Camerlenghi2018}. The latter enhances the theoretical investigation of the model, 
leading to the almost-conjugate posterior representation we have already discussed, 
and playing a crucial role in deriving closed-form expressions for the prior moments and dependence structure. 

The work is structured as follows. 
Section~\ref{sec:hcrv} introduces and characterizes hierarchical completely random vectors (hCRVs). Properties of their normalization are studied in Section~\ref{sec:normalization}, where we clarify the connections with other hierarchical models in the literature. Section~\ref{sec:posterior} provides a characterization of the posterior distribution for normalized hCRVs. These results are specialized to the gamma-gamma model in Section~\ref{sec:gamma-gamma}, where we derive our new posterior sampling algorithms and perform relevant comparisons in terms of computational complexity on simulated datasets. Finally, Section~\ref{sec:comparison} sets rules for a fair comparison with the hierarchical Dirichlet process. Using both simulated and real datasets, we argue that a meaningful comparison between 
dependent priors should entail matching the first two marginal moments and measure of dependence. Proofs and technical details for the sampling algorithms are deferred to the Supplementary Material. 

\paragraph{Notation} 
The Cartesian product of $d$ copies of a set A is denoted as $A^d = A \times \dots \times A$. The product measure of $d$ probabilities $P_1,\dots,P_d$ is $\prod_{i=1}^d P_i$, or $P^d$ if $P_1 = \dots = P_d = P$.
Bold symbols indicate vectors, such as $\bm{s} = (s_1,\dots,s_d)$, and $\ddr \bm{s} = \ddr s_1 \cdot \dots \cdot \ddr s_d$ is the Lebesgue measure on $\R^d$. The normal distribution with mean $\mu$ and variance $\sigma^2$ is denoted by $N(\mu,\sigma^2)$.
If $f : \mathbb{X} \to \mathbb{Y}$ is a measurable map and $\rho$ is a measure on $\mathbb{X}$, 
then $f_{\#} \rho$ is the pushforward measure on $\mathbb{Y}$ defined as $(f_{\#} \rho)(B) = \rho(f^{-1}(B))$, for any measurable set $B$. The symbol $\sim$ underlines the randomness of a random measure ($\crm$). 
The abbreviation a.s.~stands for \emph{almost surely}; we repeatedly use $\Omega_d = [0,+\infty)^d \setminus \{\bm{0}\}$. Contents in the Supplementary Material 
are labeled with an ‘S’ prefix, e.g. Section S1.


\section{Hierarchical completely random vectors}
\label{sec:hcrv}

In this section, we recall the definition of hierarchical random measures, highlight that they are homogeneous completely random vectors, and recover both their multivariate Laplace exponent and their L\'evy measure. The proofs follow a structure similar to the \emph{subordination} of L\'evy processes, first defined in \cite{Bochner1955}, and beautifully described in  \cite{Bertoin1996} and \cite{Sato1999}. Their joint L\'evy measure displays some similarities with that of compound random measures \citep{GriffinLeisen2017}, but we show that there is no intersection between the two classes, at least for the same choice of the subordinating random measure. We note that hierarchical completely random vectors can be approximately sampled with standard techniques, exploiting their hierarchical structure.
The specification of their law involves an outer and an inner L\'evy measure, whose identifiability is studied with relevant examples.

A vector of random measures $\bm{\crm} = (\crm_1,\dots,\crm_d)$ is a measurable function on $M_{\mathbb{X}}^d$, where $M_{\mathbb{X}}$ denotes the space of boundedly finite measures on a Polish space $\mathbb{X}$. A completely random vector \citep{Catalano2021} is the natural multivariate generalization of a completely random measure (CRM), defined in \citet{Kingman1967}. For a Borel set $A$ of $\X$, we use the notation $\bm{\crm}(A) = (\crm_1(A),\dots,\crm_d(A))$, which is a random vector in $\R^d$.

\begin{definition}
	\label{def:crv}
	A vector of random measures $\bm{\crm} = (\crm_1,\dots,\crm_d)$ is a completely random vector ({\rm CRV}) if, given pairwise disjoint Borel sets $A_1,\dots,A_k$ of $\X$, the random vectors $\bm{\crm}(A_1),\dots,\bm{\crm}(A_k)$ are mutually independent. 
\end{definition}

\noindent
Refer to Section~\ref{app:background} for a brief and self-contained account on completely random measures, L\'evy measures, L\'evy intensities, Laplace exponents, and their multivariate extension to completely random vectors. Henceforth, $\crm \sim \CRM(\rho\otimes P_0)$ denotes a CRM with product L\'evy intensity $\ddr \rho (s) \, \ddr P_0(x)$, and $\id(\rho)$ indicates a pure-jump infinitely divisible distribution with L\'evy measure $\rho$; its expression in integrals is $\ddr P_{\scriptscriptstyle \id(\rho)}$ and its probability density function (p.d.f), if it exists, is denoted by $f_{\scriptscriptstyle \id( \rho)}(s)$.

\begin{definition} \label{def:hcrv}
	Let $\rho_0$ and $\rho$ be L\'evy measures on $(0,+\infty)$ and let $P_0$ be an atomless measure. We say that $\bm{\crm}= (\crm_1,\dots,\crm_d)\sim {\rm hCRV}(\rho, \rho_0, P_0)$ is a hierarchical {\rm CRV} with idiosyncratic L\'evy measure $\rho$, base L\'evy measure $\rho_0$, and base measure $P_0$ if
	\begin{equation*}
		\crm_1,\dots,\crm_d \mid \crm_0 \simiid \CRM(\rho \otimes \crm_0); \qquad \crm_0 \sim \CRM(\rho_0 \otimes P_0).
	\end{equation*}
\end{definition}

The next theorem shows that the vector of random measures $\bm{\crm} \sim \mathrm{hCRV}(\rho,\rho_0,P_0)$ is in fact a completely random vector in the sense of Definition~\ref{def:crv}. Moreover, the theorem provides the expression of its multivariate Laplace functional through the Laplace exponent and determines its multivariate L\'evy intensity; since $\bm{\crm}$ is a CRV, such intensity automatically satisfies the integrability conditions (i) and (ii) recalled in Section~\ref{app:background}. This construction admits a natural interpretation in terms of subordination. Indeed, the proof follows the same structure as the classical subordination of L\'evy processes \citep{Bochner1955,Bertoin1996,Sato1999}, and hierarchical CRVs may be regarded as the corresponding random-measure analogue. A related construction, based on the subordination of CRMs by infinitely divisible random measures, is also developed in \cite{Brueck2026}.

\begin{theorem} 
	\label{th:crv}
	Let $\bm{\crm} \sim {\rm hCRV}(\rho, \rho_0, P_0)$ with $P_0$ a probability measure. Then $\bm{\crm}$ is a homogeneous \emph{CRV} with Laplace exponent $\psi_h:[0,+\infty)^d \mapsto [0,+\infty)$ and L\'evy intensity $\nu_{h} = \rho_{h} \otimes P_0$ such that, for every $\bm{\lambda} = (\lambda_1,\dots,\lambda_d) \in [0,+\infty)^d$ and $\bm{s} = (s_1,\dots,s_d)\in \Omega_d$,
	\begin{equation*}
		\psi_{h}(\bm{\lambda}) = \psi_0\bigg(\sum_{i=1}^d \psi(\lambda_i) \bigg), \qquad \ddr \rho_{h}({\bm s}) = \int_0^{+\infty} \prod_{i=1}^d  \ddr P_{\scriptscriptstyle \id(t \rho)}( s_i)  \, \ddr\rho_0(  t).
	\end{equation*}
\end{theorem}

\noindent
In particular, if $\id(t \rho)$ has p.d.f.~$f_{ \scriptscriptstyle \id(t \rho)}$, then $\rho_{h}$ has a L\'evy density
\begin{equation*}
	\rho_{h}(s_1,\dots,s_d) = \int_0^{+\infty} \prod_{i=1}^d f_{\id(t \rho)}(s_i) \, \ddr \rho_0(t).
\end{equation*}
A sufficient condition for $\id(t \rho)$ to have a probability density function, for every $t > 0$, is that the L\'evy measure $\rho$ has infinite mass and is absolutely continuous \cite[Theorem 27.7]{Sato1999}. 

\begin{remark}
	\label{rem:compound}
	The structure of the L\'evy density above resembles that of \emph{compound random measures} introduced in \cite{GriffinLeisen2017},  \cut
	\begin{equation*}
		\rho_{\rm co}(s_1,\dots,s_d) = \int_0^{+\infty} \frac{1}{t^d} \, H\left( \frac{s_1}{t},\dots,\frac{s_d}{t} \right) \ddr \rho_0(t),
	\end{equation*}
	where $H$ is a p.d.f.~on $[0,+\infty)^d$.
	However, we show in Section~\ref{proof_remark} that there is no intersection between the two classes, at least for the same choice of 
	base L\'evy measure $\rho_0$.
\end{remark}

\begin{remark} \label{rmk:sampling}
	Sampling from $\bm{\crm} \sim {\rm hCRV}(\rho, \rho_0, P_0)$ can be easily performed exploiting its hierarchical structure. Firstly, the base random measure $\crm_0$ can be sampled with different standard techniques; notably, when $\rho_0$ has infinite mass, $\crm_0$ comprises a countably infinite number of jumps, and we may only obtain an approximate sample.
	The most common strategies for infinitely active CRMs are based on the Ferguson-Klass representation \citep{Ferguson1972}, and sequentially sample their jumps in decreasing order, up to a certain truncation level $L$, yielding the discrete approximation 
	\begin{equation*}
		\tilde \mu_0^{\mathrm{approx}} = \sum_{\ell=1}^L \omega_{0\ell}\,\delta_{\phi_\ell},
	\end{equation*}
	with each $\phi_\ell \sim P_0$ independently.
	The practical implementations of such algorithms rely on the inversion of the tail of the L\'evy measure \citep{Wolpert1998,walker2000} or rejection sampling from a dominating L\'evy density \citep{Rosinski2001};
	we refer to \cite{Campbell2019,Zhang2024} for further details and alternative sampling techniques. Secondly, exploiting this approximation for $\crm_0$, the L\'evy intensity of each $\crm_i \mid \crm_0^\mathrm{approx}$ can be decomposed as $\nu_i = \sum_{\ell = 1}^L \omega_{0\ell}\rho \, \delta_{\phi_\ell}$.
	Since the sum of L\'evy measures corresponds to a sum of independent CRMs, and a CRM $\crm$ with atomic base measure $P_0 = \delta_{\phi}$ satisfies $\crm = \crm(\X) \delta_{\phi}$, then
	$$ \crm_i \eqd \sum_{\ell=1}^L \omega_{i\ell}\,\delta_{\phi_\ell},$$
	where each $\omega_{i\ell} \sim {\rm ID}(\omega_{0\ell}\rho)$ independently. In other words, conditionally on a discrete approximation of the measure $\crm_0$, the vector $\bm{\crm}$ can be sampled exactly, provided that an exact simulation strategy for the infinitely divisible distribution ${\rm ID}(t\rho)$ for $t > 0$ is available. 
	A detailed description of our sampling algorithms for the gamma-gamma hCRV is given in Sections~\ref{app:sampling_random_measure} and~\ref{app:expint},
	tailored to posterior sampling thanks to conditional conjugacy (Proposition~\ref{th:posterior_hcrv}).
\end{remark}

The next result studies the identifiability of the parameters of a hierarchical CRV. For $c \in \mathbb{R}$, denote by $c_{\#}$ the pushforward measure of the multiplication map $s \mapsto cs$.

\begin{theorem}
	\label{th:identifiability}
	Let $\bm{\crm}^{(\ell)} \sim {\rm hCRV}(\rho^{(\ell)}, \rho_0^{(\ell)}, P_0^{(\ell)})$, for $\ell=1,2$. Then $\bm{\crm}^{(1)} = \bm{\crm}^{(2)}$ in distribution if and only if there exists $c>0$ such that
	\begin{equation*}
		\rho_0^{(2)} = (c^{-1})_{\#} \rho_0^{(1)} ,\qquad  \rho^{(2)} = c \rho^{(1)}, \qquad P_0^{(1)} = P_0^{(2)}.
	\end{equation*}
	In particular, if $\rho^{(\ell)}$ and $\rho_0^{(\ell)}$ have L\'evy densities, for $\ell=1,2$, this is equivalent to
	\begin{equation*}
		\rho_0^{(2)}(s) = c \rho_0^{(1)} (cs),\qquad  \rho^{(2)}(s) = c \rho^{(1)} (s), \qquad P_0^{(1)} = P_0^{(2)}.
	\end{equation*}
\end{theorem}

\noindent
The last condition can be easily checked on specific classes of models; for example, if we restrict to $\rho_0 = \alpha \rho$ for some $\alpha>0$, then $\rho$ and $P_0$ are identifiable. To clarify the interpretation of the identifiability conditions, note that $\rho_0^{(2)}(s) = c \rho_0^{(1)} (cs)$ if and only if $\crm_0^{(1)} = c \crm_0^{(2)}$.

We conclude this section with two leading examples of hierarchical CRVs.

\begin{example} 
	\label{ex:gamma-gamma}
	We term $\bm{\tilde \mu}$ a gamma-gamma hierarchical CRV if there exist shape parameters $\alpha, \alpha_0>0$, rate parameters $b,b_0>0$, and $P_0$ a base measure such that
	\begin{equation*}
		\crm_1,\dots, \crm_d \mid \crm_0 \simiid \CRM\bigg(\alpha \frac{e^{-bs}}{s} \ddr s \otimes \crm_0\bigg); \qquad \crm_0 \sim \CRM\bigg(\alpha_0 \frac{e^{-b_0 s}}{s} \ddr s \otimes P_0 \bigg).
	\end{equation*}
	By applying Theorem~\ref{th:crv} and the expression of the $1$-dimensional Laplace exponent in Definition~\ref{def:gamma}, the multivariate Laplace exponent of $\bm{\tilde \mu}$, for $P_0$ 
	a probability measure, is
	\begin{equation*}
		\psi_h(\lambda_1,\dots,\lambda_d) = \alpha_0 \log\bigg( 1+ \frac{\alpha}{b_0} \sum_{i=1}^d \log \bigg( 1 + \frac{\lambda_i}{b} \bigg) \bigg).
	\end{equation*}
	Therefore, different gamma-gamma hCRVs coincide in distribution if they have the same ratio $\alpha/b_0$. Moreover, since $t \rho(s) = t\alpha\, s^{-1} e^{-bs}$ is the L\'evy density of a gamma CRM with shape parameter $t \alpha$ and scale parameter $b$, by Theorem~\ref{th:crv} the multivariate L\'evy density of $\bm{\tilde \mu}$ is 
	\begin{equation*}
		\rho_h(s_1,\dots,s_d) = \alpha_0 \, e^{-b \sum_{i=1}^d s_i} \int_0^{+\infty} \frac{b^{ dt}}{\Gamma(t)^d} \prod_{i=1}^d s_i^{t-1} \, \frac{e^{-(b_0/\alpha)\,t}}{t} \, \ddr t.
	\end{equation*}
\end{example}

\begin{example} 
	\label{ex:stable-stable}
	We term $\bm{\tilde \mu}$ a stable-stable hierarchical CRV if there exist shape parameters $\alpha, \alpha_0>0$, discount parameters $\sigma,\sigma_0 \in (0,1)$, and a base measure $P_0$ such that
	\begin{align*}
		\crm_1,\dots, \crm_d \mid \crm_0 & \simiid \CRM\bigg(\frac{\alpha\,\sigma}{\Gamma(1-\sigma)} \frac{1}{s^{1+\sigma}} \ddr s \otimes \crm_0\bigg); \\
		\crm_0 & \sim \CRM\bigg(\frac{\alpha_0\sigma_0}{\Gamma(1-\sigma_0)}\frac{1}{s^{1+\sigma_0}}  \ddr s \otimes P_0 \bigg).
	\end{align*}
	Theorem~\ref{th:crv} and
	Definition~\ref{def:stable} imply that the multivariate Laplace exponent of $\bm{\tilde \mu}$ is
	\begin{equation*}
		\psi_h(\lambda_1,\dots,\lambda_d) = \alpha_0 \, \alpha^{\sigma_0} \, (\lambda_1^\sigma + \dots + \lambda_d^\sigma)^{\sigma_0}.
	\end{equation*}
	Therefore, different stable-stable hCRVs coincide in distribution if they have the same value for $\alpha_0 \, \alpha^{\sigma_0}$. For $d=1$, we retrieve the Laplace functional of the marginal $\crm_i$, namely $\psi(\lambda) = \alpha_0 \, \alpha^{\sigma_0} \, \lambda^{\sigma \sigma_0}$, which is the Laplace exponent of a stable CRM with shape $\alpha_0\, \alpha^{\sigma_0}$ and discount parameter $\sigma \sigma_0$. Remarkably, we recover the well-known fact that the subordination of a stable L\'evy process with a stable process is again a stable process \citep{Bertoin1996, Sato1999, Camerlenghi2018}. An explicit expression for the multivariate L\'evy density is available only for $\sigma = 1/2$, as it requires the density of the stable infinitely divisible distribution. In such case, the L\'evy density $t\rho(s) = t\alpha\, (4\pi)^{-1/2} s^{-3/2} $ is that of a L\'evy distribution, and 
	\begin{align*}
		\rho_h(s_1,\dots,s_d) = \frac{\alpha_0 \, \alpha^{\sigma_0} \sigma_0}{\pi^{d/2}\, 2^{\sigma_0+1}} \frac{\Gamma( \frac{d-\sigma_0}{2})}{\Gamma(1-\sigma_0)} \prod_{i=1}^ds_i^{-3/2} \bigg( \sum_{i=1}^d \frac{1}{s_i} \bigg)^{-\frac{d-\sigma_0}{2}}.
	\end{align*}
	
\end{example}


\section{Normalized hierarchical completely random vectors}
\label{sec:normalization}

One of the most common uses of completely random measures in Bayesian statistics is their normalization \citep{Regazzini2003}, which defines random probabilities whose law can act as nonparametric priors. The same construction can be extended to vectors of dependent random measures, such as hierarchical CRVs. In this section, we provide conditions for the normalization to be well-defined and investigate connections with popular models, such as the hierarchical Dirichlet process \citep{Teh2006, Camerlenghi2019} and the hierarchical normalized $\sigma$-stable process \citep{Camerlenghi2019}. Moreover, we discuss general techniques to measure the dependence of normalized hierarchical CRVs.

For ${\bm \crm} \sim {\rm hCRV}(\rho, \rho_0,P_0)$, we derive a vector of dependent random probabilities as
\begin{equation}
	\label{eq:norm}
	\frac{\bm{\crm}}{\bm{\crm}(\X)} := \bigg(\frac{\crm_1}{\crm_1(\X)}, \dots , \frac{\crm_d}{\crm_d(\X)}\bigg) , 
\end{equation}
which is well-defined if $0<\crm_i<+\infty$ a.s., for $i=1,\dots,d$. The upper bound forces $P_0$ to be a finite measure; thus, we can assume without loss of generality that $P_0$ is a probability measure. The lower bound forces each $\crm_i$ to be infinitely active, that is, the corresponding L\'evy measures to have infinite mass. 

\begin{lemma}
	\label{th:infinite_activity}
	Let $\bm{\crm} \sim {\rm hCRV}(\rho, \rho_0, P_0)$. Then each $\crm_i$ is infinitely active if and only if
	\begin{equation*}
		\int_0^{+\infty} \ddr \rho_0 (t) = \int_0^{+\infty} \ddr \rho (t)  = +\infty.
	\end{equation*}
\end{lemma}

\noindent Therefore, $\crm_i$ is infinitely active if and only if both $\crm_0$ and $\crm_i \mid \crm_0$ are infinitely active.

\begin{remark} 
	\label{rem:fubini-tonelli} The main subtlety of the proof of Lemma~\ref{th:infinite_activity} is that $\rho$ has a finite mass if and only if $\id(t \rho)$ gives positive probability to $\{0\}$, as 
	shown in \cite{Regazzini2003}; in this case,
	\begin{equation*}
		\int_{\Omega_d} \ddr P_{\scriptscriptstyle \id(t \rho)}^d < \int_{[0,+\infty)^d} \ddr P_{\scriptscriptstyle \id(t \rho)}^d = 1.
	\end{equation*}
	Hence, we need $\rho$ to have infinite mass to conclude that $\int_{\Omega_d} \ddr \rho_{h}(\bm{s}) = \int_0^{+\infty} \ddr \rho_0(t)$ by Theorem~\ref{th:crv} and Fubini-Tonelli theorem.
\end{remark}

The construction in \eqref{eq:norm} is similar to the normalized hierarchical model in \cite{Camerlenghi2019, Sankhya2024}, where, however, the base random measure is normalized as well, that is
\begin{equation}
	\label{def:herarchical_structure}
	\crm_1,\dots,\crm_d \mid \crm_0 \simiid \CRM\bigg(\rho \otimes \frac{\crm_0}{\crm_0(\X)} \bigg); \qquad \crm_0 \sim \CRM(\rho_0 \otimes P_0).
\end{equation} 
This slight modification has crucial implications on the overall law of $\bm{\crm}$, which is no longer a CRV, and whose marginals $\crm_i$'s are not CRMs. Interestingly, at least two popular hierarchical specifications can be expressed in terms of a normalized hierarchical CRV, as discussed in the following.

Recall that $\bm{\tilde P} = (\tilde P_1,\dots,\tilde P_d) \sim {\rm HDP}(\alpha, \alpha_0,P_0)$ is a hierarchical Dirichlet process \citep{Teh2006} with concentration parameters $\alpha, \alpha_0>0$ and base probability $P_0$ if
\begin{equation}
	\label{def:hdp}
	\tilde P_1,\dots,\tilde P_d \mid \tilde P_0 \simiid {\rm DP}( \alpha \tilde P_0); \qquad \tilde P_0 \sim {\rm DP}(\alpha_0 P_0),
\end{equation}
where ${\rm DP}(\alpha_0 P_0)$ denotes a Dirichlet process \citep{Ferguson1973} with base measure $\alpha_0 P_0$. In fact, the normalization of the gamma-gamma hCRV in Example~\ref{ex:gamma-gamma} recovers the HDP with a specific gamma prior on the concentration parameter. Here, ${\rm Gamma}(a,b)$ denotes the gamma distribution with shape $a$ and rate $b$. 

\begin{proposition} 
	\label{th:hdp_hyper}
	For parameters $\alpha,\alpha_0>0$ and $b, b_0 > 0$, and base measure $P_0$, let $\bm{\crm}$ and $\bm{\tilde P}$ be vectors of random measures such that 
	\begin{gather*} 
		\crm_1,\dots, \crm_d \mid \crm_0 \simiid {\rm CRM}\bigg( \alpha  \frac{e^{-bs}}{s} \ddr s \otimes \crm_0 \bigg); \qquad \crm_0 \sim {\rm CRM}\bigg( \alpha_0  \frac{e^{-b_0 s}}{s} \ddr s \otimes P_0 \bigg), \\
		\tilde P_1,\dots,\tilde P_d \mid \tilde \alpha \simiid {\rm HDP}(\tilde \alpha, \alpha_0, P_0); \qquad \tilde \alpha \sim {\rm Gamma}(\alpha_0, b_0/\alpha). 
	\end{gather*}
	Then, with the notation in \eqref{eq:norm}, it holds that $\bm{\crm}/\bm{\crm}(\X) \eqd \bm{\tilde P}$.
\end{proposition}

\noindent
The distribution of a normalized gamma-gamma hCRV depends in fact only on $\alpha_0$ and $\alpha/b_0$. The role of the ratio $\alpha/b_0$ for identifiability is highlighted in Example~\ref{ex:gamma-gamma}, while $b$ is a scale parameter for $\bm{\crm}$ and disappears with normalization; see also Figure~\ref{fig:structure} and Section~\ref{app:normalized_jumps}. In practice, one may restrict to $b = b_0 = 1$, without loss of generality.

Additionally, when the idiosyncratic component is a stable CRM, the normalized hierarchical model of \cite{Camerlenghi2019} can be expressed as a normalized hierarchical CRV. This result generalizes, with a different technique, a result of \cite{Camerlenghi2018} which assumes the base CRM to be stable as well; this same fact is also observed for L\'evy processes in \cite{Bertoin1996}. The key property is that, for a stable CRM $\crm$ with L\'evy measure $\rho$, any proportional measure $c \crm$ with $c>0$ is a stable CRM with proportional L\'evy measure $ c' \rho$, for some $c'>0$.

\begin{proposition} 
	\label{th:stable}
	Let $\bm{\crm} ^{(1)}$ and $\bm{\crm}^{(2)}$ be two vectors of random measures defined by 
	\begin{align*} 
		\crm^{(1)}_1,\dots, \crm^{(1)}_d \mid \crm_0 &\simiid {\rm CRM}\bigg( \frac{\alpha \,\sigma}{\Gamma(1-\sigma)} \frac{1}{s^{\sigma+1}}  \ddr s \otimes \crm_0 \bigg); \\
		\crm^{(2)}_1,\dots,\crm^{(2)}_d \mid \crm_0 &\simiid {\rm CRM}\bigg(\frac{\alpha \,\sigma}{\Gamma(1-\sigma)} \frac{1}{s^{\sigma+1}}  \ddr s \otimes \frac{\crm_0}{\crm_0(\X)} \bigg),
	\end{align*}
	where $\alpha>0$, $\sigma \in (0,1)$, and $\crm_0$ is an infinitely active {\rm CRM}. Then, 
	\begin{equation*}
		\frac{\bm{\crm}^{(1)}}{\bm{\crm}^{(1)}(\X)} \eqd \frac{\bm{\crm}^{(2)}}{\bm{\crm}^{(2)}(\X)}.
	\end{equation*} 
\end{proposition}

\begin{remark}
	There is a strong connection between hCRVs and the hierarchical models in \eqref{def:herarchical_structure}, even beyond the two specifications discussed in Proposition~\ref{th:hdp_hyper} and Proposition~\ref{th:stable}. Indeed, conditionally on $\crm_0$, a hierarchical CRV satisfies
	\begin{equation*}
		\crm_1,\dots,\crm_d \mid \crm_0 \simiid \CRM(\rho \otimes \crm_0 ) \equiv \CRM \bigg(\crm_0(\X)\, \rho \otimes \frac{\crm_0}{\crm_0(\X)} \bigg);
	\end{equation*}
	therefore, a hierarchical CRV can be regarded as a hierarchical model in \eqref{def:herarchical_structure} whose idiosyncratic L\'evy measure is endowed with a hyperprior on the concentration parameter, that is, 
	\begin{equation*}
		\crm_1,\dots,\crm_d \mid \crm_0 \simiid  \CRM(\tilde \alpha \rho \otimes \tilde P_0 ), \qquad (\tilde \alpha, \tilde P_0) = \bigg(\crm_0(\X), \, \frac{\crm_0}{\crm_0(\X)}\bigg).
	\end{equation*}
	The subtlety of this hyperprior with respect to standard proposals is that, in general, it is dependent on the base measure $\tilde P_0$, since $\crm_0(\X)$ and $\crm_0/\crm_0(\X)$ are dependent. The only exception is the gamma CRM, which leads to Proposition~\ref{th:hdp_hyper}.
\end{remark}

In the following, we provide closed form expression for summary statistics encoding the mean, the variance, and the dependence structure of hCRVs. The simplest and most widely used measure of dependence between random probabilities $\bm \tilde P = \bm \crm/\bm \crm(\X)$ is the pairwise linear correlation $\corr({\tilde P_i}(A),{\tilde P_j}(A))$, for a Borel set $A$. Note that its expression does not depend on $A$, since normalized hCRVs are normalized homogeneous CRVs and thus belong to the class of multivariate species sampling processes \citep{Franzolini2025}. At the level of the random measures $\bm \crm$, its computation is a slight modification of \citet[Section~8]{Sankhya2024}, where the expressions are derived by leveraging on the conditional independence structure. 
We report these results in Proposition~\ref{th:moments}, and present an alternative proof that builds on their joint infinite divisibility. Interestingly, this second technique greatly simplifies the derivation in the normalized case.

\begin{proposition}
	\label{th:moments_norm}
	Let $\bm{\crm}\sim {\rm hCRV}(\rho, \rho_0, P_0)$, and let $\psi$ and $\psi_0$ denote the Laplace exponents of $\rho$ and $\rho_0$, respectively. For any Borel set $A$ and every $i \neq j$, the normalization $\bm{\tilde P} = \bm{\crm}/\bm{\crm}(\X)$ satisfies  $\E(\tilde P_i(A)) = P_0(A)$ and
	\begin{align*}
		\var(\tilde P_i(A)) & = - P_0(A)(1-P_0(A)) \int_0^{+\infty} u\,e^{-\psi_0(\psi(u))} (\psi_0 \circ \psi)'' (u) \, \ddr u, \\
		\cov(\tilde P_i(A), \tilde P_j(A)) & = \var \left(\frac{\crm_0}{\crm_0(\X)} \right) = - P_0(A)(1-P_0(A)) \int_0^{+\infty} u\, e^{-\psi_0(u)} \psi_0'' (u) \, \ddr u,
	\end{align*}
	where $(\psi_0 \circ \psi)''(u) = \psi_0''(\psi(u)) \, \psi'(u)^2 + \psi_0'(\psi(u)) \, \psi''(u)$. Moreover, if $P_0(A) \notin \set{0,1}$, 
	\begin{equation*}
		\corr(\tilde P_i(A), \tilde P_j(A)) = \frac{\int_0^{+\infty} u\,e^{-\psi_0(u)} \psi_0'' (u) \, \ddr u}{\int_0^{+\infty} u\,e^{-\psi_0(\psi(u))} (\psi_0 \circ \psi)'' (u) \, \ddr u}.
	\end{equation*}
\end{proposition}

\begin{example} 
	\label{ex:moments_gamma}
	Let $\bm{\crm}$ be a gamma-gamma hCRV as in Example~\ref{ex:gamma-gamma} with $b = b_0 = 1$, and let $A$ be such that $P_0(A) \notin \set{0,1}$. The marginal and mixed moments of $\bm{\crm}$ and $\bm{\tilde P} = \bm{\crm}/\bm{\crm}(\X)$ are the following. 
	
	\noindent
	\textbf{Moments of the unnormalized measures}
	\begin{gather*}
		\E\big(\tilde\mu_i(A)\big) = \alpha_0 \alpha\, P_0(A),
		\qquad \var\big(\tilde\mu_i(A)\big) = \alpha_0 \alpha(1+\alpha)\,P_0(A), \\
		\cov\big(\tilde\mu_i(A),\tilde\mu_j(A)\big) = \alpha_0 \alpha^2\,P_0(A),\\ 
		\corr\big(\tilde\mu_i(A),\tilde\mu_j(A)\big) = \frac{\alpha}{1+\alpha}.
	\end{gather*}
	
	\noindent
	\textbf{Moments of the normalized random probabilities}
	\begin{gather*}
		\E\big(\tilde P_i(A)\big) = P_0(A), 
		\qquad \var\big(\tilde P_i(A)\big) = \left(1 + \frac{\alpha_0}{\alpha} e^{1/\alpha} E_{\alpha_0}\!\left(\tfrac1\alpha\right)\right) \frac{P_0(A)(1-P_0(A))}{1+\alpha_0}, \\
		\cov\big(\tilde P_i(A),\tilde P_j(A)\big)= \frac{P_0(A)(1-P_0(A))}{1+\alpha_0}, \\
		\corr\big(\tilde P_i(A),\tilde P_j(A)\big) = \left(1 + \frac{\alpha_0}{\alpha} e^{1/\alpha} E_{\alpha_0}\!\left(\tfrac1\alpha\right)\right)^{-1},
	\end{gather*}
	where $E_\eta(x) = \int_1^{+\infty} t^{-\eta} e^{-t x} \ddr t = x^{\eta-1} \Gamma(1-\eta,x)$ is the generalized exponential integral.
\end{example}


\section{Posterior representation for normalized hCRVs}
\label{sec:posterior}

Vectors of dependent random probability measures are commonly employed in Bayesian statistics to model partially exchangeable observations. Indeed, any infinitely active CRV is suitable for this scope through normalization \eqref{eq:norm}. Many models in the literature fall within this framework, including GM-dependent measures \citep{lijoi2014class, lijoi2014bayesian}, compound random measures \citep{GriffinLeisen2017}, L\'evy copulas \citep{EpifaniLijoi2010}, and thinned random measures \citep{LauCripps2022}. In this section, we derive the expression of the posterior distribution for a generic normalized CRV; this can be seen as the multivariate extension of \cite{James2009} and a special case of FuRBI random measures \citep{ascolani2024} with shared atoms. These results are later specialized to hierarchical normalized CRVs, and their posterior structure is explored in greater detail.

Let $\bm{X}_i = (X_{i1}, \dots, X_{in_i})$ be the $i$-th group of observations, for $i=1,\dots,d$, 
and consider the model 
\begin{equation}
	\label{eq:model}
	\bm{X}_1, \dots, \bm{X}_d \mid \bm{\crm} \sim \bigg(\frac{\crm_1}{\crm_1(\X)}\bigg)^{n_1} \times \cdots \times \bigg(\frac{\crm_d}{\crm_d(\X)}\bigg)^{n_d}; \qquad \bm{\crm} \sim \text{CRV}(\nu),
\end{equation}
where $P^m$ denotes the $m$-fold product measure and $\ddr \nu(\bm{s}, x) = \ddr \rho_x(\bm{s}) \, \ddr P_0(x)$ is a multivariate L\'evy intensity with $P_0$ a diffuse probability. 
In the following, we often use the compact 
notation $\bm{X}_{1:d} = (\bm{X}_1,\dots, \bm{X}_d)$.

The almost-sure discreteness and the dependence structure of the random measures $\bm{\crm}$ imply that the observations $\bm{X}_{1:d}$ display tied values with positive probability, both within and across groups.
Let $\bm{x}^* = (x_1^*, \dots, x_k^*)$ be the distinct values taken by the observed values $\bm{x}_{1:d}$, that is, for every $j=1,\dots,k$, there exist $x_{ih}$ such that $x_{ih} = x_j^*$, and $x_j^* \neq x_{\ell}^*$ for every $j \neq \ell$. Moreover, for every $i = 1, \dots, d$ and $j=1,\dots,k$, denote by $n_{ij}$ the number of observed values in $\bm{x}_i$ equal to $x_j^*$. This implies that, for every $i=1,\dots,d$, the decomposition $n_i = n_{i1} + \cdots + n_{ik}$ holds.
Finally, for each $\bm{m}\in \mathbb{N}^d$, $x \in \mathbb{X}$ and $\bm{u} \in \Omega_d$, define the cumulant
\begin{equation} 
	\label{eq:cumulants}
	\tau_{\bm{m} \mid x}(\bm{u}) = \int_{\Omega_d} e^{-\bm{u} \cdot \bm{s}} \prod_{i=1}^d s_i^{m_i} \ddr \rho_x(\bm{s}).
\end{equation}

The next theorem provides a general expression for the posterior of $\bm{\crm} \mid \bm{X}_{1:d}$, showing that it preserves the CRV property, conditionally on a set of dependent latent variables. For this purpose, define a vector of latent variables $\bm{U} = (U_1,\dots, U_d)$ with joint p.d.f. 
\begin{equation}
	\label{eq:latent}
	f_{\scriptscriptstyle \bm{U}}(\bm{u}) \propto e^{- \psi(\bm{u})} \prod_{i=1}^d u_i^{n_i-1} \prod_{j=1}^k \tau_{n_{1j},\dots, n_{dj}\mid x_j^*}(\bm{u}),
\end{equation}
where $\psi$ denotes the Laplace exponent of $\bm{\crm}$.

\begin{theorem}
	\label{th:posterior_crv}
	Let $\bm{X}_1,\dots,\bm{X}_d$ follow the model \eqref{eq:model} with $\ddr \nu(\bm{s}, x) = \ddr \rho_x(\bm{s}) \, \ddr P_0(x)$, where $P_0$ is a diffuse probability and $\rho_x$ is an infinitely active L\'evy measure $P_0$-a.s. Then, there exist $\bm{U}$ with p.d.f. \eqref{eq:latent} such that 
	\begin{equation*}
		\mathcal{L}(\bm{\tilde{\mu}} \mid \bm{X}_{1:d} = \bm{x}_{1:d}) = \mathcal{L}\bigg( \bm{\tilde{\mu}}^*+\sum_{j=1}^{k} \bm{J}_j \delta_{x_j^*} \bigg),
	\end{equation*}
	where $\bm{\tilde{\mu}}^*$ and $(\bm{J}_1,\dots,\bm{J}_k)$ are conditionally independent given $\bm{U}$, and such that
	\begin{enumerate}
		\item[(i)] the vector $\bm{\tilde{\mu}}^* \mid \bm{U}$ is a {\rm CRV} with L\'evy intensity $\ddr \nu^*_{\scriptscriptstyle \bm{U}}(\bm{s}, x)=e^{-\bm{U} \cdot \bm{s}}  \, \ddr \nu (\bm{s}, x)$;
		\item[(ii)] for $j=1,\dots,k$, the vector $\bm{J}_j = (J_{1j},\dots, J_{dj})$ contains the jumps at the shared fixed point of discontinuity $x_j^*$, whose conditional distribution satisfies
		\begin{equation}
			\label{eq:jumps}
			\ddr P_{ \scriptscriptstyle \bm{J}_j \mid \bm{U}}(\bm{s}) \propto e^{-\bm{U} \cdot \bm{s}} \prod_{i=1}^d s_i^{n_{ij}} \,\ddr \rho_{x_j^*}(\bm{s}).
		\end{equation}
	\end{enumerate}
\end{theorem}

\noindent
Note that the L\'evy measure $\rho_x$ does not need a density on $\Omega_d$. This may be seen as a technical detail on $(0,+\infty)$, but it is relevant on $\Omega_d$, as some popular models, such as GM-dependent measures, do not admit a density; see \citet[Lemma 6]{Catalano2024}.

We now specialize Theorem~\ref{th:posterior_crv} to hierarchical CRVs. For simplicity, we assume that $\rho$ and $\rho_0$ have L\'evy densities on $(0,+\infty)$, denoted with the same notation; this implies that both $\id(\rho)$ and $\id(\rho_0)$ have a p.d.f. \cite[Theorem 27.7]{Sato1999}. A first result shows the conditional quasi-conjugacy of the model. Indeed, conditionally on $\bm{U}$, we can interpret the vector $\bm{\crm}^*$ as a hierarchical CRV with heterogeneous marginal distributions.

\begin{proposition}
	\label{th:posterior_hcrv}
	Let $\rho$ and $\rho_0$ have L\'evy densities and let $\bm{\crm} \sim {\rm hCRV}(\rho, \rho_0, P_0)$. Conditionally on $\bm{U}$, the {\rm CRV} $\bm{\crm}^*$ in Theorem~\ref{th:posterior_crv} satisfies
	\begin{align*}
		\bm{\crm}^* \mid \crm_0^*, \bm{U} & \sim \prod_{i=1}^d \CRM \left( e^{-U_i s} \rho(s)\,\ddr s \otimes \crm_0^* \right); \\ 
		\crm_0^* \mid \bm{U} & \sim \CRM \left( e^{-\sum_{i=1}^d \psi(U_i)\,s} \rho_0(s)\,\ddr s \otimes P_0 \right).
	\end{align*}
\end{proposition}

In order to sample from the posterior, one has to sample the jumps $\bm{J}_j$'s at fixed locations and the latent variables $\bm{U}$, which are $d$-dimensional random variables. When $d$ is small, the sampling task can be performed via $d$-dimensional rejection sampling or approximated by a Metropolis-Hastings scheme. However, when $d$ is moderate or large, reducing the dimension of the proposal becomes essential to preserve efficiency. For hierarchical CRVs, the sampling of jumps can be reduced to the sampling of $1$-dimensional random variables.
To this end, define the quantity 
\begin{equation} 
	\label{eq:tau_bar}
	\bar \tau_m(u, t) = \int_0^{+\infty} s^m e^{-u s} f_{\scriptscriptstyle \id (t \rho)}(s) \, \ddr s.
\end{equation}
Remarkably, the quantity above may be computed without integration. Indeed, whenever an explicit expression for the Laplace transform of ${\rm ID}(t\rho)$ is available, one may instead compute its derivatives, since
\begin{equation*}
	\bar \tau_m(u, t) = (-1)^m \frac{\ddr^m}{\ddr u^m} \int_0^{+\infty} e^{-u s} f_{\scriptscriptstyle \id (t \rho)}(s) \, \ddr s. 
\end{equation*}

\begin{proposition} 
	\label{prop:jumps}
	Let $\rho$ and $\rho_0$ have L\'evy densities and let $\bm{\crm} \sim {\rm hCRV}(\rho, \rho_0, P_0)$. Conditionally on the latent variables $\bm{U}$ and for every $j=1,\dots,k$, the jumps ${\bm J}_j = (J_{1j},\dots, J_{dj})$ in Theorem~\ref{th:posterior_crv} satisfy
	\begin{equation*}
		J_{1j}, \dots, J_{dj} \mid \bm{U}, J_{0j} \, \sim \, 
		f_{ \scriptscriptstyle {\bm J}_j \mid \bm{U}, J_{0j}} (\bm{s}) = \prod_{i=1}^d \frac{s_i^{n_{ij}} e^{-U_i s_i} f_{\scriptscriptstyle \id(J_{0j} \rho)}(s_i)}{\bar \tau_{n_{ij}}(U_i, J_{0j})},
	\end{equation*}
	where $J_{0j}$ is a random variable having p.d.f.
	$\displaystyle f_{\scriptscriptstyle J_{0j} \mid \bm{U}}(t) \propto \prod_{i=1}^d \bar \tau_{n_{ij}}(U_i,t) \,\rho_0(t)$.
\end{proposition}

\noindent
Therefore, Proposition~\ref{prop:jumps} reduces the sampling of each $d$-dimensional jumps vector to the easier task of sampling $d$ conditionally independent $1$-dimensional jumps, given an additional variable. These random variables can be sampled exactly in the case of gamma-gamma hCRVs, as detailed in Section~\ref{sec:gamma-gamma}.

As for the latent variables ${\bm U}$, the evaluation of their p.d.f.~\eqref{eq:latent} up to a normalizing constant requires to compute the cumulants $\tau_{\bm{m} \mid x}$ in \eqref{eq:cumulants}.
Although this task can be tackled on a case-by-case basis, the $d$-dimensional integral in their definition can be reduced to a $1$-dimensional integral for hierarchical CRVs. Given their homogeneity, the cumulant $\tau_{\bm{m} \mid x} = \tau_{\bm{m}}$ does not depend on $x \in \X$.

\begin{lemma}
	\label{prop:cumulants}
	Let $\rho$ and $\rho_0$ have L\'evy densities on $(0,+\infty)$ 
	and let $\bm{\crm} \sim {\rm hCRV}(\rho, \rho_0, P_0)$. For $\bm{m} \in \mathbb{N}^d$ and $\bm{u} \in \Omega_d$, the cumulants in \eqref{eq:cumulants} are expressed as 
	\begin{equation*}
		\tau_{\bm{m}}(\bm{u}) = \int_0^{+\infty} \prod_{i=1}^d \bar \tau_{m_i} (u_i, t) \rho_0(t) \, \ddr t.
	\end{equation*}
\end{lemma}

\noindent
Interestingly, the cumulant $\tau_{n_{1j},\dots, n_{dj}}(\bm{U})$ is the normalizing constant for the density of the jump $J_{0j}$ in Proposition~\ref{prop:jumps}. 
These quantities represent the building blocks for any computational method approximating the $d$-dimensional distribution of latent variables $\bm{U}$. 
However, when $d$ is large, reducing to lower-dimensional distributions becomes essential. Therefore, we propose an alternative representation of these latent variables as conditionally independent gamma random variables with random rate parameters; this result may also be of interest in the univariate case.
For $\bm{m} \in \mathbb{N}^k$ and $\bm{t} \in (0,+\infty)^{k+1}$, set $m_\bullet = \sum_{j=1}^k m_j$ and define the integral
\begin{equation} \label{eq:norm_constant}
	C(\bm{m}; \bm{t}) = \int_{(0,+\infty)^{k+1}} (s_0 + s_1 + \cdots + s_k)^{-m_{\bullet}} f_{\scriptscriptstyle \id (t_0 \rho)}(s_{0}) \prod_{j=1}^k s_{j}^{m_{j}}  f_{\scriptscriptstyle \id(t_j \rho)}(s_{j}) \,\ddr \bm{s}.
\end{equation}

\begin{proposition}
	\label{prop:latent}
	Let $\rho$ and $\rho_0$ have L\'evy densities and let $\bm{\crm} \sim {\rm hCRV}(\rho, \rho_0, P_0)$. Then,
	\begin{equation*}
		U_1,\dots, U_d \mid \beta_1, \dots, \beta_d \sim \prod_{i=1}^d {\rm Gamma}(n_{i}, \beta_i),
	\end{equation*}
	where $\beta_i = S_{i0} + S_{i1} + \dots, S_{ik}$, for each $i = 1, \dots, d$. Moreover, the $\bm{S}_i = (S_{i0},\dots, S_{ik})$ are conditionally independent vectors given $\bm T = (T_0,T_1,\dots,T_k)$, and their p.d.f.s satisfy
	\begin{align*}
		f_{ \scriptscriptstyle\bm{S}_i \mid {\bm T}}(s_{i0},\dots,s_{ik})  & \propto (s_{i0} + s_{i1} + \dots + s_{ik})^{-n_{i}} f_{\scriptscriptstyle \id (T_0 \rho)}(s_{i0}) \prod_{j=1}^k s_{ij}^{n_{ij}} f_{\scriptscriptstyle \id(T_j \rho)}(s_{ij}), \\
		f_{ \scriptscriptstyle {\bm T}}(t_0,\dots,t_k) & \propto \prod_{i=1}^d C(n_{i1},\dots,n_{ik}; \bm{t}) f_{\scriptscriptstyle \id(\rho_0)}(t_0) \prod_{j=1}^k \rho_0(t_j).
	\end{align*}
\end{proposition}

\noindent
Proposition~\ref{prop:latent} reduces the sampling of the $d$-dimensional vector of dependent latent variables $\bm{U}$ to the sampling of $d$ conditionally independent $1$-dimensional random variables $\beta_1,\dots,\beta_d$, which represent the scale parameters of $U_1, \dots, U_d$. This sampling task can be approached through standard computational methods or further simplified, as described in the following section.


\section{The normalized gamma-gamma hCRV}
\label{sec:gamma-gamma}

\subsection{Posterior representation}
\label{sec:posterior_sampling_gamma}

This section specifies results in Section~\ref{sec:posterior} to the fundamental example of gamma-gamma hCRVs, and provides details for the subsequent implementation of posterior sampling algorithms. In the following, $((t))_n = \Gamma(t+n)/\Gamma(t)$ denotes the ascending factorial.

\begin{proposition} \label{prop:gamma_jumps}
	Let $\bm{\crm}$ be a gamma-gamma hCRV, as in Example~\ref{ex:gamma-gamma}, and define
	\begin{equation*}
		\lambda(\bm{U}) = \frac{b_0}{\alpha} + \sum_{i=1}^d \log\left( 1 + \frac{U_i}{b}\right).
	\end{equation*}
	Then, conditionally on the latent variables $\bm{U}$, and for each $j = 1, \dots, k$,
	\begin{itemize}
		\item[(a)] the {\rm CRV} $\bm{\crm}^*$ in Proposition~\ref{th:posterior_hcrv} is a hierarchy of conditionally gamma CRMs, 
		\begin{align*}
			\crm_1^*,\dots,\crm_d^* \mid \crm_0^*, \bm{U} & \sim \prod_{i=1}^d \CRM\left( \alpha \, s^{-1}\, e^{-b(1 + U_i/b)\,s} \ddr s \otimes \crm_0^*\right); \\
			\crm_0^* \mid \bm{U} & \sim \CRM \left( \alpha_0 \, s^{-1}\, e^{-\alpha\, \lambda(\bm{U})\,s} \ddr s \otimes  P_0 \right);
		\end{align*}
		\item[(b)] the jumps ${\bm J}_j$ in Proposition~\ref{prop:jumps} are conditionally independent and, for $i = 1, \dots, d$,
		\begin{equation*}
			J_{ij} \mid U_i, J_{0j} \sim {\rm Gamma}\big(\alpha J_{0j} + n_{ij}, b\,(1 + U_i/b)\big); 
		\end{equation*}
		\item[(c)] the density of the rescaled random variable $\alpha J_{0j}$ is proportional to
		\begin{equation} \label{eq:basejump}
			f_{\scriptscriptstyle \alpha J_{0j} \mid \bm{U}}(t)   \propto t^{-1} e^{-\lambda(\bm{U})\,t} \prod_{i=1}^d ((t))_{n_{ij}}.
		\end{equation}
	\end{itemize}
\end{proposition}

\noindent
Details on sampling algorithms for the hierarchy in (a) are given in Sections~\ref{app:sampling_random_measure} and~\ref{app:expint}, where we describe an efficient implementation of Newton's method for inverting the exponential integral, based on a logarithmic transformation that guarantees convergence for each starting point. 
Notably, this same algorithm can be readily employed to sample \emph{a priori} from a gamma-gamma hCRV; see Remark~\ref{rmk:sampling} and Section~\ref{app:sampling_random_measure}.
Moreover, the nontrivial task of sampling the jumps $\bm{J}_j$'s is here reduced to sampling variables $\alpha J_{0j}$'s in (c), whose densities can be computed up to normalizing constants. A convenient option is to resort to computational methods that generate approximate samples, such as Metropolis-Hastings algorithms. Details are provided in Section~\ref{app:mh_sampling}, where we empirically show that the random-walk Metropolis-Hastings scheme on the log-scale with Gaussian increments outperforms the approach proposed in \cite{Barrios2013}, based on gamma proposals.
Remarkably, we also develop an exact sampler in Section~\ref{sec:exact_sampling_gamma}.

The remaining step to sample from the posterior is sampling the latent variables 
$\bm{U}$, whose characterization in Proposition~\ref{prop:latent} can be simplified through a change of variables. 

\begin{proposition} \label{prop:gamma_latent}
	Let $\bm{\crm}$ be a gamma-gamma hierarchical {\rm CRV}. The distribution of latent variables $\bm{U} = (U_1, \dots, U_d)$ can be decomposed as follows:
	\begin{itemize}
		\item[(a)] for each $i = 1, \dots, d$, one has
		$U_i \mid \beta_i \sim {\rm Gamma}(n_i, \beta_i)$, with $\beta_i \mid T \sim {\rm Gamma}(\alpha T, b)$;
		\item[(b)] the density of $\alpha T$ is proportional to
		\begin{equation} 
			\label{eq:joint_baselatent_simplex}
			f_{\scriptscriptstyle \alpha T}(t) \propto t^{\alpha_0-1} e^{-(b_0/\alpha)\,t} \prod_{i=1}^d \frac{1}{((t))_{n_{i}}} \int_{\Delta^k} v_0^{\alpha_0-1} \prod_{j=1}^k \left( v_j^{-1}\prod_{i=1}^d ((t v_j))_{n_{ij}} \right) \ddr \bm{v},
		\end{equation}
		where $\bm{v} = (v_0,\dots,v_k)$ is an auxiliary vector on $k$-dimensional unit simplex $\Delta^k$.
	\end{itemize}
\end{proposition}

\noindent
For implementation convenience, one can directly sample $U_i/b$. Indeed, these are the relevant quantities in Proposition~\ref{prop:gamma_jumps} for sampling the jumps $J_{i1}, \dots, J_{ik}$ and the random measure $\crm^*_i$, and computing $\lambda(\bm{U})$. The non-standard step for sampling $\bm{U}$ is the marginal sampling of the variable $\alpha T$ in \eqref{eq:joint_baselatent_simplex}, whose joint density with the auxiliary vector $\bm{V}$ is known up to a normalizing constant. In Section~\ref{app:mh_sampling} we describe a Metropolis-Hastings within Gibbs algorithm to sample from the marginal distribution of $\alpha T$.
Alternative procedures to obtain exact samples from $\alpha T$ are considered in Section~\ref{sec:exact_sampling_gamma}. 

\subsection{Exact sampling}
\label{sec:exact_sampling_gamma}

As highlighted by Propositions~\ref{prop:gamma_jumps}-\ref{prop:gamma_latent}, posterior sampling from the normalized gamma-gamma hCRV mainly consists in sampling gamma random variables, except for two critical steps, namely the sampling of random variables $\alpha J_{01}, \dots, \alpha J_{0k}$ in \eqref{eq:basejump} and the marginal sampling of random variable $\alpha T$ in \eqref{eq:joint_baselatent_simplex}. In both cases, one may resort to MCMC algorithms, based on Metropolis-Hastings steps, to obtain approximate samples from such distributions. In the following, we propose alternative algorithms based on analytical manipulations of their density functions, which instead allow for exact sampling.
For this purpose, 
introduce the coefficients $S(q_1,\dots,q_d; h)$, defined by the recursive relation 
\begin{equation} 
	\label{eq:multivariate_stirling}
	S(q_1,\dots,q_\ell+1,\dots,q_d; h) = q_\ell S(q_1,\dots,q_d; h) + S(q_1,\dots,q_d; h-1),
\end{equation}
with boundary conditions $S(0,\dots,0; 0) = 1$ and $S(q_1,\dots,q_d; 0) = S(0,\dots,0; h) = 0$ for $q_{\bullet}> 0$ or $h > 0$. Remarkably, these coefficients are the natural generalization of the unsigned Stirling numbers of the first kind to blocked permutations. Indeed, $S(q_1,\dots,q_d; h)$
represents the number of permutations with $h$ cycles in the Young subgroup of $\mathcal{S}_{q_\bullet}$ (group of permutations of $q_\bullet$ elements) induced by the integer partition $(q_1,\dots, q_d)$. Henceforth, we refer to the coefficients $S(q_1,\dots,q_d; h)$ as multivariate Stirling numbers.
For convenience, let $m_{ij} = \min(1,n_{ij}) \in \{0,1\}$, and define $m_{\bullet j} = \sum_{i=1}^d m_{ij}$. 

\begin{proposition} 
	\label{prop:gamma_basejumps_exact}
	For each $j = 1, \dots, k$, the density of $\alpha J_{0j}$ in \eqref{eq:basejump} can be written as
	\begin{equation*}
		f_{\scriptscriptstyle \alpha J_{0j} \mid \bm{U}}(t) \propto \sum_{h_j=m_{\bullet j}}^{n_{\bullet j}} S(n_{1j},\dots,n_{dj}; h_j) \ t^{h_j-1} e^{-\lambda(\bm{U})\,t}.
	\end{equation*}
\end{proposition}

\noindent
Therefore, the random variable $\alpha J_{0j}$ is in fact a finite mixture of gamma distributions, with mixing weights $p_{jh_j} \propto \lambda(\bm{U})^{-h_j} \, \Gamma(h_j) \, S(n_{1j},\dots,n_{dj}; h_j)$, for $h_j = m_{\bullet j}, \dots, n_{\bullet j}$. The evaluation of the $S(n_{1j},\dots,n_{dj}; h_j)$'s via the recursive relation in \eqref{eq:multivariate_stirling} has computational cost $\mathcal{O}\big(n_{\bullet j}^2\big)$; see Remark~\ref{remark:computation_stirling}.

\begin{proposition} 
	\label{prop:gamma_baselatent_exact}
	The density function of 
	random variable $\alpha T$ in \eqref{eq:joint_baselatent_simplex} can be written as 
	\begin{equation} \label{eq:density_baselatent}
		f_{\scriptscriptstyle \alpha T}(t) \propto t^{\alpha_0-1} \, e^{-(b_0/\alpha)\,t} \prod_{i=1}^d \frac{1}{((t))_{n_{i}}} \left( \sum_{h=m}^n \frac{c_h}{((\alpha_0))_h} \, t^h \right),
	\end{equation}
	where $m = \sum_{j=1}^k m_{\bullet j}$ and each coefficient $c_h$ is defined, for $S(q_1,\dots,q_d; h)$ in \eqref{eq:multivariate_stirling}, by
	\begin{equation*}
		c_h = \sum_{\substack{h_1 + \cdots + h_k = h \\ m_{\bullet j} \le h_j \le n_{\bullet j}}} \prod_{j=1}^k \Gamma(h_j) \, S(n_{1j},\dots,n_{dj}; h_j).
	\end{equation*}
\end{proposition}

\noindent
The evaluation of the $c_h$'s may seem computationally expensive when $k$ is large. Indeed, a naive approach would involve $k$ nested cycles, with a computational cost of $\mathcal{O}\big( \prod_{j=1}^k n_{\bullet j}\big)$. However, these coefficients can be obtained through a sequence of discrete convolutions and computed at cost $\mathcal{O}\big(\sum_{j < \ell} n_{\bullet j} n_{\bullet \ell} \big)$, as detailed in Section~\ref{app:coefficients_exact}. Exact samples from the distribution of $\alpha T$ can be obtained resorting to rejection sampling algorithms. Indeed, considering a real parameter $r$, \eqref{eq:density_baselatent} can be rewritten as
\begin{equation*}
	f_{\scriptscriptstyle \alpha T}(t) \propto t^{\alpha_0+r-1} \, e^{-(b_0/\alpha)\,t} \, (t^{-r} R(t)), \qquad R(t) = \prod_{i=1}^d \frac{1}{((t))_{n_{i}}} \left( \sum_{h=m}^n \frac{c_h}{((\alpha_0))_h} \, t^h \right),
\end{equation*}
where $R(t)$ is a ratio of polynomials, and continuous for $t \ge 0$.
The distribution of $\alpha T$ is absolutely continuous with respect to a ${\rm Gamma}(\alpha_0 + r, b_0 / \alpha)$, with Radon-Nikodym derivative proportional to $t^{-r} R(t)$, from which we construct a rejection sampling algorithm. A necessary condition for the Radon-Nikodym derivative is to be bounded above: this is guaranteed when $0 \le r \le m-d$. Within this interval, we choose the value of $r$ that maximizes the acceptance probability, which is equivalent to maximizing 
\begin{equation*}
	r \log t^*(r) - \log R(t^*(r)) + (\alpha_0 + r) \log(b_0/\alpha) - \log \Gamma(\alpha_0 + r),
\end{equation*}
where $t^*(r)$ is the value of $t$ that maximizes $t^{-r} R(t)$. Details on the optimal choice of $r$ can be found in Section~\ref{app:ropt}. Alternatively, the adaptive rejection sampling algorithm of \cite{gilks1992} may considerably improve the acceptance rate. Although computationally convenient, theoretical guarantees are only available if the density of $\alpha T$ is logarithmically concave, which we were unable to prove in general. Therefore, for numerical experiments, we consider the standard rejection sampler outlined above.

\subsection{Posterior sampling algorithms}
\label{sec:algorithms_gamma}

\begin{figure}[t]
	\centering
	\includegraphics[width=.8\linewidth]{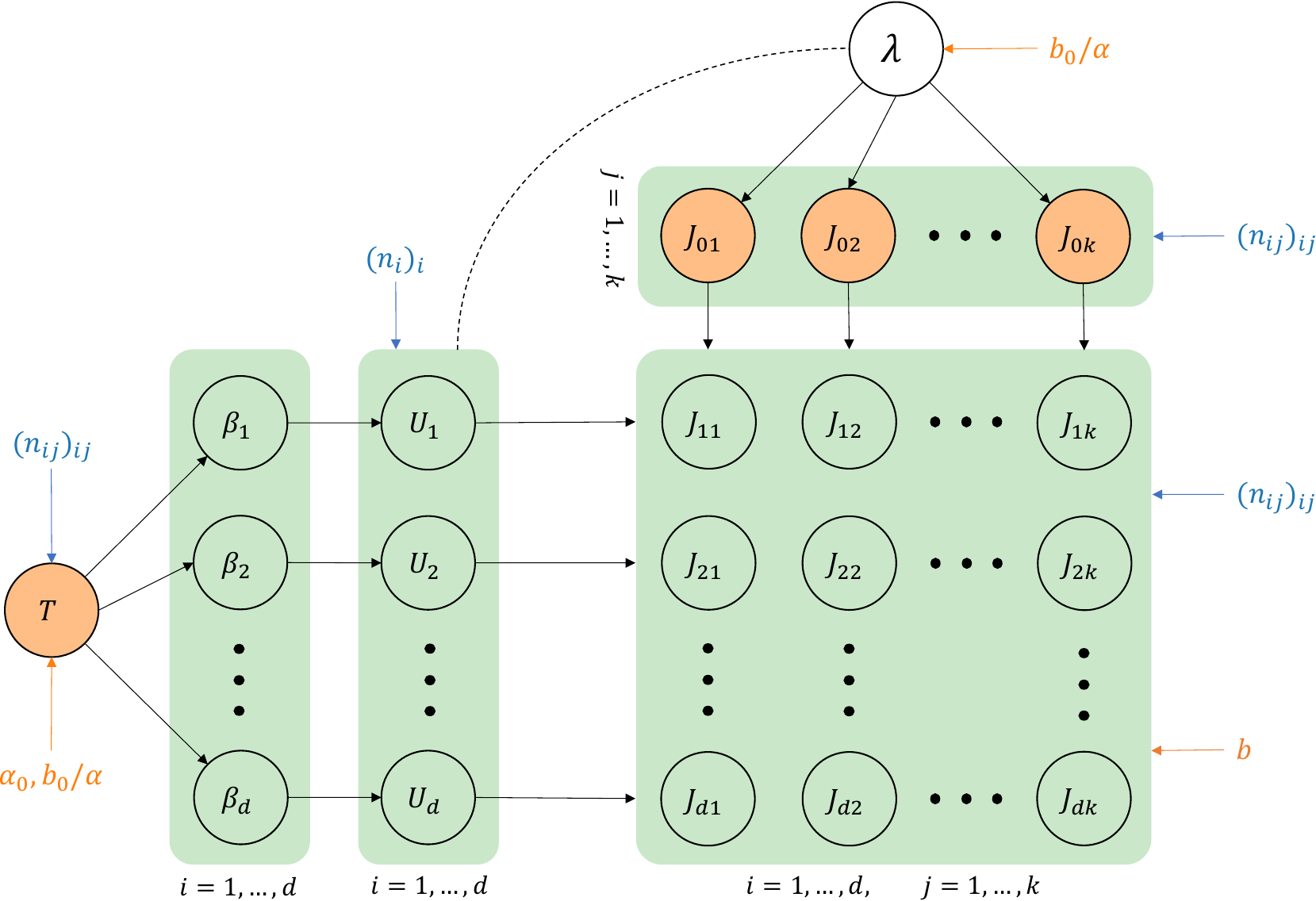}
	\captionsetup{width=0.9\textwidth,font=small}
	\caption{DAG of conditional dependencies between random variables in Algorithms~\ref{alg:mcmc} and~\ref{alg:exact}.
		Red circles represent
		computational bottlenecks; quantities in empty circles are sampled from gamma distributions. Some variables are reported up to scaling w.r.t.~model parameters, e.g. $T$ instead of $\alpha T$.}
	\label{fig:structure}
\end{figure}

Collecting the distributional results described above, as well as the details in
Section~\ref{app:sampling}, we obtain algorithms to sample jumps $\bm{J}$ at fixed locations in the posterior distribution (Proposition~\ref{th:posterior_hcrv}) induced by the gamma-gamma hCRV. 
The Julia implementation of these algorithms is available at
\href{https://github.com/claudiodelsole/hCRV.jl}{github.com/claudiodelsole/hCRV.jl}, together with an interface to allow their integration within the R environment. The common structure of conditional dependencies within the sampling procedures is summarized in Figure~\ref{fig:structure}, where the computational bottlenecks are highlighted by red circles.
Remarkably, the graph of dependencies has no proper cycles, i.e.~it is a directed acyclic graph (DAG): variables are sampled in topological order, with each variable depending only on (some~of) its ancestors. Such dependence structure may enhance mixing.
Furthermore, the proposed algorithms allow for parallelization at different stages. Indeed, after the initial sampling of $\alpha T$ from its marginal, one can parallelize across groups to sample the latent variables $\bm U$ (Proposition~\ref{prop:gamma_latent}). 
Likewise, after the computation of $\lambda(\bm{U})$, one can parallelize across distinct values to sample the conditionally independent variables $\alpha J_{01}, \dots, \alpha J_{0k}$ and further parallelize across groups and distinct values to sample jumps $\bm J$ (Proposition~\ref{prop:gamma_jumps}).

The MCMC scheme for posterior sampling the jumps $\bm{J}$ is summarized in Algorithm~\ref{alg:mcmc}. The state of the Markov chain consists of the latent variable $\alpha T$ and jumps $\alpha J_{01}, \dots, \alpha J_{0k}$, plus auxiliary variables $\bm{V}$, thus having dimension $2k+1$. The exact posterior sampling procedure developed in Section~\ref{sec:exact_sampling_gamma} is instead summarized in Algorithm~\ref{alg:exact}. This scheme outputs i.i.d.~samples from the posterior, but requires a non-trivial initialization. The main computational bottleneck is represented by the rejection sampling step for the latent variable $\alpha T$, which may suffer from low acceptance rates. Numerical illustrations proving the effectiveness of the proposed algorithms are discussed in Section~\ref{app:illustrations}. 

\begin{algorithm}[t]
	\caption{MCMC posterior sampling algorithm for jumps $\bm J$.} 
	\emph{Current state}: $\alpha T, (\alpha J_{01}, \dots, \alpha J_{0k}), (\bm{V})$\;
	$(\alpha T) \leftarrow$ Metropolis-Hastings step from its marginal (via auxiliary variables $\bm{V}$) \;
	\For{$i=1,\dots,d$}{
		$(b \beta_i) \mid (\alpha T) \sim \text{Gamma}(\alpha T,1)$; $(U_i/b) \mid (b\beta_i) \sim \text{Gamma}(n_{i}, b \beta_i)$\;
	}
	compute $\lambda(\bm{U}) = b_0/\alpha + \sum_{i=1}^d \log(1 + U_i/b)$\;
	\For{$j=1,\dots,k$}{
		$(\alpha J_{0j}) \leftarrow$ Metropolis-Hastings step from the distribution of $(\alpha J_{0j}) \mid \lambda(\bm{U})$ \;
	}
	\For{$i=1,\dots,d$}{
		\For{$j=1,\dots,k$}{
			$J_{ij} \mid (\alpha J_{0j}), (U_i/b) \sim \text{Gamma}(n_{ij} + \alpha J_{0j}, b\,(1+U_i/b))$\;
		}
	}
	\label{alg:mcmc}
\end{algorithm}

As highlighted in Proposition~\ref{th:hdp_hyper}, the normalized gamma-gamma hCRV model coincides with the hierarchical Dirichlet process (HDP) with a particular gamma prior on the concentration parameter. Therefore, the algorithms described above can be compared with standard samplers for posterior analysis of the HDP.
Among the various alternatives and their countless variations, we focus on marginal algorithms, and consider 
(i) the Gibbs sampler based on the restaurant franchise metaphor \citep{Teh2006}, in which the allocation to dishes is observed (and thus fixed), and (ii) a collapsed Gibbs sampler for the number of tables serving each dish, derived from the pEPPF; cfr. Section~\ref{app:tables} for details. The latter algorithm does not appear in the existing literature, and may be of independent interest. In fact, it is related to a sampler considered in \cite{baccallado2022} for the hierarchical Pitman-Yor process \citep[see also][]{Camerlenghi2019}, but relies on a further marginalization of the latent tables through the multivariate Stirling numbers in \eqref{eq:multivariate_stirling}, reducing from $kd$ to $k$ latent variables. The marginal sampler (i) involves $n$ latent variables, one for each observation's allocation to a table; hence, its computational complexity rapidly increases with the sample size. In this respect, the collapsed Gibbs sampler (ii) provides a substantial improvement, since its state space has dimension $k$, which is typically much smaller than $n$. 
However, updating each of the $k$ variables (namely the number of tables serving each dish) requires the evaluation of the multivariate Stirling numbers in \eqref{eq:multivariate_stirling}. These grow extremely fast and often cause numerical overflows even for moderately large sample sizes. For this reason, when $n$ is (moderately) large, the marginal sampler (i) remains the practical choice.

In fact, our proposals present some interesting advantages. Firstly, in presence of a (moderately) large number of observations, the MCMC approach in Algorithm~\ref{alg:mcmc} relies on a state of dimension $2k+1$, thereby reducing the computational burden of the marginal sampler (i), without requiring the evaluation of multivariate Stirling numbers. Moreover, the structure of conditional (in)dependencies enables parallelization and potentially speeds up mixing.
Finally, running a Markov chain on the Euclidean space $(0,\infty)^{k+1} \times \Delta^k$ is more convenient than sampling on the constrained space of partitions: many standard tools for MCMC diagnostics are designed for Euclidean spaces, and acceptance rates of Metropolis-Hastings steps can be optimized by tuning the variance of the proposal; see Section~\ref{app:mh_sampling}. On the other hand, when $n$ is small or moderately large, so that the computation of multivariate Stirling numbers in the collapsed sampler (ii) remains feasible, Algorithm~\ref{alg:exact} allows for exact i.i.d.~sampling from the posterior. This approach avoids any potential issue inherent to MCMC schemes, without increasing the computational complexity. 

\begin{algorithm}[t]
	\caption{Exact posterior sampling algorithm for jumps $\bm J$.} 
	\emph{Initialize}: compute coefficients in Propositions~\ref{prop:gamma_basejumps_exact}-\ref{prop:gamma_baselatent_exact}, the optimal $r$ and the upper bound $t^*(r)$. \\
	\While{proposal is not accepted}{
		propose $\alpha T \sim \text{Gamma}(\alpha_0 + r, b_0/\alpha)$\;
		accept proposal with probability proportional to $(\alpha T)^{-r} R(\alpha T)$ (Section~\ref{sec:exact_sampling_gamma})\;
	}
	\For{$i=1,\dots,d$}{
		$(b \beta_i) \mid (\alpha T) \sim \text{Gamma}(\alpha T,1)$; $(U_i/b) \mid (b\beta_i) \sim \text{Gamma}(n_{i}, b \beta_i)$\;
	}
	compute $\lambda(\bm{U}) = b_0/\alpha + \sum_{i=1}^d \log(1 + U_i/b)$\;
	\For{$j=1,\dots,k$}{
		$H_j \mid \lambda \sim \text{Categorical}(p_{j m_{\bullet j}}, \dots, p_{j n_{\bullet j}})$ from rescaled weights (Section~\ref{sec:exact_sampling_gamma})\;
		$(\alpha J_{0j}) \mid H_j, \lambda \sim \text{Gamma}(H_j, \lambda)$\;
	}
	\For{$i=1,\dots,d$}{
		\For{$j=1,\dots,k$}{
			$J_{ij} \mid (\alpha J_{0j}), (U_i/b) \sim \text{Gamma}(n_{ij} + \alpha J_{0j}, b\,(1+U_i/b))$\;
		}
	}
	\label{alg:exact}
\end{algorithm}

\subsection{Simulation studies}
\label{sec:simulations}

In the following, we carry out several simulation studies to compare the algorithms in terms of their execution time, as the dimensions of the input data grow. Specifically, we compare four algorithms: 
(i) the MCMC sampler in Algorithm~\ref{alg:mcmc} with symmetric random-walk Metropolis-Hastings steps on the log-scale (mcmc); (ii) the exact sampler in Algorithm~\ref{alg:exact}; (iii) the marginal Gibbs sampler of \cite{Teh2006} based on the restaurant franchise metaphor (CRF); (iv) the collapsed Gibbs sampler for the number of tables serving each dish, detailed in Section~\ref{app:tables}. 
Algorithms are compared in terms of execution time per effective sample, averaging over $100$ simulated datasets for each experimental setting. The burn-in time for algorithms (i), (iii) and (iv) and the initialization time for algorithm (ii) are deducted from the total execution times. The exact and collapsed Gibbs samplers are stopped when the computations of the multivariate Stirling numbers in \eqref{eq:multivariate_stirling} encounter a numerical overflow. The corresponding execution times are plotted when at least 75 out of 100 experiments are completed without errors.

\begin{figure}[t]
	\centering
	\includegraphics[width=\textwidth]{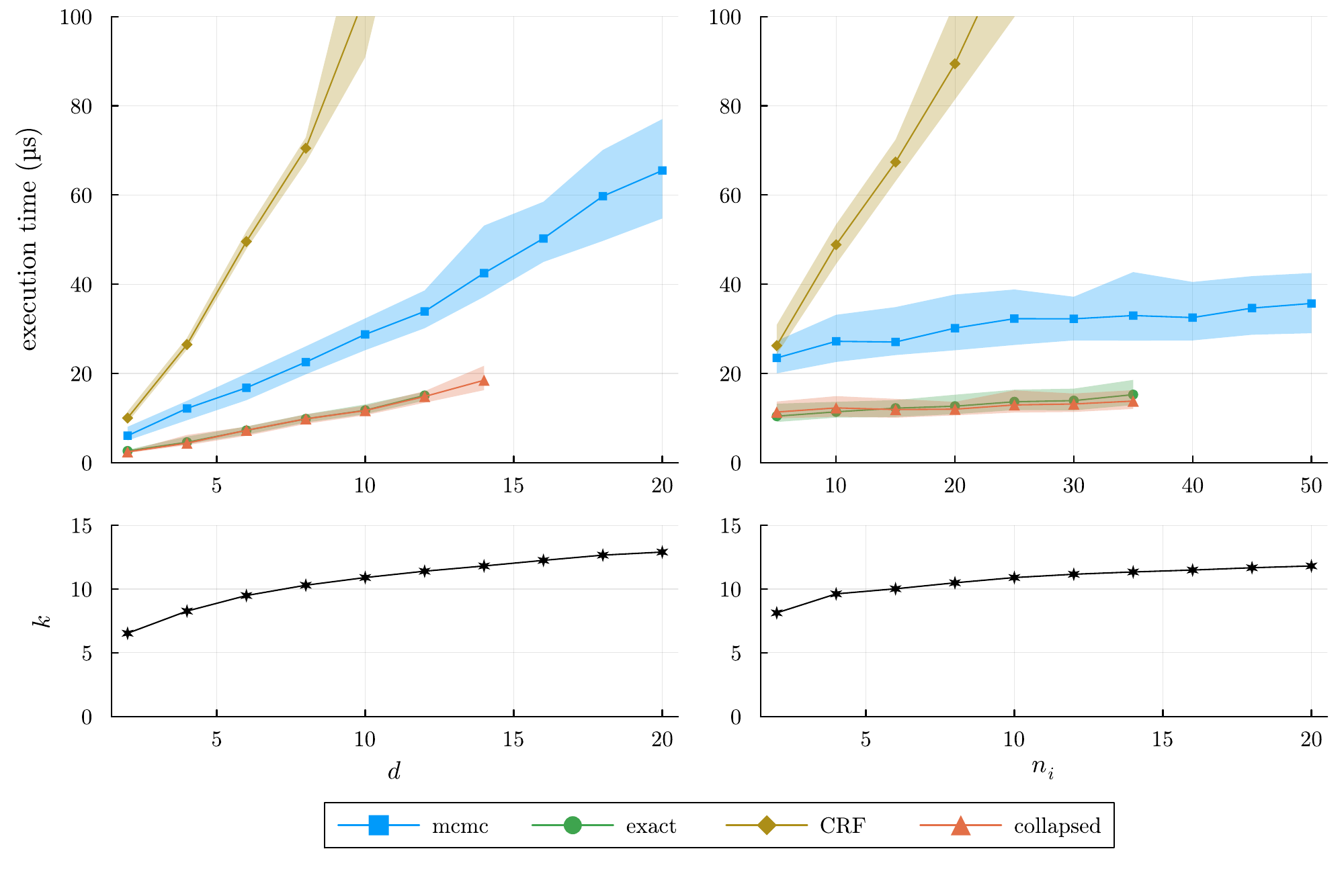}
	\captionsetup{width=.9\textwidth,font=small}
	\caption{Execution times per effective sample for different algorithms with increasing number of groups $d$ (left) and increasing number of observations per group $n_{i}$ (right). Results are averaged over $100$ simulated datasets per experimental setting. Solid curves represent median values, with shaded areas between the first and third quartiles. Plots at the bottom display the mean number of distinct values $k$ for each setting.}
	\label{fig:times_dn}
\end{figure}

The left panel of Figure~\ref{fig:times_dn} compares the algorithms for increasing number of groups $d$, while the number of observations per group $n_{i} = 50$ is fixed. Similarly, the right panel compares the algorithms for increasing number of observations per group $n_{i}$, while the number of groups $d = 20$ is fixed. 
In both cases, observations are sampled from a hierarchical Dirichlet process with concentration parameters $\alpha = 5$ and $\alpha_0 = 3$. Note that parameter $\alpha_0$ impacts the number of columns in the counts matrix $(n_{ij})_{ij}$, while parameter $\alpha$ controls its sparsity.
For the CRF-based Gibbs sampler, the time per effective sample grows more than linearly in the total number of observations, as expected given its sequential allocation structure. The other algorithms show a linear growth rate in the number of groups (Figure~\ref{fig:times_dn}, left panel), with the exact and collapsed samplers displaying a smaller slope compared with the MCMC approach.
These algorithms appear less affected by the number of observations per group, showing a nearly constant behavior as $n_i$ grows (Figure~\ref{fig:times_dn}, right panel).
We argue that their execution times may instead mostly depend on the number of distinct values $k$, which determines the dimension of the state space for both the MCMC and collapsed samplers. In fact, $k$ is slowly increasing in the experimental settings considered above, as displayed in the bottom panels of Figure~\ref{fig:times_dn}. Refer to Section~\ref{app:timesk} for a further simulation study supporting this claim.

In conclusion, the time complexity of both proposed sampling schemes scales linearly in the number of groups $d$ and distinct values $k$, while it is essentially unaffected by the number of observations. In our implementation, the exact sampler shows better performances, yielding around twice as many independent posterior samples as the MCMC approach, in the same computation time. This is comparable with the collapsed Gibbs sampler for the HDP, which however does not output i.i.d.~samples. Despite their efficiency, both algorithms suffer numerical issues as the sample size grows: in our experimental setting, their applicability is limited to $500$-$700$ observations. On the contrary, the MCMC algorithm runs efficiently regardless of the sample size. In particular, it proves much faster than the CRF-based Gibbs sampler for moderate to large sample sizes. 

\subsection{Extensions to other hCRVs}
\label{sec:extensions}

The normalized gamma-gamma hCRV represents both the most relevant and the analytically simplest hCRV construction, since  Proposition~\ref{th:hdp_hyper} establishes its tight connection with the HDP, arguably the cornerstone of Bayesian nonparametric dependent models, and several latent variables involved in its posterior characterization collapse to gamma distributions. Nevertheless, some analytical tractability is preserved beyond the gamma-gamma case, with potential advantages in terms of model flexibility, e.g.~in the asymptotic behavior of the number of distinct values. 
We argue that the tractability of the gamma-gamma hCRV is primarily due to the idiosyncratic component $\rho$ being the L\'evy measure of a gamma CRM. 
Indeed, the results in Propositions~\ref{prop:jumps} and~\ref{prop:latent} heavily depend on the distributional form of ${\rm ID}(t \rho)$, while the choice of $\rho_0$ only affects the distribution of the $J_{0j}$'s and of the vector $\bm{T}$. 
Therefore, one may consider different choices of the base L\'evy measure $\rho_0$ to enhance model flexibility, at a reduced cost in terms of tractability; for example, $\rho_0$ can be taken to be the L\'evy measure of a generalized gamma or stable CRM.
The main limitation is the need for pointwise evaluations of the density of ${\rm ID}(\rho_0)$, whose analytical form is rarely available; however, this task can be performed through well-established numerical methods. 
Instead, the extension of the idiosyncratic measure $\rho$ to other choices is less straightforward; in particular, within the present posterior representation, the availability of a closed-form expression for the density of ${\rm ID}(t \rho)$ for $t > 0$ is essential to sample the jumps $J_{ij}$'s in Proposition~\ref{prop:jumps}, as well as the vectors $\bm{S}_i$'s in Proposition~\ref{prop:latent}. As a result, reasonable fully tractable choices for $\rho$ are limited to few special cases, namely the inverse Gaussian process and $\sigma$-stable process with $\sigma = 0.5$. Nonetheless, the possibility to directly sample the $\beta_i$'s in Proposition~\ref{prop:latent} is limited to the gamma case. The development of alternative techniques to handle ${\rm ID}(t \rho)$
beyond these specifications would considerably enlarge the class of tractable hCRVs, and
represents an interesting direction for future work.
In conclusion, we point out that there are no particular reasons to constrain $\rho$ and $\rho_0$ to take the same analytical form, neither from the modelling nor the computational perspectives. Instead, their choice should be informed by the desired flexibility in the marginal and dependence structures; see Proposition~\ref{th:moments_norm} and the discussion in \citet[][Section 8]{Sankhya2024}.


\section{Application: a fair comparison with the HDP}
\label{sec:comparison}

\begin{figure}[t]
	\centering
	\includegraphics[width=\linewidth]{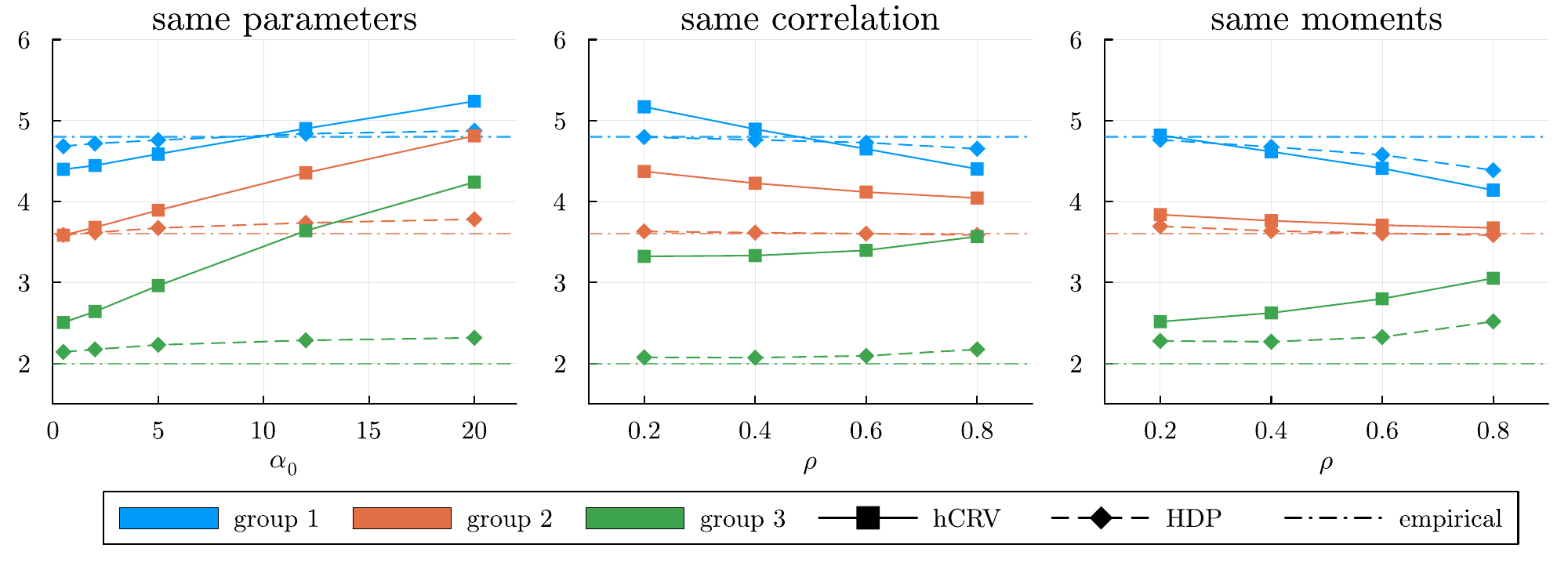}
	\captionsetup{width=0.9\textwidth,font=small}
	\caption{Comparison between predictive means for the normalized gamma-gamma hCRV and the HDP fixing the same hyperparameters $(\alpha,\alpha_0)$, with $\alpha = 1$ (left); the same correlation and different variances, with $\sigma^2$ equal to $0.2$ and $0.8$ respectively (middle); the same variance and correlation, with $\sigma^2 = 0.5$ (right). Data are 3 groups of independent Poisson observations of size $n_{i} = 10$. The prior mean is $P_0 = N(8,1)$.}
	\label{fig:comparison}
\end{figure}

Section~\ref{sec:elicitation} investigates the impact of the hyperparameters on the borrowing of information and shrinkage induced by the normalized gamma-gamma hCRV, and discusses the proper elicitation of its prior dependence structure. In this section, we compare such model with the hierarchical Dirichlet process.
We consider a simulated dataset with $d = 3$ groups of independent Poisson observations, each of size $n_{i} = 10$, with means equal to $2$, $3$ and $4$. 
Proposition~\ref{th:hdp_hyper} shows that the normalized gamma-gamma hCRV is equivalent to the HDP with a particular gamma prior on the concentration parameter $\alpha$. Hence, it may seem natural to compare the normalized gamma-gamma hCRV with the HDP by fixing the same parameters $(\alpha, \alpha_0, P_0)$, as in Figure~\ref{fig:comparison} (left plot), where $P_0 = N(8,1)$.
However, this approach may alter the comparison, since fixing the same parameters can entail very different marginal variance and correlation structures. In particular, for the HDP$(\alpha,\alpha_0,P_0)$, following \cite{Camerlenghi2019},
\begin{gather*}
	E(\tilde P_i(A)) = P_0(A), \qquad \var(\tilde P_i(A)) = \frac{1 + \alpha + \alpha_0}{(1+\alpha_0)(1 + \alpha)} \, P_0(A)(1-P_0(A)), \\
	\corr(\tilde P_i(A), \tilde P_j(A)) = \frac{1 + \alpha}{1 + \alpha + \alpha_0}.
\end{gather*}
Table~\ref{table:limits} reports some limiting behaviors, which notably differ from those of the normalized gamma-gamma hCRV, derived from Example~\ref{ex:moments_gamma}. Therefore, rather than considering the same values of $(\alpha, \alpha_0)$, a fair comparison should consider the same values of the mean, variance and correlation. Specifically, one should choose $\sigma^2, \rho \in (0,1)$ such that $\var(\tilde P_i(A)) = \sigma^2 \, P_0(A)(1-P_0(A))$ and $ \corr(\tilde P_i(A), \tilde P_j(A)) = \rho$, and set the hyperparameters of the models accordingly. For the normalized gamma-gamma hCRV, the corresponding values of $\alpha$ and $\alpha_0$ are obtained by solving the system of non-linear equations
\begin{equation*}
	\rho\, \big(1+\alpha_0/\alpha \, e^{1/\alpha} E_{\alpha_0}(1/\alpha)\big) - 1 = 0, \qquad \sigma^2 (1 + \alpha_0) - 1/\rho = 0,
\end{equation*}
using standard numerical methods. The same task is achieved for the HDP by setting
\begin{equation*}
	\alpha = \frac{1}{1-\rho} \bigg( \frac{1}{\sigma^2}-1 \bigg), \qquad \alpha_0 = \frac{1}{\rho \sigma^2} -1.
\end{equation*}
Figure~\ref{fig:elicitation} displays the values of $(\alpha, \alpha_0)$ for the two models as functions of $(\sigma^2, \rho)$.
The right plot of Figure~\ref{fig:comparison} compares the normalized gamma-gamma hCRV with the HDP fixing the same mean, variance, and correlation. In contrast, the middle plot illustrates a common approach in which the correlation is fixed without accounting for variance. In conclusion, when fixing the same parameters $(\alpha, \alpha_0)$, one might mistakenly conclude that the gamma-gamma hCRV borrows much more information than the HDP. A similar misinterpretation arises when fixing the correlation without adjusting for variance. Ultimately, by matching variance and correlation, it becomes clear that the gamma-gamma hCRV exhibits only slightly more borrowing of information. This increase can be attributed to the hierarchical CRV leveraging not only the common base random probability but also information contained in the random concentration parameter.

\begin{figure}[t]
	\centering
	\includegraphics[width=\linewidth]{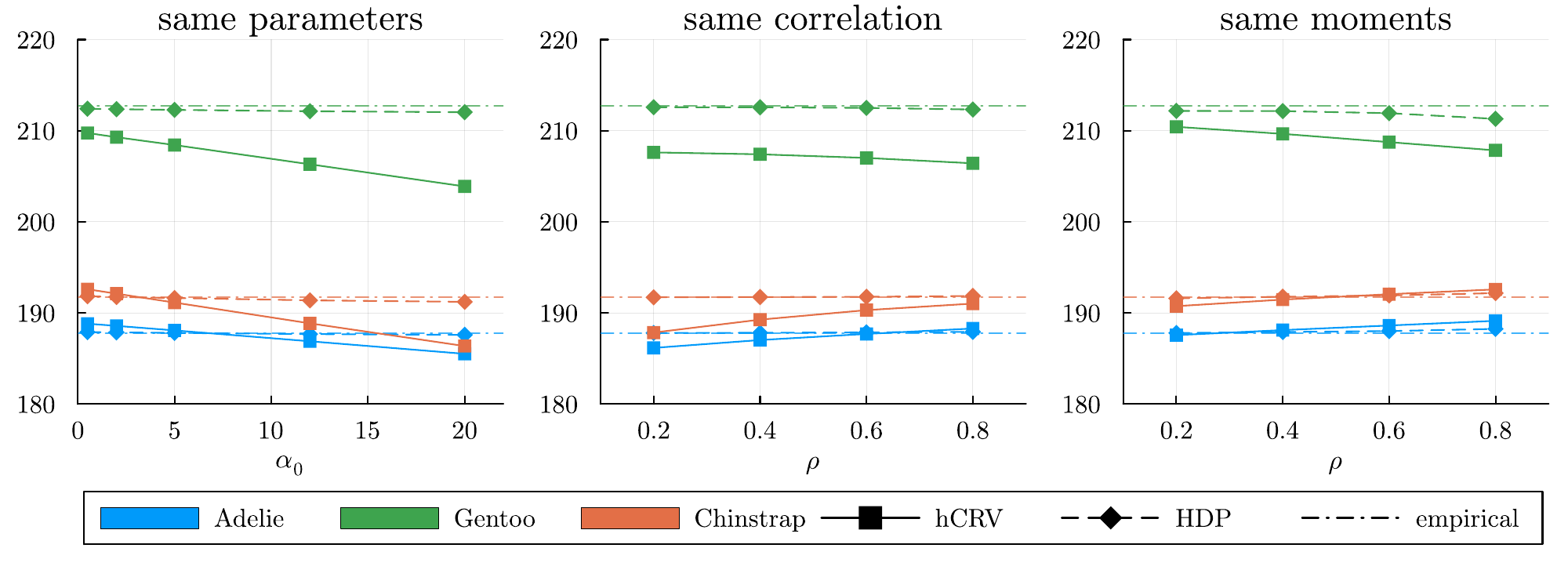}
	\captionsetup{width=0.9\textwidth,font=small}
	\caption{Comparison between predictive means for the normalized gamma-gamma hCRV and the HDP  fixing the same hyperparameters $(\alpha,\alpha_0)$, with $\alpha = 1$ (left); the same correlation and different variances, with $\sigma^2$ equal to $0.2$ and $0.8$ respectively (middle); the same variance and correlation, with $\sigma^2 = 0.5$ (right). Data are flipper lengths for 3 species of female penguins of the Palmer Archipelago \citep{palmerpenguins}. 
	}
	\label{fig:penguins}
\end{figure}

A similar analysis is reproduced on the Palmer Archipelago (Antarctica) penguin data \citep{palmerpenguins}, with the goal of predicting the flipper length for female penguins. The dataset includes three species of penguins (Adelie, Chinstrap, Gentoo) with many ties ($n = 165$, $k = 41$). The prior guess is $P_0 = N(100, 10)$, which provides a significant bias towards lower values of the flipper length. As shown in Figure~\ref{fig:penguins}, fixing the same value of the hyperparameters $(\alpha, \alpha_0)$ could incorrectly suggest that the gamma-gamma hCRV gives too much weight to the prior with respect to the HDP (left plot). This effect is mitigated by fixing the same correlation but different variance (middle plot) and disappears by fixing the same variance and correlation (right plot). Under this fair model comparison, the gamma-gamma hCRV borrows slightly more information than the HDP. This behavior is the same observed for the simulated data, and can be ascribed to the information contained in the common concentration parameter.


\section{Discussion}

This work introduces normalized hierarchical completely random vectors (hCRVs), a new way of building dependent priors that combines the naturalness of hierarchical structures with a convenient posterior representation due to multivariate infinite divisibility. On the one hand, they provide an alternative sampling strategy for existing models, such as the hierarchical Dirichlet process \citep{Teh2006}; on the other hand, they represent a general recipe for constructing dependent priors with a convenient posterior representation. As such, normalized hCRVs enter the flourishing context of partially exchangeable models, which includes, among others, dependent stick-breaking constructions \citep{MacEachern1999, MacEachern2000, DunsonPark2008, Horiguchi2025}, additive random measures \citep{Mueller2004, lijoi2014bayesian}, nested structures \citep{Rodriguez2008, CamerlenghiDunson2019, Beraha2021, Lijoi2023}, L\'evy copulas \citep{EpifaniLijoi2010, RivaPalacio2018}, compound random measures \citep{GriffinLeisen2017}, dependent Pólya trees \citep{ChristensenMa2020},
FuRBI priors \citep{ascolani2024}; see \cite{Quintana2022} for a recent review. 

Future work will explore additional theoretical properties and methodological advantages of normalized hCRVs in different settings and for different families of CRMs. A first natural development is the derivation of the induced random partition through the pEPPF and of the predictive distribution for each group of observations, in the spirit of \cite{Camerlenghi2019}.
Moreover, in the context of multigroup data, the combination of normalized hCRVs with probability kernels in mixture models could lead to a novel class of priors for density estimation and clustering \citep{Wade2025}. This raises interesting questions concerning the contraction rate to the true densities for partially exchangeable models \citep{Catalano2022}, the behavior of the mixing measure \citep{Nguyen2016}, and that of the clustering structure \citep{Ascolani2022}. Additionally, normalized hCRVs could be employed in stochastic block models for multi-layer network data \citep{Durante2025},
and, in the one-dimensional setting ($d = 1$), they could represent a viable alternative as priors beyond Gibbs type \citep{Camerlenghi2018} or as building blocks for multiview clustering \citep{Dombowsky2025}. Finally, from a numerical perspective, it would be of interest to evaluate the performance of hCRVs in more structured application domains, such as topic modelling, genomics, or sequential data, in order to assess the regimes in which the table-free representation yields the largest practical gains in terms of mixing, scalability, and computational cost. 


\section*{Funding}

Claudio Del Sole was partially supported by the European Union - Next Generation EU PRIN-PNRR (project P2022H5WZ9).

\singlespacing
\bibliographystyle{chicago}
\bibliography{references}


\clearpage
\onehalfspacing

\appendix

\renewcommand{\thesection}{S\arabic{section}}
\setcounter{section}{0}
\renewcommand{\thesubsection}{S\arabic{section}.\arabic{subsection}}
\setcounter{subsection}{0}
\renewcommand{\theequation}{S\arabic{equation}}
\setcounter{equation}{0}
\renewcommand{\thetheorem}{S\arabic{section}.\arabic{theorem}}
\setcounter{theorem}{0}
\renewcommand{\thefigure}{S\arabic{figure}}
\setcounter{figure}{0}
\renewcommand{\thetable}{S\arabic{table}}
\setcounter{table}{0}

\begin{center}
	\emph{\LARGE Supplementary Material to:} \\[.5\baselineskip]
	{\LARGE Hierarchical Random Measures without Tables}
\end{center}

\bigskip

\begin{abstract}
	\noindent
	Section~\ref{app:background} presents a brief and self-contained account on completely random measures, L\'evy measures, L\'evy intensities, Laplace exponents, and their multivariate extension to completely random vectors. Proofs of the results in the main manuscript are presented in Section~\ref{sec:proofs}. Section~\ref{app:sampling} provides additional technical details on the implementation of posterior sampling algorithms for the normalized gamma-gamma hCRV introduced in Section~\ref{sec:gamma-gamma}, and contains further numerical illustrations supporting their effectiveness. 
	Finally, Section~\ref{sec:elicitation} investigates the impact of the hyperparameters on the borrowing of information and shrinkage induced by the normalized gamma-gamma hCRV, and discusses the proper elicitation of its prior dependence structure.
\end{abstract}

\section{Background on completely random measures}
\label{app:background}

This appendix contains a concise account of completely random measures \citep{Kingman1967} and their multivariate extension to vectors of random measures. In particular, we introduce the key notions of L\'evy measure and Laplace exponent.

Let $(\mathbb{X}, d_{\mathbb{X}})$ be a Polish space endowed with a distance $d_{\mathbb{X}}$. The space $M_{\mathbb{X}}$ of boundedly finite measures on $\mathbb{X}$ is a Borel space with the topology of weak$^\sharp$ convergence \citep{DaleyVereJones2002}. A random measure is a measurable function $\crm: \Omega \to M_{\mathbb{X}}$ from some probability space $\Omega$. 

\begin{definition}
	A random measure $\crm: \Omega \to M_{\mathbb{X}}$ is a {\em completely random measure} ({\rm CRM}) if, given a finite collection of pairwise disjoint and bounded Borel sets $A_1, \cdots, A_k$ of $\X$, the random variables $\crm(A_1), \dots, \crm(A_k)$ are mutually independent.
\end{definition}

\noindent
\cite{Kingman1967} provides a remarkable representation theorem that decomposes any CRM $\crm$ into a unique sum $\crm = \mu_0 + \crm_f + \crm_c$ of three independent components: (i) a deterministic measure $\mu_0$, (ii) a random measure with fixed atoms $\crm_f$, and (iii) an a.s. discrete random measure without fixed atoms $\crm_c$. The use of CRMs as priors in Bayesian nonparametric models usually restricts to the third component. For this reason, unless differently specified, we only focus on this class of CRMs and henceforth assume $\crm = \crm_c$. Moreover, \cite{Kingman1967} shows that for any CRM $\crm$ there exists a Poisson random measure $\tilde N$ on $(0,+\infty) \times \X$ such that, for any Borel set $A$ of $\X$,
\begin{equation}
	\label{eq:poisson}
	\crm(A) = \iint_{(0,\infty) \times A} s \, \ddr \tilde N (s, x).
\end{equation}
The mean measure $\nu$ of the Poisson random measure $\tilde N$ characterizes the law of $\crm$ and is termed the \emph{L\'evy intensity} of the CRM, justifying the notation $\crm \sim {\rm CRM}(\nu)$. This measure on $(0,+\infty) \times \X$ can have infinite mass on sets of the form $(0,\varepsilon) \times A$. However, for every $\varepsilon>0$ and for every bounded Borel set $A$, it must satisfy the following constraints:
\begin{enumerate}
	\item[(a)] the jump component is bounded out of the origin, $\nu((\varepsilon,+\infty) \times A) <+\infty$;
	\item[(b)] the jump component is integrable in the origin, $\iint_{(0,\varepsilon)\times A} s \, \ddr \nu(s,x) <+\infty$;
	\item[(c)] the atom component has no point masses: for every $x \in \mathbb{X}$, $\nu((0,+\infty) \times \{x\}) = 0$.
\end{enumerate}
Details can be found in \citet[Theorem 10.1.III]{DaleyVereJones2007}. The identity \eqref{eq:poisson} is used in \cite{Kingman1967} to derive a L\'evy-Khintchine representation of the Laplace transform
\begin{equation}
	\label{eq:levy-khintchine}
	\log \E(e^{-\lambda \crm (A)})  = - \iint_{(0,+\infty) \times A} (1-e^{-\lambda s}) \, \ddr \nu(s,x),
\end{equation}
for every $\lambda \ge 0$. 
Therefore $\tilde \mu(A)$ has a pure-jump infinitely divisible distribution with L\'evy measure $\ddr \rho_A(s) = \int_{A} \ddr \nu(s, x)$, which we compactly write as  $\tilde \mu(A) \sim {\rm ID}(\rho_A)$; refer to \citet[Theorem 8.1]{Sato1999} for further details. For $A = \mathbb{X}$, the expression in \eqref{eq:levy-khintchine} provides the \emph{Laplace exponent} of the CRM, namely the function $\psi: [0,+\infty) \to [0,+\infty)$ defined as
\begin{equation*}
	\psi(\lambda) = - \log \E(e^{-\lambda \crm (\mathbb{X})}).
\end{equation*}
Remarkably, thanks to the independence of the evaluations on disjoint sets, the Laplace transform characterizes the law of the CRM.
The usual specification of L\'evy intensities is through disintegration
\begin{equation}
	\label{eq:disintegration}
	\ddr \nu(s,x) = \ddr \rho_x(s) \,\ddr P_0(x),
\end{equation}
where $P_0$ is a $\sigma$-finite atomless measure on $\X$ for condition (c), and $\rho_x$ is a $\sigma$-finite measure on $(0,+\infty)$, $P_0$-a.e., such that, for every $\varepsilon>0$,
\begin{equation}
	\label{eq:conditions_levy}
	\text{(a)} \quad \rho_x((\varepsilon,+\infty)) <+\infty, \qquad \text{(b)} \quad \int_{(0,\varepsilon)} s \, \ddr \rho_x(s) <+\infty.
\end{equation}
Note that $P_0$ does not have to be a probability measure nor a measure with finite mass, unless we extend condition (a) to unbounded Borel sets. In this case, we can assume $P_0$ to be a probability measure, and the disintegration \eqref{eq:disintegration} is unique. Conditions \eqref{eq:conditions_levy} imply that $\rho_x$ is a L\'evy measure on $(0,+\infty)$; when $\rho_x$ is absolutely continuous with respect to the Lebesgue measure, we term its Radon-Nikodym derivative \emph{L\'evy density}. In our context, we are interested in two additional conditions. Firstly, we consider L\'evy measures $\rho_x$ that have infinite mass near the origin, that is, such that $\rho_x((0,+\infty)) = \infty$; we term such L\'evy measures \emph{infinitely active}. Secondly, we consider disintegration in product form, which leads to the definition of homogeneous CRM.

\begin{definition}
	Let $\crm \sim {\rm CRM}(\nu)$ such that $\nu$ satisfies \eqref{eq:disintegration}; we refer to $\crm$ as a \emph{homogeneous} CRM if $\rho_x = \rho$, $P_0$-a.e, and write $\crm \sim {\rm CRM}(\rho \otimes P_0)$.
\end{definition}

\noindent
For a homogeneous $\crm \sim {\rm CRM}(\rho \otimes P_0)$ the L\'evy-Khintchine representation in \eqref{eq:levy-khintchine} simplifies consistently. Indeed, under homogeneity and assuming $P_0$ to be a probability measure, the Laplace exponent is 
$$\psi(\lambda) = \int_{(0,+\infty)} (1-e^{-\lambda s}) \, \ddr \rho(s),$$
and $\log \E(e^{-\lambda \crm (A)})  = - P_0(A) \, \psi(\lambda)$, for $\lambda \ge 0$. We observe that $\psi$ is a non-negative and infinitely differentiable function whose derivative is completely monotone. Moreover, it vanishes at 0 and its derivative vanishes at $+\infty$. Thanks to the L\'evy-Khintchine representation of Bernstein functions (see, e.g. \cite{Schilling2012}, Theorem 3.2) any such function is the Laplace exponent of a CRM. 

The two fundamental examples are the gamma CRM and the stable CRM. We recall their definitions, which are used as building blocks throughout the manuscript.

\begin{definition} \label{def:gamma} 
	A random measure $\crm \sim \text{CRM} (\rho \otimes P_0)$ is a gamma CRM of shape $\alpha>0$ and rate $b>0$ if $\rho$ has L\'evy density and Laplace exponent equal to, respectively,
	\begin{equation*}
		\rho(s) = \alpha \frac{e^{-bs}}{s}, \qquad \psi(\lambda) = \alpha \log \bigg(1+\frac{\lambda}{b}\bigg).
	\end{equation*}
\end{definition}

\begin{definition} \label{def:stable} 
	A random measure $\crm \sim \text{CRM} (\rho \otimes P_0)$ is a stable CRM of shape $\alpha>0$ and discount parameter $\sigma \in (0,1)$ if $\rho$ has L\'evy density and Laplace exponent equal to, respectively,
	\begin{equation*}
		\rho(s) = \frac{\alpha \sigma}{\Gamma(1-\sigma)}\frac{1}{s^{1+\sigma}}, \qquad \psi(\lambda) = \alpha \lambda^\sigma.
	\end{equation*}
\end{definition}

The notion of completely random measure can be naturally extended to vectors $\bm{\crm} 
= (\crm_1,\dots,\crm_d)$ of random measures, as in Definition~\ref{def:crv}. Following \cite{Catalano2021}, we term them \emph{completely random vectors} (CRVs). The representation in \eqref{eq:poisson} can be extended to CRVs by considering a multivariate L\'evy intensity $\nu$ on $\Omega_d \times \X$, where $\Omega_d  = [0,+\infty)^d \setminus \{\bm{0}\}$. Similarly, the Laplace transform of the random vector $\bm{\crm}(A)$ is characterized by
\begin{equation*}
	\label{eq:levy-khintchine_crv}
	\log \E(e^{-\bm{\lambda} \cdot \bm{\crm}(A)})  = - \iint_{\Omega_d \times A} (1-e^{-\bm{\lambda} \cdot \bm{s}}) \, \ddr \nu(\bm{s},x),
\end{equation*}
where $\bm{\lambda}= (\lambda_1,\dots,\lambda_d) \in [0,+\infty)^d, \bm{s}= (s_1,\dots,s_d) \in \Omega_d$ and $\cdot$ denotes the scalar product. The multivariate Laplace exponent is defined as
\begin{equation}
	\label{eq:laplace_exponent_crv}
	\psi(\bm{\lambda}) = - \log \E(e^{-\bm{\lambda} \cdot \bm{\crm}(\X)}) = \iint_{\Omega_d \times \mathbb{X}} (1-e^{-\bm{\lambda} \cdot \bm{s}}) \, \ddr \nu(\bm{s},x),
\end{equation}
Consistently with the univariate case, a CRV is homogeneous if $\nu = \rho \otimes P_0$, where $P_0$ is an atomless measure on $\mathbb{X}$ and $\rho$ is $d$-dimensional L\'evy measure on $\Omega_d$ such that 
\begin{equation*}
	\label{eq:conditions_levy_crv}
	\text{(i)} \quad \rho_x((\varepsilon,+\infty)^d) <+\infty, \qquad 
	\text{(ii)} \quad \int_{[0,\varepsilon)^d \setminus \{\bm{0}\}} \|\bm{s}\| \, \ddr \rho_x(\bm{s}) <+\infty.
\end{equation*}
Under homogeneity, if $P_0$ is a probability measure, the multivariate Laplace exponent is 
\begin{equation*}
	\psi(\bm{\lambda}) = \int_{\Omega_d} (1-e^{-\bm{\lambda} \cdot \bm{s}}) \, \ddr \rho(\bm{s}),
\end{equation*}
and $\log \E(e^{-\bm{\lambda} \cdot \bm{\crm}(A)}) = - P_0(A) \, \psi(\bm{\lambda})$, for $\bm{\lambda} \in [0,+\infty)^d$.
Finally, if $\bm{\crm} \sim {\rm CRV}(\nu)$, it easily follows from the definition that the marginal random measures $\crm_i$ are CRMs. Their L\'evy intensities can be obtained by marginalization of $\nu$ as $\ddr \nu_i(s_i,x) = \int_{ \bm{s}_{-i} \in [0,+\infty)^{d-1}} \, \ddr \nu(\bm{s},x)$, where $\bm{s}_{-i} = (s_1,\dots,s_{i-1},s_{i+1},\dots,s_d)$.


\section{Proofs}
\label{sec:proofs}

\allowdisplaybreaks
\subsection{Proof of Theorem~\ref{th:crv}}

\emph{Step 1. Show that $\bm{\crm}$ is a CRV.} 
Let $\bm{\crm}(A) = ( \crm_1(A),\dots, \crm_d(A))$ be the random vector of evaluations of $\bm{\crm}$ on the Borel set $A$ of $\X$. First, we observe that $\bm{\crm}$ takes values in the space of boundedly finite measures. Indeed, by Definition~\ref{def:hcrv}, for every bounded Borel set \(A\) and conditionally on \(\crm_0\),
$$ \crm_i(A)\mid \crm_0 \sim \id(\crm_0(A)\rho), $$
where \(\id(\eta)\) denotes the infinitely divisible distribution with L\'evy measure \(\eta\). Since \(\crm_0\) is a completely random measure, \(\crm_0(A)<+\infty\) almost surely; hence, \(\crm_0(A)\rho\) is a L\'evy measure, and \(\crm_i(A)<+\infty\) a.s. for every \(i=1,\ldots,d\). At this point, we prove that the random vectors $\bm{\crm}(A_1), \dots, \bm{\crm}(A_k)$ are mutually independent, for every $A_1, \cdots, A_k$ mutually disjoint sets of $\X$. Specifically, we show that all linear combinations are mutually independent, that is, for coefficients $\lambda_{ij}>0$, with $i=1,\dots,d$ and $j=1,\dots, k$,
\begin{equation*}
	\E\Big(e^{- \sum_{j=1}^k \sum_{i=1}^d \lambda_{ij} \crm_i(A_j)}\Big) = \prod_{j=1}^k \E\Big(e^{- \sum_{i=1}^d \lambda_{ij} \crm_i(A_j)}\Big).
\end{equation*}
This identity is proved exploiting the following properties:
\begin{itemize}
	\item[(i)] $\crm_1,\dots,\crm_d$ are conditionally independent given $\crm_0$;
	\item[(ii)] since $\crm_i \mid \crm_0$ is a CRM, its evaluations on disjoint sets are independent;
	\item[(iii)] Since $\crm_i(A_j) \mid \crm_0 \sim \id(\crm_0(A_j) \rho)$, then $\mathcal{L}(\crm_i(A_j) \mid \crm_0) = \mathcal{L}(\crm_i(A_j) \mid \crm_0(A_j))$, where $\mathcal{L}(X)$ is the probability law of the random object $X$;
	\item[(iv)] since $\crm_0$ is a CRM, $\crm_0(A_1),\dots, \crm_0(A_k)$ are independent;
	\item[(v)] again, $\crm_1,\dots,\crm_d$ are conditionally independent given $\crm_0$.
\end{itemize}
This entails
\begin{align*}
	\E\Big(e^{- \sum_{j=1}^k \sum_{i=1}^d \lambda_{ij} \crm_i(A_j)}\Big) 
	& \overset{\text{(i)}}{=} \E \bigg(\prod_{i=1}^d \E\Big(e^{- \sum_{j=1}^k \lambda_{ij} \crm_i(A_j)} \mid \crm_0 \Big) \bigg) \\
	& \overset{\text{(ii)}}{=} \E \bigg(\prod_{i=1}^d \prod_{j=1}^k \E\Big(e^{- \lambda_{ij} \crm_i(A_j)}\mid \crm_0\Big)\bigg) \\
	& \overset{\text{(iii)}}{=} \E \bigg(\prod_{i=1}^d \prod_{j=1}^k \E\Big(e^{- \lambda_{ij} \crm_i(A_j)} \mid \crm_0(A_j)\Big)\bigg) \\
	& \overset{\text{(iv)}}{=} \prod_{j=1}^k \E \bigg(\prod_{i=1}^d \E\Big(e^{- \lambda_{ij} \crm_i(A_j)} \mid \crm_0(A_j)\Big)\bigg) \\
	& \overset{\text{(v)}}{=} \prod_{j=1}^k \E\Big(e^{- \sum_{i=1}^d \lambda_{ij} \crm_i(A_j)}\Big).
\end{align*}

\emph{Step 2. Derive the Laplace transform.} 
From the specification in Definition~\ref{def:hcrv}, it follows that $\crm_i(A) \mid \crm_0 \simiid \id(\crm_0(A) \rho)$. Therefore,
\begin{align*}
	\log \E\Big(e^{- \bm{\lambda} \cdot \bm{\crm}(A)}\Big) 
	& = \log \E \Big(\E\Big(e^{- \bm{\lambda} \cdot \bm{\crm}(A)} \mid \crm_0\Big)\Big) 
	= \log \E \bigg( \prod_{i=1}^d e^{- \psi(\lambda_i) \crm_0(A)} \bigg) \\
	& = \log \E \Big( e^{- \crm_0(A) \sum_{i=1}^d \psi(\lambda_i)} \Big) = - P_0(A) \, \psi_0\bigg(\sum_{i=1}^d \psi(\lambda_i)\bigg).
\end{align*}

\emph{Step 3. Determine homogeneity and Laplace exponent.} 
The multivariate L\'evy intensity $\nu_h$ and the Laplace exponent $\psi_h$ of $\bm{\crm}$ are uniquely characterized by the Laplace transform. Given the product form of Step 2, it follows that $\nu_h = \rho_h \otimes P_0$ for some multivariate $\rho_h$ and the Laplace exponent satisfies
\begin{equation*}
	\psi_h (\bm{\lambda}) = \psi_0 \bigg(\sum_{i=1}^d \psi(\lambda_i) \bigg).
\end{equation*}

\emph{Step 4. Determine the L\'evy measure.} 
The L\'evy measure $\rho_h$ is characterized by 
\begin{equation*}
	\psi_h (\bm{\lambda}) = \int_{\Omega_d} (1-e^{-\bm{\lambda} \cdot \bm{s}}) \, \ddr \rho_h (\bm{s}).    
\end{equation*}
We use the notation $\mathscr{L}_{\rho}(\lambda) = \E(e^{-\lambda X}) = e^{-\psi(\lambda)}$, where $\lambda\in[0,+\infty)$ and $X \sim \id(\rho)$, to denote the Laplace transform of an infinitely divisible random variable with L\'evy measure $\rho$ and Laplace exponent $\psi$; note that $\mathscr{L}_{\rho}(\lambda)^t = \mathscr{L}_{t\rho}(\lambda) $. Defining $X_i \simiid \id(t\rho)$, one has
\begin{align*}
	\psi_h (\bm{\lambda}) & = \psi_0 \bigg(\sum_{i=1}^d \psi(\lambda_i) \bigg) \\
	& = \int_0^{+\infty} \Big(1-e^{-\sum_{i=1}^d \psi(\lambda_i)\,t}\Big) \, \ddr \rho_0(t) = \int_0^{+\infty} \bigg(1-\prod_{i=1}^d e^{- \psi(\lambda_i)\,t} \bigg) \, \ddr \rho_0(t) \\
	& = \int_0^{+\infty}  \bigg(1-\prod_{i=1}^d \mathscr{L}_{\rho}(\lambda_i)^t \bigg)  \, \ddr \rho_0(t) = \int_0^{+\infty}  \bigg(1-\prod_{i=1}^d \mathscr{L}_{t\rho}(\lambda_i) \bigg)  \, \ddr \rho_0(t) \\
	& = \int_0^{+\infty} \bigg(1-\prod_{i=1}^d \E_{X_i}\Big(e^{-\lambda_i X_i}\Big)\bigg) \, \ddr \rho_0(t) = \int_0^{+\infty} \E_{ \bm{X} }\Big(1-e^{-\sum_{i=1}^d \lambda_i X_i}\Big) \, \ddr \rho_0(t) \\
	& = \int_0^{+\infty} \int_{[0,+\infty)^d} \Big(1-e^{-\sum_{i=1}^d \lambda_i s_i}\Big) \prod_{i=1}^d \ddr P_{\scriptscriptstyle \id(t \rho)} (s_i)   \, \ddr \rho_0(t) \\
	& = \int_0^{+\infty} \int_{\Omega_d} \Big(1-e^{-\sum_{i=1}^d \lambda_i s_i}\Big) \prod_{i=1}^d \ddr P_{\scriptscriptstyle \id(t \rho)} (s_i)  \, \ddr \rho_0(t) \\
	& = \int_{\Omega_d}  ( 1- e^{-\bm{\lambda}\cdot \bm{s}})  \int_0^{+\infty} \prod_{i=1}^d \ddr P_{\scriptscriptstyle \id(t \rho)} (s_i) \, \ddr \rho_0(t) ,
\end{align*}
by exchanging the order of the integrals thanks to Fubini-Tonelli theorem, and observing that the integrand vanishes in $\bm{s} = \bm{0}$. If $\id(t \rho)$ has density $f_{\scriptscriptstyle \id(t \rho)}$ with respect to the Lebesgue measure, then this is equal to
\begin{equation*}
	\int_{\Omega_d}  ( 1- e^{-\bm{\lambda}\cdot \bm{s}})  \int_0^{+\infty} \prod_{i=1}^d f_{\scriptscriptstyle \id(t \rho)}(s_i)   \, \ddr \rho_0(t) \, \ddr \bm{s}.
\end{equation*}

\subsection{Proof of Remark~\ref{rem:compound}}
\label{proof_remark}

Assume that the L\'evy measure $\rho_{\rm co}$ of a compound random measure coincides with $\rho_h$ in Theorem~\ref{th:crv} for some base L\'evy measure $\rho_0$.
Since $\rho_{\rm co}$, $\rho_h$ and $\rho_0$ are $\sigma$-finite measures, by the uniqueness of the disintegration of measures $\rho_0$-a.e.~it holds that
\begin{equation*}
	\frac{1}{t^d} \, H \bigg(\frac{s_1}{t},\dots,\frac{s_d}{t}\bigg) \ddr \bm{s} = \prod_{i=1}^d \ddr P_{\scriptscriptstyle \id(t \rho)}(s_i).
\end{equation*}
This forces $H$ to be in product form and $\id(t\rho)$ to have a density. Therefore, there exists a density $H_1$ on $[0,+\infty)$ such that
\begin{equation*}
	\frac{1}{t} \, H_1\bigg(\frac{s}{t}\bigg) = f_{\scriptscriptstyle \id(t \rho)}(s), \qquad \text{$\rho_0$-a.e.}
\end{equation*}
Let $X$ be a random variable with density $H_1$; then $tX$ has density $f(s) = H_1(s/t)/t$. From the identity above, $tX$ is an infinitely divisible distribution with L\'evy density $t \rho(s)$. Therefore, $X$ has (strictly) stable distribution \citep[Definition 13.1 and Theorem 14.3]{Sato1999}; in particular, it has stability index $1$, corresponding to a L\'evy density $\rho(s) \propto s^{-2}$. However, this choice of $\rho(s)$ does not satisfy the integrability condition (b) in \eqref{eq:conditions_levy}, and thus is not a valid L\'evy density for a CRM.

\subsection{Proof of Theorem~\ref{th:identifiability}}

\emph{Step 1. Express the condition in terms of Laplace exponents.} 
Since hierarchical CRVs are CRVs, their law is uniquely determined by their multivariate Laplace exponent. 
The $1$-dimensional Laplace exponents $\psi_0^{(1)}$ and $\psi^{(2)}$, associated to $\rho_0^{(1)}$ and $\rho^{(2)}$ respectively, are strictly increasing by definition, which implies that they are invertible. Therefore, from the expression of Theorem~\ref{th:crv}, $\bm{\crm}_1$ and $\bm{\crm}_2$ have the same multivariate Laplace exponent if and only if 
\begin{equation*}
	\sum_{i=1}^d \psi^{(1)} \circ (\psi^{(2)})^{-1} (s_i) = (\psi_0^{(1)})^{-1}\circ \psi_0^{(2)} \bigg(\sum_{i=1}^d s_i \bigg), \qquad P_0^{(1)} = P_0^{(2)},
\end{equation*}
for every $(s_1,\dots,s_d) \in \Omega_d$. Define $f = \psi^{(1)} \circ (\psi^{(2)})^{-1}$ and $g = (\psi_0^{(1)})^{-1}\circ \psi_0^{(2)}$. Since the inverse of a continuous function is continuous, $f,g: (0,+\infty) \to (0,+\infty)$ are continuous functions such that $f(0) = g(0) = 0$ and
\begin{equation} \label{eq:add}
	f (s_1)+\dots + f(s_d) = g (s_1+\dots + s_d).
\end{equation}

\emph{Step 2. Show that $f = g$ and are linear functions.} 
Let $s_1 = s$ and $s_2 = \dots = s_d = 0$. Then, since $f(0) = 0$, we have $g = f$. To prove linearity, we show that $f(s_1 + s_2) = f(s_1) + f(s_2)$ for every $s_1,s_2 >0$ and $f(c s) = c f(s)$ for every $c,s>0$. The first property follows by taking $s_3 = \dots = s_d = 0$ in \eqref{eq:add}. As for the second, we prove it first for $c$ a natural number and then for $c$ a rational number. We conclude by density thanks to the continuity of $f$. For $n \in \mathbb{N}$ natural number, by \eqref{eq:add}, $f(ns) = f(s + \dots + s) = n f(s)$, while for $q = n/m$ rational number, with $n,m \in \mathbb{N}$, $ nf(s) = f(ns) = f(qms) = m f(qs)$, which implies $f(qs) = nf(s)/m = q f(s)$.

\emph{Step 3. Conclusion.} We have proved that there exists $c>0$ such that, for every $s>0$, $$\psi^{(1)} \circ (\psi^{(2)})^{-1} (s)= (\psi_0^{(1)})^{-1}\circ \psi_0^{(2)}(s) = c s.$$ 
Thus for every $s>0$,
$$\psi_0^{(2)}(s) = \psi_0^{(1)} (c s), \qquad  \psi^{(2)}(s) = \frac{1}{c} \, \psi^{(1)} (s), \qquad P_0^{(1)} = P_0^{(2)}, $$
Using a change of variable, this condition can be equivalently expressed in terms of the L\'evy measures as
\begin{equation*}
	\rho_0^{(2)} = c_{\#} \rho_0^{(1)}; \qquad  \rho^{(2)}(s) = \frac{1}{c} \, \rho^{(1)} (s); \qquad P_0^{(1)} = P_0^{(2)}.
\end{equation*}
If the L\'evy measures have densities, this is equivalent to
\begin{equation*}
	\rho_0^{(2)}(s) = \frac{1}{c} \, \rho_0^{(1)} \bigg(\frac{s}{c} \bigg);\qquad  \rho^{(2)}(s) = \frac{1}{c} \, \rho^{(1)} (s); \qquad P_0^{(1)} = P_0^{(2)},
\end{equation*}
which coincides with the statement of the proposition by substituting $c$ with $1/c$.

\subsection{Proof of Lemma~\ref{th:infinite_activity}}

\emph{Step 1. Find the L\'evy measure $\rho_1$ of $\crm_i$.} 
By Theorem~\ref{th:crv}, the multivariate L\'evy measure of $\bm \crm$ is the measure on $\Omega_d$ satisfying
\begin{equation*}
	\ddr \rho_{h}(s_1,\dots, s_d) = \int_0^{+\infty} \prod_{i=1}^d \ddr P_{\scriptscriptstyle \id(t \rho)} (s_i)  \, \ddr \rho_0(t).
\end{equation*}
The marginal L\'evy measure of each $\crm_i$ coincides with the one of $\crm_1$,
that is,
\begin{equation*}
	\ddr \rho_1(s) =  \int_{[0,+\infty)^{d-1}}\int_0^{+\infty} \ddr P_{ \scriptscriptstyle \id(t \rho)}(s) \, \prod_{i=2}^d \ddr P_{\scriptscriptstyle \id(t \rho)} (s_i) \, \ddr \rho_0(t) = \int_0^{+\infty} \ddr P_{\scriptscriptstyle \id(t \rho)}(s)  \, \ddr \rho_0(t),
\end{equation*}
where we have exchanged the order of integration thanks to Fubini-Tonelli's theorem, since measures are positive and $\sigma$-finite.

\emph{Step 2. Show that if $\rho$ and $\rho_0$ are infinitely active, then $\rho_1$ is infinitely active.} 
If $\rho$ is infinitely active, then $\id(t \rho)$ gives zero probability to the origin, thus
\begin{equation*}
	\int_{\Omega_1}\ddr \rho_1(s) = \int_{\Omega_1}\int_0^{+\infty} \ddr P_{ \scriptscriptstyle \id(t \rho)}(s)  \, \ddr \rho_0(t) = \int_0^{+\infty} \, \ddr \rho_0(t),
\end{equation*}
where again we have used Fubini-Tonelli's theorem to exchange the order of the integrals. Since $\rho_0$ is infinitely active, the last integral is $+\infty$.

\emph{Step 3. If $\rho_0$ is finitely active, then $\rho_1$ is finitely active}. 
Since $\id(t \rho)$ is a probability measure on $[0,+\infty)$, its mass on $(0,+\infty)$ is smaller than or equal to 1. Thus, by Fubini-Tonelli,
\begin{equation*}
	\int_{\Omega_1}\ddr \rho_1(s) = \int_0^{+\infty }\bigg( \int_{\Omega_1} \ddr P_{\scriptscriptstyle \id(t \rho)} (s) \bigg) \, \ddr \rho_0(t) \le \int_0^{+\infty} \, \ddr \rho_0(t) <+\infty,
\end{equation*}
where the last inequality is due to the finite activity of $\rho_0$.

\emph{Step 4: If $\rho$ is finitely active, then $\rho_1$ is finitely active.}
If $\rho$ is finitely active with total mass $a$, then $Y \sim \id(t \rho)$ has positive mass in 0, namely $\Prob(Y=0) = e^{-t a}$. This follows by observing that $Y$ is a compound Poisson distribution, and $\Prob(Y=0)$ coincides with the probability mass in $0$ of a Poisson random variable with rate $\lambda = a t$. Thus, by Fubini-Tonelli's theorem,
\begin{equation*}
	\int_{\Omega_1}\ddr \rho_1(s) = \int_0^{+\infty} \int_{\Omega_1} \ddr P_{\scriptscriptstyle \id(t \rho)} (s) \ \ddr \rho_0(t) =  \int_0^{+\infty} (1-e^{-at}) \, \ddr \rho_0(t) = \psi_0(at) <+\infty. 
\end{equation*}

\subsection{Proof of Proposition~\ref{th:hdp_hyper}}

The Dirichlet process arises as a normalization of a gamma CRM \citep{Ferguson1973}, thus
\begin{equation*}
	\frac{\crm_1}{\crm_1(\X)},\dots,\frac{\crm_d}{\crm_d(\X)} \Bigm| \crm_0 \, \simiid \, {\rm DP}( \alpha \crm_0) = {\rm DP}\bigg( \alpha \crm_0(\X) \,\frac{\crm_0}{\crm_0(\X)}\bigg). 
\end{equation*}
For a gamma CRM, the total mass $\crm_0(\X)$ is independent from the normalization \cite[Lemma 1]{Vershik2004}. Moreover, $\crm_0(\X) \sim \text{Gamma}(\alpha_0, b_0)$ and thus $\alpha \crm_0(\X) \sim \text{Gamma}(\alpha_0, b_0/\alpha)$ is independent of $\tilde P_0 = \crm_0 / \crm_0(\X)$.

\subsection{Proof of Proposition~\ref{th:stable}}

We prove that, conditionally on $\tilde \mu_0$, it holds $\crm^{(1)} \eqd c(\crm_0) \, \crm^{(2)}$ for some positive value $c(\crm_0)$ depending on $\tilde \mu_0$. This implies that the normalizations of $\crm^{(1)}$ and $\crm^{(2)}$ are equal in distribution.
Conditionally on $\crm_0$, both $\crm^{(1)}$ and $\crm^{(2)}$ are CRMs, and we can express the condition $\crm^{(1)} \eqd c(\crm_0) \, \crm^{(2)}$ in terms of their conditional L\'evy measures $\rho^{(1)}$ and $\rho^{(2)}$ as
\begin{equation*}
	\rho^{(2)}(s) \otimes \frac{\crm_0}{\crm_0(\X)} \eqd c(\crm_0) \, \rho^{(1)}(c(\tilde \mu_0) s) \otimes \crm_0.
\end{equation*}
Plugging in the expressions for the stable L\'evy measures, we need to find $c(\crm_0)$ such that
\begin{equation*}
	\frac{\alpha \, \sigma}{\Gamma(1-\sigma)} \frac{1}{s^{\sigma+1}} \otimes \frac{\crm_0}{\crm_0(\X)} \eqd \frac{\alpha \, \sigma}{\Gamma(1-\sigma)} \frac{1}{c(\tilde \mu_0)^{\sigma}} \frac{1}{s^{\sigma+1}} \otimes \crm_0.
\end{equation*}
The proof is concluded by choosing $c(\crm_0) = \crm_0(\X)^{1/\sigma}$.

\subsection{Proof of Proposition~\ref{th:moments_norm}}

Firstly, we state and prove a related result concerning the unnormalized random measures. For this purpose, denote by $M_r(\rho)$ the $r$-th moment of a L\'evy measure $\rho$,
\begin{equation*}
	M_r(\rho) = \int_{[0,\infty)} s^r\, \ddr \rho(s).
\end{equation*}

\begin{proposition}
	\label{th:moments}
	Let $\bm{\crm}\sim \mathrm{hCRV}(\rho, \rho_0, P_0)$ and let $A$ be a Borel set s.t. $P_0(A) \neq 0,1$. Then, for every $i \neq j$,
	\begin{gather*}
		\E(\crm_i(A)) = P_0(A) M_1(\rho_0) M_1(\rho), \\ 
		\var(\crm_i(A)) = P_0(A) M_1(\rho_0) M_2(\rho)  + P_0(A) M_2(\rho_0) M_1(\rho)^2, \\
		\cov(\crm_i(A), \crm_j(A)) = P_0(A) M_2(\rho_0) M_1(\rho)^2.
	\end{gather*}
	In particular, for every $i \neq j$,
	\begin{equation*}
		{\rm corr}(\crm_i(A), \crm_j(A)) = \frac{M_2(\rho_0) M_1(\rho)^2}{ M_1(\rho_0) M_2(\rho)  +  M_2(\rho_0) M_1(\rho)^2}.
	\end{equation*}
\end{proposition}

\begin{proof}
	The expressions can be derived (i) through the hierarchical structure using the tower property and the law of total (co)variance, or (ii) using the expression of the moments of (jointly) infinitely divisible distributions in terms of their L\'evy measures. We provide a proof exploiting both techniques. Recall that, by Campbell's theorem, the mean and variance of an infinitely divisible random variable $X \sim \id(\rho)$ satisfy
	\begin{equation*}
		\E(X) = \int_0^{+\infty} s \, \ddr \rho(s) = M_1(\rho), \qquad \var(X) = \int_0^{+\infty} s^2 \, \ddr \rho(s) = M_2(\rho).
	\end{equation*}
	\emph{Proof through hierarchical structure}. 
	Since $\crm_i(A) \mid \crm_0 \sim \id( \crm_0(A) \rho)$, by the tower property, 
	\begin{equation*}
		\E(\crm_i(A)) = \E( \E(\crm_i(A) \mid \crm_0)) = \E \left( \crm_0(A) \int_0^{+\infty} s \, \ddr \rho (s)\right) =  P_0(A) M_1(\rho_0) M_1(\rho).
	\end{equation*}
	Similarly, by the law of total variance,
	\begin{align*}
		\var(\crm_i(A)) & = \E( \var(\crm_i(A) \mid \crm_0)) + \var( \E( \crm(A) \mid \crm_0)) \\
		& = \E \bigg(  \crm_0(A) \int_0^{+\infty} s^2 \, \ddr \rho (s) \bigg) + \var \bigg( \crm_0(A) \int_0^{+\infty} s \, \ddr \rho (s)\bigg) \\
		& = P_0(A) M_1(\rho_0) M_2(\rho)  + P_0(A) M_2(\rho_0) M_1(\rho)^2. 
	\end{align*}
	Finally, by the law of total covariance and thanks to the conditional independence of the random measures $\crm_i$ and $\crm_j$ for $i\neq j$, given $\crm_0$,
	\begin{align*}
		\cov(\crm_i(A), \crm_j(A)) & = \cov ( \E(\crm_i(A) \mid \crm_0), \E(\crm_j(A) \mid \crm_0)) \\
		& = \var \bigg( \crm_0(A) \int_0^{+\infty} s \, \ddr \rho (s)\bigg) = P_0(A) M_2(\rho_0) M_1(\rho)^2.
	\end{align*}
	\emph{Proof through joint infinite divisibility}. 
	Since $\crm_i$ is a CRM, by considering its L\'evy measure in Theorem~\ref
	{th:crv}, \cut
	\begin{align*}
		\E(\crm_i(A)) & = P_0(A) \int_0^{+\infty} s \int_0^{+\infty} \ddr P_{ \scriptscriptstyle \id(t \rho)}(s) \, \ddr \rho_0(t) \\
		& = P_0(A)  \int_0^{+\infty} \int_0^{+\infty} s \, \ddr P_{\scriptscriptstyle \id(t \rho)} (s) \, \ddr \rho_0(t) \\
		& = P_0(A) \int_0^{+\infty} \E_{X \sim \id(t \rho)} (X) \, \ddr \rho_0(t) \\
		& = P_0(A) M_1(\rho) \int_0^{+\infty} t \, \ddr \rho_0(t) = P_0(A) M_1(\rho) M_1(\rho_0).
	\end{align*}
	Similarly, for the variance,
	\begin{align*}
		\var(\crm_i(A)) & = P_0(A) \int_0^{+\infty} s^2 \int_0^{+\infty} \ddr P_{\scriptscriptstyle \id(t \rho)}(s) \, \ddr \rho_0(t) \\
		& = P_0(A)  \int_0^{+\infty} \int_0^{+\infty} s^2 \, \ddr P_{\scriptscriptstyle \id(t \rho)}(s) \, \ddr \rho_0(t) \\
		& = P_0(A) \int_0^{+\infty} \E_{X \sim \id(t \rho)} (X^2) \, \ddr \rho_0(t) \\
		& = P_0(A) \int_0^{+\infty} (\var_{X \sim \id(t \rho)}(X) + \E_{X \sim \id(t \rho)}(X)^2) \, \ddr \rho_0(t) \\
		& = P_0(A) \int_0^{+\infty} (t M_2(\rho) + t^2 M_1(\rho)^2) \, \ddr \rho_0(t) \\
		& = P_0(A) M_2(\rho) M_1(\rho_0) + P_0(A) M_1(\rho)^2 M_2(\rho_0).
	\end{align*}
	For the covariance, we observe that, for any $(X_1,X_2) \sim \id(\rho)$ jointly infinitely divisible random variables with a multivariate L\'evy measure $\rho$, it holds $\cov(X_1,X_2) = \int_{\Omega_2} s_1 s_2 \, \ddr \rho(s_1,s_2)$; see, e.g., \citet[25.8]{Sato1999}. Thus,
	\begin{align*}
		\cov(\crm_i(A), \crm_j(A)) 
		& = P_0(A) \int_{\Omega_2} s_1 s_2 \, \ddr \rho_{h}(s_1,s_2) \\
		& = P_0(A) \int_0^{+\infty} \int_{\Omega_2} s_1 s_2 \, \ddr P_{\scriptscriptstyle \id(t \rho)}(s_1) \, \ddr P_{\scriptscriptstyle \id(t \rho)}(s_2) \, \ddr \rho_0(t) \\
		& = P_0(A) \int_0^{+\infty} \E_{X \sim \id(t \rho)}(X)^2 \, \ddr \rho_0(t) \\
		& = M_1(\rho) ^2  \int_0^{+\infty} t^2 \ddr \rho_0(t) = P_0(A) M_1(\rho) ^2 M_2(\rho_0). \tag*{\qedhere}
	\end{align*}
\end{proof}

We now prove Proposition~\ref{th:moments_norm}. Similarly to the proof of Proposition~\ref{th:moments}, we could derive the desired results in two ways: (i) exploiting the hierarchical structure, or (ii) using the fact that we are normalizing a CRV. When focusing on the normalization, exploiting the CRV structure is particularly convenient for deriving the mean and variance, whereas using the hierarchical structure brings to a straightforward calculation for the covariance. The univariate results in \cite{James2006} show that if $\crm \sim \CRM(\rho \otimes P_0)$, where $P_0$ is a probability measure and $\psi$ denotes its Laplace exponent, the normalization $\tilde P = \crm/\crm(\X)$ satisfies
\begin{equation*}
	\E(\tilde P(A)) = P_0(A), \qquad \var(\tilde P(A)) = - P_0(A)(1-P_0(A)) \int_0^{+\infty} u e^{-\psi(u)} \psi'' (u) \, \ddr u.
\end{equation*}
Since $\tilde P_i(A)$ is the normalization of $\crm_i \sim \CRM( \rho_h \otimes P_0)$, this immediately implies that $\E(\tilde P_i(A)) = P_0(A)$. Moreover, by considering the Laplace exponent of $\crm_i$ in Theorem~\ref
{th:crv},
\begin{equation*}
	\var(\tilde P_i(A)) = - P_0(A)(1-P_0(A)) \int_0^{+\infty} u e^{-\psi_0(\psi(u))} (\psi_0 \circ \psi)'' (u) \, \ddr u.
\end{equation*}
By the law of total covariance and the conditional independence of the random probability measures $\tilde P_i$ and $\tilde P_j$ for $i\neq j$, given $\crm_0$,
\begin{align*}
	\cov(\tilde P_i(A), \tilde P_j(A)) & = \cov ( \E(\tilde P_i(A) \mid \crm_0), \E(\tilde P_j(A) \mid \crm_0)) \\
	& = \var \bigg(\frac{\crm_0(A)}{\crm_0(\X)}\bigg) = - P_0(A)(1-P_0(A)) \int_0^{+\infty} u e^{-\psi_0(u)} \psi_0'' (u) \, \ddr u.
\end{align*}

\subsection{Proof of Theorem~\ref{th:posterior_crv}}

Since the random probabilities $\crm_i/\crm_i(\X)$ are a.s. discrete with random atoms from a continuous distribution, the model for the observations $\bm{X}_{1:d} \mid \bm{\crm}$ is non-dominated. Therefore, we cannot rely on Bayes theorem to find the posterior distribution, and should use an alternative strategy, based on the multivariate Laplace functional. This proof can be seen as the multivariate extension of the proof in \cite{James2009}. The multivariate Laplace functional
\begin{equation*}
	\E \exp \left( - \sum_{i=1}^d \int f_i \,\ddr \crm_i \right)
\end{equation*}
is defined for any non-negative measurable functions $f_1,\dots f_d$ and characterizes the law of any vector of random measures. In particular, when $\bm{\crm} \sim {\rm CRV}(\nu)$, it satisfies
\begin{equation}
	\label{laplace_functional_crv}
	\log \E \exp \left( - \sum_{i=1}^d \int f_i \,\ddr \crm_i \right) = - \int_{\Omega_d \times \X} (1-e^{- \sum_{i=1}^d s_i f_i(x)}) \, \ddr \nu(\bm{s},x).
\end{equation}
Our goal is to find an appropriate expression of $\E( e^{- \sum_{i=1}^d \int f_i \,\ddr \crm_i} \mid \bm{X}_{1:d} = \bm{x}_{1:d})$, for any non-negative measurable functions $f_1,\dots f_d$. Throughout the proof, we exploit the following properties of the conditional expectation, which hold for any random variable $X,Y$ and for any Borel set $A$ such that $\Prob(Y \in A) \ge 0$:
\begin{itemize}
	\item[(a)] conditional expectation w.r.t.~events: $\E(X \mid Y \in A) = \E(X \mathbbm{1}_A(Y))/\Prob(Y \in A)$;
	\item[(b)] tower property: $\E(X) = \E( \E(X\mid Y))$ and $\Prob(A) = \E(\mathbbm{1}_A) = \E( \Prob(A \mid Y))$.
\end{itemize} 

\smallskip
\emph{Step 1. Express the conditional expectation in terms of events.} 
By the dominated convergence theorem and the exchangeability of the observations
\begin{equation*}
	\E\Big( e^{- \sum_{i=1}^d \int f_i \ddr \crm_i } \mid \bm{X}_{1:d} = \bm{x}_{1:d}\Big) = \lim_{\varepsilon \to 0} \, \E \bigg( e^{- \sum_{i=1}^d \int f_i \ddr \crm_i } \Bigm| \bm{X}_{1:d} \in \prod_{i=1}^d \prod_{j=1}^k B_{\varepsilon}^{n_{ij}}(x_j^*) \bigg),
\end{equation*}
where $\bm{x}^* = (x_1^*,\dots,x_k^*)$ are the unique values in the observations, with multiplicities $n_{i1},\dots,n_{ik}$ for each group $i = 1, \dots, d$; see Section~\ref{sec:posterior}. Here, $B_{\varepsilon}(x) = \{ \omega: d(x,\omega) \le \varepsilon\}$ denotes the ball of radius $\varepsilon > 0$ centered in $x$, and we use $B^{m}_{\varepsilon}(x) = B_{\varepsilon}(x) \times \dots \times B_{\varepsilon}(x)$ for their $m$-cartesian product. Without loss of generality, we always consider $\varepsilon$ sufficiently small for the balls $\{B_{\varepsilon}(x_j^*)\}_j$ to be pairwise disjoint.

\emph{Step 2. Condition with respect to events.}
Using property (a),
\begin{multline*}
	\E\Big( e^{- \sum_{i=1}^d \int f_i \,\ddr \crm_i } \mid \bm{X}_{1:d} = \bm{x}_{1:d}\Big) \\
	= \lim_{\varepsilon \to 0} \, \frac{\displaystyle \E \Big(e^{- \sum_{i=1}^d \int f_i \,\ddr \crm_i } \, \mathbbm{1}_{\prod_{i=1}^d \prod_{j=1}^k B_{\varepsilon}^{n_{ij}}(x_j^*)} (\bm{X}_{1:d}) \Big)}{\Prob \big(\bm{X}_{1:d} \in \prod_{i=1}^d \prod_{j=1}^k B_{\varepsilon}^{n_{ij}}(x_j^*)\big)} =: \lim_{\varepsilon \to 0} \frac{N_\varepsilon(\bm{x}^*)}{D_\varepsilon(\bm{x}^*)},
\end{multline*}
where we have introduced the notation $N_\varepsilon(\bm{x}^*)$ for the numerator and $D_\varepsilon(\bm{x}^*)$ for the denominator. In the next steps (3--6) we show that both numerator and denominator decrease at the same speed as $\varepsilon \to 0$, namely
\begin{align*}
	N_\varepsilon(\bm{x}^*) &= C_N \prod_{j=1}^k P_0 (B_{\varepsilon}(x_j^*)) + o\bigg(\prod_{j=1}^k P_0 (B_{\varepsilon}(x_j^*)) \bigg), \\
	D_\varepsilon(\bm{x}^*) &= C_D \prod_{j=1}^k P_0 (B_{\varepsilon}(x_j^*)) + o\bigg(\prod_{j=1}^k P_0 (B_{\varepsilon}(x_j^*)) \bigg),
\end{align*}
for some constants $C_N,C_D >0$, where $P_0$ is the diffuse base probability of the CRV. It then follows that the limit above coincides with $C_N/C_D$.
Since the denominator is a special case of the numerator when $f_1=\dots = f_d =0$, we focus on finding $C_N$ and then specialize the result to $C_D$.

\emph{Step 3. Express the numerator in terms of $\bm{\crm}$.} 
By the tower property (b),
\begin{align*}
	N_\varepsilon(\bm{x}^*) & = \E \bigg( e^{- \sum_{i=1}^d \int f_i \, \ddr \crm_i } \, \Prob \bigg( \bm{X}_{1:d} \in \prod_{i=1}^d \prod_{j=1}^k B_{\varepsilon}^{n_{ij}}(x_j^*) \Bigm| \bm{\crm} \bigg) \bigg) \\
	& = \E \bigg( e^{- \sum_{i=1}^d \int f_i \, \ddr \crm_i } \prod_{i=1}^d \prod_{j=1}^k \frac{\crm_i(B_{\varepsilon}(x_j^*))^{n_{ij}}}{\crm(\X)^{n_{ij}}} \bigg) \\
	& = \E \bigg( \prod_{i=1}^d \frac{1}{\crm_i(\X)^{n_i}} \, e^{-\int f_i \, \ddr \crm_i} \prod_{j=1}^k \crm_i(B_{\varepsilon}(x_j^*))^{n_{ij}} \bigg).
\end{align*}

\emph{Step 4. Use the gamma trick to separate the integrand into independent components.} 
Using the density of a gamma with shape $n_i$ and rate $\crm_i(\X)$, we rewrite
\begin{equation*}
	\frac{1}{\crm_i(\X)^{n_i}} = \frac{1}{\Gamma(n_i)} \int_0^{+\infty} u_i^{n_i-1}e^{-\crm_i(\X) u_i} \, \ddr u_i.
\end{equation*}
Henceforth, we adopt the compact notation $B_j = B_\varepsilon(x_j^*)$ and $B_0 = \X \setminus \{B_1 \sqcup \dots \sqcup B_k\}$, with the convention $n_0 = n_{i0} = 0$ for $i=1,\dots,d$. Using Fubini-Tonelli's theorem and the independence property of a CRV on disjoint set-wise evaluations,
\begin{align*}
	N_\varepsilon(\bm{x}^*) 
	& = \E \bigg( \prod_{i=1}^d \frac{1}{\Gamma(n_i)} \int_0^{+\infty} u_i^{n_i-1} \prod_{j=0}^k e^{- \crm_i(B_j) u_i} \, \ddr u_i \prod_{j=0}^k e^{-\int_{B_j} f_i \, \ddr \crm_i} \crm_i (B_j)^{n_{ij}} \bigg)  \\
	& = \frac{1}{ \prod_{i=1}^d\Gamma(n_i)} \int_{(0,+\infty)^d} \bigg(\prod_{i=1}^d u_i^{-1} \bigg) \, n_\varepsilon(u;\bm{x}^*) \, \ddr \bm{u}, 
\end{align*}
where
\begin{equation*}
	n_\varepsilon(u;\bm{x}^*) = \prod_{j=0}^k \E\bigg( \prod_{i=1}^d e^{-\int_{B_j} (f_i(x) +u_i) \, \ddr \crm_i(x)} (u_i\,\crm_i(B_j))^{n_{ij}} \bigg).
\end{equation*}
We now study the asymptotic behaviour of the quantity $n_\varepsilon(u;\bm{x}^*)$.

\emph{Step 5. Express the integrand $n_\varepsilon(\bm{x}^*)$ in terms of the derivative of the multivariate Laplace functional.} 
Let $\eta_1, \dots, \eta_d \ge 1$ be auxiliary quantities, such that $n_\varepsilon(u;\bm{x}^*)$ can be written as
\begin{align*}
	n_\varepsilon(u;\bm{x}^*) & = \prod_{j=0}^k \E\bigg( \prod_{i=1}^d \lim_{\eta_i \to 1^+} e^{-\int_{B_j} (f_i(x) + \eta_i u_i) \, \ddr \crm_i(x)} (u_i\,\crm_i(B_j))^{n_{ij}} \bigg) \\
	& = \lim_{\bm{\eta} \to 1^+} \prod_{j=0}^k \E\bigg( \prod_{i=1}^d e^{-\int_{B_j} (f_i(x) + \eta_i u_i) \, \ddr \crm_i(x)} (u_i\,\crm_i(B_j))^{n_{ij}} \bigg),
\end{align*}
where $\bm{\eta} \to 1^+$ is a compact notation for $\eta_i \to 1^+$ for each $i = 1, \dots, d$, and we have exchanged limit and expectation by monotone convergence theorem.
We observe that, for some $\ell \in \{1, \dots, d\}$,
\begin{equation*}
	\frac{\partial}{\partial \eta_\ell} \prod_{i=1}^d e^{-\int_{B_j} (f_i(x) + \eta_i u_i) \, \ddr \crm_i(x)} = -u_\ell\,\crm_\ell(B_j) \prod_{i=1}^d e^{-\int_{B_j} (f_i(x) + \eta_i u_i) \, \ddr \crm_i(x)}.
\end{equation*}
This formula can be applied recursively, using the convention $d^0/du^0 = \text{Id}$, for $\text{Id}$ the identity function. Specifically,
\begin{equation*}
	n_\varepsilon(u;\bm{x}^*) = \lim_{\bm{\eta} \to 1^+}  \prod_{j=0}^k (-1)^{n_{\bullet j}} \, \E \bigg( \frac{\partial^{n_{\bullet j}}}{\partial \eta_1^{n_{1j}}\dots \partial \eta_d^{n_{dj}}} \prod_{i=1}^d e^{-\int_{B_j} (f_i(x) + \eta_i u_i) \, \ddr \crm_i(x)} \bigg),
\end{equation*}
where $n_{\bullet j} = n_{1j} + \dots + n_{dj}$. 
Since $\eta_i \ge 1$ and $f_i \ge 0$, the derivative is bounded above by the product 
\begin{equation*}
	\prod_{i=1}^d (u_i \, \crm_i(B_j))^{n_{ij}} \exp (-u_i \crm_i(B_j)),
\end{equation*}
which has finite mean since the exponential decay of $\exp (- \sum_{i=1}^d u_i \crm_i(B_j))$ is not compromised by the slower polynomial growth of  $\prod_{i=1}^d (u_i \, \crm_i(B_j))^{n_{ij}}$.
Therefore, the derivative is uniformly integrable  in $\eta_1, \dots, \eta_d$ and we can exchange derivative and expectation. Remarkably, we need to introduce $\bm{\eta}$ and cannot derive the expressions directly with respect to $\bm{u}$: otherwise, the derivative would not be uniformly integrable for CRMs with unbounded moment measures, such as the $\sigma$-stable CRM. Using the expression \eqref{eq:laplace_exponent_crv} of the multivariate Laplace functional of a CRV, we obtain
\begin{align*}
	n_\varepsilon(u;\bm{x}^*) & = \lim_{\bm{\eta} \to 1^+} \prod_{j=0}^k (-1)^{n_{\bullet j}} \, \frac{\partial^{n_{\bullet j}}}{\partial \eta_1^{n_{1j}} \dots \partial \eta_d^{n_{dj}}} \, \E \bigg( e^{- \sum_{i=1}^d \int_{B_j} (f_i(x) + \eta_i u_i) \,\ddr \crm_i(x)} \bigg) \\
	& = \lim_{\bm{\eta} \to 1^+} \prod_{j=0}^k (-1)^{n_{\bullet j}} \, \frac{\partial ^{n_{\bullet j}}}{\partial \eta_1^{n_{1j}}\dots \partial \eta_d^{n_{dj}}} \, \exp\bigg(- \int_{\Omega_d \times B_j} \big(1-e^{- \bm{s} \cdot (\bm{f}(x) + \bm{\eta u})}\big) \, \ddr \nu(\bm{s}, x)\bigg),
\end{align*}
where we have used the notation 
$\bm{f}(x) = (f_1(x),\dots, f_d(x))$ and $\bm{\eta u} = (\eta_1 u_1, \dots, \eta_d u_d)$. For $j=0,\dots, k$, define the function
\begin{equation*}
	g_j(\bm{\eta}, \bm{u}) = - \int_{\Omega_d \times B_j} \big(1-e^{- \bm{s} \cdot (\bm{f}(x) + \bm{\eta u})} \big) \, \ddr \nu(\bm{s}, x).
\end{equation*}

\emph{Step 6. Determine the asymptotic behaviour of the partial derivatives of $e^{-g_j}$.} 
For $j=1,\dots, k$, the partial derivatives of $g_j$ satisfy, as $\varepsilon \to 0$,
\begin{multline*}
	\frac{\partial ^{n_{\bullet j}}}{\partial \eta_1^{n_{1j}} \dots \partial \eta_d^{n_{dj}}}  \,g_j(\bm{\eta}, \bm{u})= \\
	\begin{aligned}
		& = (-1)^{n_{\bullet j}} \bigg(\prod_{i=1}^d u_i^{n_{ij}} \bigg) \int_{\Omega_d \times B_j} \prod_{i=1}^d s_i^{n_{ij}} e^{- \bm{s} \cdot (\bm{f}(x) + \bm{\eta u}) } \, \ddr \nu(\bm{s}, x) \\
		& = (-1)^{n_{\bullet j}} \bigg(\prod_{i=1}^d u_i^{n_{ij}} \bigg) \int_{\Omega_d} \prod_{i=1}^d s_i^{n_{ij}} e^{- \bm{s} \cdot (\bm{f}(x_j^*) + \bm{\eta u}) } \ddr \rho_{x_j^*}(\bm{s}) \, P_0(B_j) + o(P_0(B_j)).
	\end{aligned}
\end{multline*}
The first line comes from uniform integrability of the derivative, which allows exchanging the integral and derivative operators; the second line follows from Lebesgue differentiation theorem for atomless measures on Polish spaces. 

The multivariate Fa\`a di Bruno formula \citep{ConstantineSavits1996} allows to express multiple partial derivatives of a function. We introduce the notation as in \cite{Hardy2006}, who observed that partial derivatives of the type $\partial^{n_{\bullet}}/\partial \eta_1^{n_{1}}\dots \partial \eta_d^{n_{d}}$ can be treated as \emph{maximally mixed} partial derivatives of the type $\partial^{n_{\bullet}}/ \partial v_1 \cdot \dots \partial v_{n_{\bullet}}$ by allowing for ties among the variables, which brings to more compact expressions. In particular, define  $v_1 = \dots = v_{n_{1j}} = \eta_1$, and $v_{\sum_{i'=1}^{i-1} n_{i'j}+1} = \dots =  v_{\sum_{i'=1}^{i} n_{i'j}} = \eta_i$ for $i=2,\dots,d$, so that,
\begin{equation}
	\label{eq:faa}
	(-1)^{n_{\bullet j}} \frac{\partial^{n_{\bullet j}}}{\partial \eta_1^{n_{1j}}\dots \partial \eta_d^{n_{dj}}} \, e^{-g_j(\bm{\eta}, \bm{u})} = e^{-g(\bm{\eta}, \bm{u})} \sum_{\pi} \prod _{A \in \pi} \frac{\partial^{|A|}}{ \prod_{i \in A} \partial v_i} g_j(\bm{\eta}, \bm{u}),
\end{equation}
where the sum is over all partitions $\pi$ of the numbers $\{1,\dots,n_{\bullet j} \}$. From this expression, one can derive a formula in terms of $\eta_1,\dots,\eta_d$ through appropriate combinatorial coefficients, retrieving the one in \cite{ConstantineSavits1996}. However, in our case, the combinatorial formulation above is not necessary, as we are only interested in the asymptotic behaviour as $\varepsilon \to 0$. From the previous discussion, all terms $\partial^{|A|}/ (\prod_{i \in A} \partial v_i) g(\bm{\eta}, \bm{u})$ are asymptotically equivalent up to a constant. Hence, from \eqref{eq:faa}, the asymptotically slowest term is the summand corresponding to the partition $\pi$ with a minimal number of sets, that is, $\pi = \{1,\dots, n_{\bullet j}\}$. Therefore, as $\varepsilon \to 0$, \eqref{eq:faa} reads
\begin{multline*}
	(-1)^{n_{\bullet j}} \frac{\partial^{n_{\bullet j}}}{\partial \eta_1^{n_{1j}}\dots \partial \eta_d^{n_{dj}}} \, e^{-g_j(\bm{\eta}, \bm{u})} 
	= \exp\bigg(- \int_{\Omega_d \times B_j} (1- e^{- \bm{s} \cdot (\bm{f}(x) + \bm{\eta u}) }) \,\ddr \nu(\bm{s},x) \bigg) \\
	\times \bigg(\prod_{i=1}^d u_i^{n_{ij}} \bigg) \int_{\Omega_d} \prod_{i=1}^d s_i^{n_{ij}} e^{- \bm{s}\cdot  (\bm{f}(x_j^*) + \bm{\eta u}) } \ddr \rho_{x_j^*}(\bm{s}) \, P_0(B_j) + o(P_0(B_j)).
\end{multline*}

\emph{Step 7. Determine the asymptotic behaviour of $n_\varepsilon(\bm{x}^*)$.} 
Considering the term $e^{-g_0}$, which does not vanish as $\varepsilon \to 0$, we obtain that $n_\varepsilon(u; \bm{x}^*)$ is asymptotically equal to
\begin{multline*}
	n_\varepsilon(\bm{u}; \bm{x}^*) = \exp\bigg(- \int_{\Omega_d \times \X} (1- e^{- \bm{s} \cdot (\bm{f}(x) + \bm{u}) }) \,\ddr \nu(\bm{s},x) \bigg) \\ 
	\times \bigg(\prod_{i=1}^d u_i^{n_{i}} \bigg) \prod_{j=1}^k \int_{\Omega_d} \prod_{i=1}^d s_i^{n_{ij}} e^{- \bm{s}\cdot  (\bm{f}(x_j^*) + \bm{u}) } \ddr \rho_{x_j^*}(\bm{s}) \prod_{j=1}^k P_0(B_j) + o \bigg( \prod_{j=1}^k P_0(B_j) \bigg),
\end{multline*}
where we have computed the limit for $\bm{\eta} \to 1^+$ under the integration using monotone convergence.

\emph{Step 8. Determine the asymptotic behaviour of the numerator $N_\varepsilon(\bm{x}^*)$.} 
From the relation between $n_\varepsilon(u; \bm{x}^*)$ and $N_\varepsilon(\bm{x}^*)$ in Step 4, by monotone convergence theorem, we have that $N_\varepsilon(\bm{x}^*) = C_N \prod_{j=1}^kP_0(B_j) + o( \prod_{j=1}^k P_0(B_j))$, where $C_N$ equals
\begin{multline*}
	\frac{1}{ \prod_{i=1}^d\Gamma(n_i)} \int_{(0,+\infty)^d} \bigg(\prod_{i=1}^d u_i^{n_{i}-1} \bigg) \, \exp\bigg(- \int_{\Omega_d \times \X} (1- e^{- \bm{s} \cdot (\bm{f}(x) + \bm{u}) }) \,\ddr \nu(\bm{s},x) \bigg) \\
	\times \prod_{j=1}^k \bigg( \int_{\Omega_d} \prod_{i=1}^d s_i^{n_{ij}} e^{- \bm{s}\cdot  (\bm{f}(x_j^*) + \bm{u}) } \ddr \rho_{x_j^*}(\bm{s}) \bigg) \, \ddr \bm{u}.
\end{multline*}

\newpage
\emph{Step 9. Determine the expression of the Laplace functional a posteriori.} 
By specializing the formula in Step 8 for $C_N$ to $f = 0$, we determine the value of $C_D$, which can be conveniently expressed in terms of the multivariate Laplace exponent $\psi$ and cumulants $\tau_{n_{1j},\dots, n_{dj} \mid x_j^*}$ as
\begin{equation*}
	C_D = \frac{1}{ \prod_{i=1}^d\Gamma(n_i)} \int_{(0,+\infty)^d} \bigg(\prod_{i=1}^d u_i^{n_{i}-1} \bigg) \, e^{- \psi(\bm{u})} \prod_{j=1}^k \tau_{n_{1j},\dots, n_{dj}|x_j^*}(\bm{u}) \, \ddr \bm{u}.
\end{equation*}
It follows that the Laplace functional a posteriori is equal to
\begin{multline*}
	\qquad \E\Big( e^{- \sum_{i=1}^d \int f_i \ddr \crm_i} \mid \bm{X}_{1:d} = \bm{x}_{1:d}\Big) \\
	= \frac{ \int_{(0,+\infty)^d} \big(\prod_{i=1}^d u_i^{n_{i}-1} \big) \, e^{- \int_\X \psi(\bm{f}(x) + \bm{u})\,\ddr P_0(x)} \, \ddr \bm{u}}{ \int_{(0,+\infty)^d} \big(\prod_{i=1}^d u_i^{n_{i}-1} \big) \, e^{- \psi(\bm{u})} \prod_{j=1}^k \tau_{n_{1j},\dots, n_{dj}|x_j^*}(\bm{u}) \, \ddr \bm{u}}. \qquad
\end{multline*} 

\emph{Step 10. Interpret the denominator as the normalizing constant of latent random variables $U_1,\dots, U_d$.} 
The multivariate Laplace exponent $\psi(\bm{u})$ can be retrieve at the numerator by multiplying and dividing by $e^{-\psi(\bm{u})}$. Indeed,
\begin{align*}
	\int_{\Omega_d \times \X}  (1-e^{- \bm{s} \cdot (\bm{f}(x) + \bm{u}) }) \, \ddr \nu(\bm{s}, x) + \psi(\bm{u}) & = \int_{\Omega_d \times \X}  (e^{- \bm{s} \cdot \bm{u}} - e^{-\bm{s} \cdot  (\bm{f}(x) + \bm{u}) }) \, \ddr \nu(\bm{s}, x) \\
	& = \int_{\Omega_d \times \X}  (1-e^{- \bm{s} \bm{f}(x) }) e^{-\bm{s} \cdot \bm{u}} \, \ddr \nu(\bm{s}, x).
\end{align*}
Therefore, the posterior Laplace functional is rewritten as
\begin{multline*}
	\E\Big( e^{- \sum_{i=1}^d \int f_i \ddr \crm_i} \mid \bm{X}_{1:d} = \bm{x}_{1:d}\Big) \\
	= \E \bigg(  e^{- \int_{\Omega_d \times \X}  (1-e^{- \bm{s} \bm{f}(x) }) e^{-\bm{s} \cdot \bm{U}} \ddr \nu(\bm{s}, x)} \prod_{j=1}^k \frac{ \int_{\Omega_d} \prod_{i=1}^d   s_i^{n_{ij}} e^{- \bm{s} \cdot  (\bm{f}(x_j^*) + \bm{U}) } \ddr \rho_{x_j^*}(\bm{s})}{\tau_{n_{1j},\dots, n_{dj}|x_j^*}(\bm{U})} \bigg),
\end{multline*}
where $\bm{U} = (U_1,\dots,U_d)$ is a vector of random variables with joint p.d.f.
\begin{equation*}
	f_{\scriptscriptstyle \bm{U}}(\bm{u}) \propto \prod_{i=1}^d u_i^{n_i-1} e^{- \psi(\bm{u})} \prod_{j=1}^k \tau_{n_{1j},\dots, n_{dj}|x_j^*}(\bm{u}).
\end{equation*}
This implies that there exist latent variables $\bm{U} = (U_1,\dots,U_d)$ such that the posterior Laplace functional of $\bm{\crm}$, conditionally on $\bm{U}$, satisfies
\begin{equation*}
	\E \Big( e^{- \sum_{i=1}^d \int f_i \ddr \crm_i } \mid \bm{X}_{1:d} = \bm{x}_{1:d}, \bm{U} \Big) =
	\E \Big(e^{- \sum_{i=1}^d \int f_i \ddr \crm_i^* } \mid \bm{U} \Big) \prod_{j=1}^k\E \Big(e^{-\bm{f}(x_j^*) \cdot \bm{J}_j} \mid \bm{U}\Big),
\end{equation*} 
where, conditionally on $\bm{U}$, the measure $\bm{\crm}^*$ is a CRV with L\'evy intensity $e^{-\bm{s} \cdot \bm{U}} \ddr \nu(\bm{s}, x)$ and $\bm{J}_j$ is a vector of jumps with distribution \cut
\begin{equation*}
	\ddr P_{\bm{J}_j \mid \bm{U}}(\bm{s}) \propto \prod_{i=1}^d s_i^{n_{ij}} e^{- \bm{s}\cdot \bm{U} } \ddr \rho_{x_j^*}(\bm{s}).
\end{equation*} 
If we define the random elements $\bm{\crm}^*$ and $\bm{J}_j$ to be conditionally independent given $\bm{U}$, the Laplace functional of their sum is the product of their Laplace functions, and thus
\begin{equation*}
	\E \Big( e^{- \sum_{i=1}^d \int f_i \ddr \crm_i } \mid \bm{X}_{1:d} = \bm{x}_{1:d}, \bm{U} \Big) = \E \Big( e^{ - \sum_{i=1}^d \int f_i \ddr (\crm_i^* + \sum_{j=1}^{k} J_{ij} \delta_{x_j^*})} \mid \bm{U} \Big).
\end{equation*} 
By uniqueness of the Laplace functional, this implies that 
\begin{equation*}
	\mathcal{L}( \bm{\crm} \mid \bm{X}_{1:d} = \bm{x}_{1:d}) =  \mathcal{L}\bigg(\bm{\crm}^* + \sum_{j=1}^k \bm{J}_j \delta_{x^*_j}\bigg).
\end{equation*} 

\subsection{Proof of Proposition~\ref{th:posterior_hcrv}}

We first prove a preliminary Lemma on the exponential tilting of a L\'evy measure.

\begin{lemma}
	\label{th:exponential_tilting}
	Let $\rho$ be a L\'evy measure on $(0,\infty)$ with Laplace exponent $\psi$ such that $\id(\rho)$ has a p.d.f. denoted by $f_{\scriptscriptstyle \id (\rho)}$. For $u > 0$, define $\ddr \rho_u(s) = e^{-us} \ddr \rho(s)$ the exponential tilting of $\rho$. Then for $s,t,u >0$, 
	\begin{equation*}
		e^{-us} f_{\scriptscriptstyle \id (t \rho)}(s) = e^{-t \psi(u)} f_{\scriptscriptstyle \id (t \rho_u )}(s).
	\end{equation*}
\end{lemma}

\begin{proof}
	Let $X \sim \id (t \rho)$ and let $X_u \sim \id(t \rho_u)$. By the uniqueness of the Laplace transform, it is enough to show that, for every $\lambda>0$,
	\begin{align*}
		\E\big(e^{-\lambda X_u}\big) & = \int_0^\infty e^{-\lambda s} f_{\scriptscriptstyle \id (t \rho_u )}(s) \,\ddr s \\
		& = e^{t \psi(u)} \int_0^\infty e^{-(\lambda+u)\,s} f_{\scriptscriptstyle \id (t \rho)}(s) \,\ddr s = e^{t \psi(u)} \, \E\big(e^{-(\lambda+u) X}\big).
	\end{align*}
	Indeed, the Laplace transform of $X_u$ is equal to
	\begin{align*}
		\E\big(e^{-\lambda X_u}\big) & = \int_0^\infty e^{-\lambda s} f_{\scriptscriptstyle \id (t \rho_u )}(s) \,\ddr s = e^{-\int_0^{+\infty}(1 - e^{-\lambda s}) \, t e^{-us} \, \ddr \rho(s)} \\
		& = e^{- t \int_0^{+\infty}(e^{-us} - e^{-(\lambda +u) s}) \ddr \rho(s)}
		= e^{t \psi(u) - t \psi(\lambda+u)} \\
		& = e^{t \psi(u)} \int_0^\infty e^{-(\lambda+u)\,s} f_{\scriptscriptstyle \id (t \rho)}(s) \,\ddr s = e^{t \psi(u)} \, \E\big(e^{-(\lambda+u) X}\big). \tag*{\qedhere}
	\end{align*}
\end{proof}

Considering $\bm{\crm} \sim {\rm hCRV}(\rho, \rho_0, P_0)$, the expression of the L\'evy intensity in Theorem~\ref{th:posterior_crv}(i) is given by
$$ \ddr \nu^*_{\scriptscriptstyle U}(\bm{s}, x) = e^{-\bm{U} \cdot \bm{s}} \ddr \nu_h(\bm{s},x) = e^{-\bm{U} \cdot \bm{s}} \rho_h(\bm{s}) \, \ddr \bm{s} \, \ddr P_0(x),$$
where we have used that ID$(t\rho)$ has a p.d.f.~on $(0,+\infty)$.
Exploiting Lemma~\ref{th:exponential_tilting} for the exponential tilting, we obtain
\begin{align*}
	e^{-\bm{U} \cdot \bm{s}} \rho_h(\bm{s}) & = \int_0^{+\infty} \prod_{i=1}^d \big(e^{-U_i s_i} f_{\scriptscriptstyle \id(t \rho)}(s_i)\big) \, \rho_0(t)\, \ddr t \\
	& = \int_0^{+\infty} \prod_{i=1}^d \big( e^{-t \psi(U_i)} f_{\scriptscriptstyle \id (t \rho_{U_i})}(s_i) \big) \, \rho_0(t) \, \ddr t \\
	& = \int_0^{+\infty} \prod_{i=1}^d f_{\scriptscriptstyle \id (t \rho_{U_i} )}(s_i) \,  e^{-t \sum_{i=1}^d \psi(U_i)} \rho_0(t) \, \ddr t.
\end{align*}
In analogy with Theorem~\ref{th:crv}, this expression can be interpreted as the Laplace exponent of a hierarchical CRV with heterogeneous marginal distributions, characterized by the  L\'evy measures
\begin{equation*}
	\ddr \rho_i^*(s) = \ddr \rho_{\scriptscriptstyle U_i}(s) = e^{-U_i\,s} \rho(s) \,\ddr s, \qquad \ddr \rho_0^*(t) = e^{-t \sum_{i=1}^d \psi(U_i)} \rho_0(t) \,\ddr t.
\end{equation*}

\subsection{Proof of Proposition~\ref{prop:jumps}}

Since $\rho$ and $\rho_0$ have L\'evy densities, the distribution of the vector of jumps $\bm{J}_j$ in \eqref{eq:jumps}, given the latent variables $\bm{U}$, has p.d.f.~proportional to 
\begin{align*}
	\prod_{i=1}^d s_i^{n_{ij}} \, e^{-\bm{U} \cdot \bm{s}} \rho_{h}(\bm{s}) & =  \prod_{i=1}^d s_i^{n_{ij}} e^{-U_i\,s_i} \int_0^{+\infty}  \prod_{i=1}^d f_{\scriptscriptstyle \id (t \rho)}(s_i) \, \rho_0(t) \, \ddr t \\
	& = \int_0^{+\infty} \prod_{i=1}^d s_i^{n_{ij}} e^{-U_i s_i} f_{\scriptscriptstyle \id (t \rho)}(s_i) \, \rho_0(t) \, \ddr t \\
	& = \int_0^{+\infty} \prod_{i=1}^d \frac{s_i^{n_{ij}} e^{-U_i s_i} f_{\scriptscriptstyle \id (t \rho)}(s_i)}{\bar \tau_{n_{ij}} (U_i, t)} \prod_{i=1}^d \bar \tau_{n_{ij}} (U_i, t) \, \rho_0(t) \, \ddr t,
\end{align*}
where $\bar \tau_{m} (u, t)$ is defined in \eqref{eq:tau_bar}.
Therefore, the distribution of $\bm{J}_j \mid \bm{U}$ is a mixture of conditionally independent random variables $J_{1j}, \dots, J_{dj}$ with densities
\begin{equation*}
	f_{\scriptscriptstyle J_{ij} \mid U_i, J_{0j}}(s) = \frac{s^{n_{ij}} e^{-U_i s} f_{\scriptscriptstyle \id (J_{0j} \rho)}(s)}{\bar \tau_{n_{ij}} (U_i, J_{0j})}, \qquad i = 1, \dots, d,
\end{equation*}
given a mixing random variable $J_{0j}$ having p.d.f. proportional to $\prod_{i=1}^d \bar \tau_{n_{ij}} (U_i, t) \, \rho_0(t)$.

\subsection{Proof of Lemma~\ref{prop:cumulants}}

Substituting the expression for $\rho_h$ as in the proof of Proposition~\ref{prop:jumps}, we use Fubini-Tonelli Theorem as in Remark~\ref{rem:fubini-tonelli} to obtain
\begin{align*}
	\tau_{\bm{m}}(\bm{u}) & = \int_{\Omega_d} \prod_{i=1}^d s_i^{m_i} e^{-\bm{u} \cdot \bm{s}} \rho_{h}(\bm{s}) \, \ddr {\bm s}
	= \int_0^{+\infty} \int_{\Omega_d} \prod_{i=1}^d s_i^{m_i} e^{-u_i s_i} f_{\scriptscriptstyle \id (t \rho)}(s_i) \, \ddr {\bm s} \, \rho_0(t) \, \ddr t \\
	& = \int_0^{+\infty} \prod_{i=1}^d \int_0^{+\infty} s_i^{m_i} e^{-u_i s_i} f_{\scriptscriptstyle \id (t \rho)}(s_i) \, \ddr s_i \, \rho_0(t) \, \ddr t
	= \int_0^{+\infty} \prod_{i=1}^d \bar \tau_{m_i} (u_i, t) \rho_0(t) \, \ddr t.
\end{align*}

\subsection{Proof of Proposition~\ref{prop:latent}}

Starting from the expression of the distribution of the latent variables $\bm{U}$ in \eqref{eq:latent}, specialized for $\bm{\crm} \sim {\rm hCRV}(\rho, \rho_0, P_0)$,
\begin{align*}
	f_{\scriptscriptstyle \bm U}(\bm{u}) & \propto \prod_{i=1}^d u_i^{n_i-1} e^{- \psi_0(\sum_{i=1}^d \psi(u_i))} \prod_{j=1}^k \tau_{n_{1j},\dots, n_{dj}}(\bm{u}) \\
	& = \prod_{i=1}^d u_i^{n_i-1} e^{- \psi_0(\sum_{i=1}^d \psi(u_i))} \prod_{j=1}^k \int_{0}^{+\infty} \prod_{i=1}^d \bar \tau_{m_i} (u_i, t_j) \rho_0(t_j) \, \ddr t_j,
\end{align*}
where we have substituted the expression of $\tau_{\bm{m}}(\bm{u})$ obtained in Lemma~\ref{prop:cumulants}. For each $j = 1, \dots, k$, exploiting the definition of $\bar \tau_{m} (u, t)$ in \eqref{eq:tau_bar}, we obtain
\begin{align*}
	\int_{0}^{+\infty} \prod_{i=1}^d \bar \tau_{m_i} (u_i, t_j) \rho_0(t_j) \, \ddr t_j & = \int_{0}^{+\infty} \prod_{i=1}^d \int_0^\infty s_{ij}^{n_{ij}} e^{-u_i\,s_{ij}} f_{\scriptscriptstyle {\rm ID}(t_j\rho)}(s_{ij}) \, \ddr s_{ij} \, \rho_0(t_j) \, \ddr t_j \\
	& = \int_{\Omega_d} \prod_{i=1}^d e^{-u_i\,s_{ij}} \int_{0}^{+\infty} \prod_{i=1}^d s_{ij}^{n_{ij}}  f_{\scriptscriptstyle {\rm ID}(t_j\rho)}(s_{ij}) \, \rho_0(t_j) \, \ddr t_j \, \ddr \bm{s}_{\bullet j},
\end{align*}
where we have exchanged the integrals thanks to Fubini-Tonelli's theorem. Recall that, conditionally on $\crm_0$, the distribution of $\crm_i(\X) \sim {\rm ID}(\crm_0(\X) \rho)$, while at the root of the hierarchy $\crm_0(\X) \sim {\rm ID}(\rho_0)$. Therefore, exploiting the definition of multivariate Laplace exponent,
\begin{align*}
	\prod_{i=1}^d u_i^{n_i-1} e^{- \psi_0(\sum_{i=1}^d \psi(u_i))} & = \prod_{i=1}^d u_i^{n_i-1} \E\big( e^{- {\bm u} \cdot {\bm \crm}(\X)} \big) \\
	& = \int_{\Omega_d} \prod_{i=1}^d u_i^{n_i-1} e^{- u_i y_i} \int_{0}^\infty \prod_{i=1}^d f_{\scriptscriptstyle {\rm ID}(z \rho)}(y_i) \, f_{\scriptscriptstyle {\rm ID}(\rho_0)}(z) \, \ddr z.
\end{align*}
In conclusion, the density of the vector of latent variables $\bm{U}$ is proportional to
\begin{multline*}
	f_{\bm{U}}(\bm{u}) \propto \int_{\Omega_d^{k+1}} \prod_{i=1}^d u_i^{n_i-1} e^{- u_i \, (s_{i0} + s_{i1} + \cdots + s_{ik})} \\
	\times \int_{(0,+\infty)^{k+1}} \prod_{i=1}^d f_{\scriptscriptstyle {\rm ID}(t_0 \rho)}(s_{i0}) \, f_{\scriptscriptstyle {\rm ID}(\rho_0)}(t_0) \prod_{j=1}^k s_{ij}^{n_{ij}} f_{\scriptscriptstyle {\rm ID}(t_j\rho)}(s_{ij}) \, \rho_0(t_j) \, \ddr \bm{t} \, \ddr \bm{s},
\end{multline*}
where we have renamed integration variables $y_i$ as $s_{i0}$ and $z$ as $t_0$.
Therefore, conditionally on a vector $\bm{\beta} = (\beta_1, \dots, \beta_d)$ of dependent random variables, the latent variables $U_1, \dots, U_d$ are independent and gamma distributed, with $U_i \sim {\rm Gamma}(n_i, \beta_i)$, for $i = 1, \dots, d$.
Moreover, each $\beta_i = S_{i0} + S_{i1} + \cdots + S_{ik}$, where the joint density of $\bm{S} = (S_{ij})_{ij}$ is proportional to
\begin{multline*}
	f_{\scriptscriptstyle \bm{S}}(\bm{s}) \propto \prod_{i=1}^d (s_{i0} + s_{i1} + \cdots + s_{ik})^{-n_i} \\
	\times \int_{(0,+\infty)^{k+1}} \prod_{i=1}^d f_{\scriptscriptstyle {\rm ID}(t_0 \rho)}(s_{i0}) \, f_{\scriptscriptstyle {\rm ID}(\rho_0)}(t_0) \prod_{j=1}^k s_{ij}^{n_{ij}} f_{\scriptscriptstyle {\rm ID}(t_j \rho)}(s_{ij}) \, \rho_0(t_j) \, \ddr \bm{t} \, \ddr \bm{s}.
\end{multline*}
This entails that, conditionally on $\bm{T} = (T_0, \dots, T_k)$,  $\bm{S}_i = (S_{i0}, S_{i1}, \cdots, S_{ik})$ are independent for each $i = 1, \dots, d$, with density proportional to
\begin{equation*}
	f_{\scriptscriptstyle \bm{S}_i \mid \bm{T}}(\bm{s}_i) \propto (s_{i0} + s_{i1} + \cdots + s_{ik})^{-n_i} f_{\scriptscriptstyle {\rm ID}(T_0 \rho)}(s_{i0})\, \prod_{j=1}^k s_{ij}^{n_{ij}} f_{\scriptscriptstyle {\rm ID}(T_j \rho)}(s_{ij}).
\end{equation*}
Finally, the density of vector $\bm{T}$ is proportional to
\begin{equation*}
	f_{\scriptscriptstyle \bm{T}}(\bm{t}) \propto \prod_{i=1}^d C(n_{i1}, \dots, n_{ik}; \bm{t}) \, f_{\scriptscriptstyle {\rm ID}(\rho_0)}(t_0) \prod_{j=1}^k \rho_0(t_j),
\end{equation*}
where $C$, defined in \eqref{eq:norm_constant}, 
represents the normalizing constant for the distribution of $\bm{S}_i \mid \bm{T}$.

\subsection{Proof of Proposition~\ref{prop:gamma_jumps}}

From Example~\ref{ex:gamma-gamma}, the gamma-gamma hCRV is characterized by
\begin{equation*}
	\rho(s) = \alpha \, \frac{e^{-bs}}{s}, \qquad \rho_0(s) = \alpha_0 \, \frac{e^{-b_0s}}{s},
\end{equation*}
where $\alpha,\alpha_0>0$ are shape parameters and $b,b_0>0$ rate parameters. The Laplace exponent in Definition~\ref{def:gamma} is $\psi(\lambda) = \alpha \log (1+\lambda/b)$.
The rest of the proof follows from Propositions~\ref{th:posterior_hcrv} and~\ref{prop:jumps}.

\begin{itemize}
	
	\item[(a)] From Proposition~\ref{th:posterior_hcrv}, and substituting the expressions for $\rho$, $\rho_0$ and $\psi$,
	\begin{gather*}
		e^{-U_i s} \rho(s) = \alpha s^{-1} e^{-bs-U_i s} = \alpha s^{-1} e^{-b(1+U_i/b)s}, \\
		e^{-\sum_{i=1}^d \psi(U_i)s} \rho_0(s) = \alpha_0 s^{-1} e^{-b_0s-\alpha \sum_{i=1}^d \log (1+U_i/b) s} = \alpha_0 s^{-1} e^{-\alpha \lambda(\bm{U})s},
	\end{gather*}
	where $\displaystyle \lambda(\bm{U}) = b_0/\alpha + \sum_{i=1}^d \log (1+U_i/b)$.
	
	\item[(b)] From Proposition~\ref{prop:jumps}, for each $j=1,\dots,k$, the jumps $J_{1j},\dots,J_{dj}$ are conditionally independent, given $\bm{U}$ and $J_{0j}$.
	Moreover, from the specification of $\rho$ above, the random variable $\text{ID}(t\rho)$ has gamma distribution with shape parameter $\alpha t$ and rate parameter $b$.
	Hence, for each $i=1,\dots,d$, the jump $J_{ij}$ has density proportional to
	\begin{equation*}
		s^{n_{ij}} e^{-U_is} f_{\scriptscriptstyle \text{ID}(J_{0j} \rho)}(s) \propto s^{\alpha J_{0j} + n_{ij} - 1} e^{-bs - U_i s},
	\end{equation*}
	which is the density of a gamma random variable with shape $\alpha J_{0j} + n_{ij}$ and rate $b+U_i$. The normalizing constant $\bar \tau_{n_{ij}}(U_i,t)$ in \eqref{eq:tau_bar} is
	\begin{equation*}
		\bar \tau_{n_{ij}}(U_i,t) = \int_0^{+\infty} s^{n_{ij}} e^{-U_is} f_{\scriptscriptstyle \text{ID}(t \rho)}(s) \, \ddr s = \frac{b^{\alpha t}}{(b + U_i)^{n_{ij} + \alpha t}} \frac{\Gamma(n_{ij} + \alpha t)}{\Gamma(\alpha t)}.
	\end{equation*}
	
	\item[(c)] Again from Proposition~\ref{prop:jumps}, for $j=1,\dots,k$, the density of $J_{0j}$, given $\bm{U}$, is proportional to \vspace{-.25\baselineskip}
	\begin{align*}
		f_{\scriptscriptstyle J_{0j} \mid \bm{U}}(t) & \propto \prod_{i=1}^d \bar \tau_{n_{ij}}(U_i,t) \,\rho_0(t) \propto \prod_{i=1}^d \left( \left(\frac{1}{1 + U_i/b} \right)^{\alpha t} ((\alpha t))_{n_{ij}} \right) \alpha_0 \, \frac{e^{-b_0\,t}}{t} \, \ddr t \\
		& \propto t^{-1} \, e^{-b_0 t - \alpha t \sum_{i=1}^d \log(1 + U_i/b)} \prod_{i=1}^d ((\alpha t))_{n_{ij}} \propto t^{-1} \, e^{-\alpha \lambda(\bm{U}) t} \prod_{i=1}^d ((\alpha t))_{n_{ij}},
	\end{align*}
	where $((\alpha t))_n = \Gamma(\alpha t+n)/\Gamma(\alpha t)$ denotes the ascending factorial. The result is obtained computing the density of the linear transformation $\alpha J_{0j}$.
	
\end{itemize}

\subsection{Proof of Proposition~\ref{prop:gamma_latent}}

Recall that, from the specification of $\rho$ in Example~\ref{ex:gamma-gamma}, the random variable $\text{ID}(t\rho)$ has gamma distribution with shape parameter $\alpha t$ and rate parameter $b$. The rest of the proof follows from Proposition~\ref{prop:latent}.

\begin{itemize}
	
	\item[(a)] For each $i=1,\dots,d$, the density of $\bm{S}_i = (S_{i0},\dots, S_{ik})$, given $\bm{T} = (T_0,\dots,T_k)$, is proportional to
	\begin{align*}
		f_{\scriptscriptstyle \bm{S}_i \mid {\bm T}}(s_{i0},\dots,s_{ik}) 
		& \propto (s_{i0} + s_{i1} + \dots + s_{ik})^{-n_i} f_{\scriptscriptstyle \id(T_0 \rho)}(s_{i0}) \prod_{j=1}^k s_{ij}^{n_{ij}} f_{\scriptscriptstyle \id(T_j \rho)}(s_{ij}) \\
		& \propto (s_{i0} + s_{i1} + \dots + s_{ik})^{-n_i} s_{i0}^{\alpha T_0-1} e^{-b s_{i0}} \prod_{j=1}^k s_{ij}^{\alpha T_j + n_{ij}-1} e^{-b s_{ij}} \\
		& \propto (s_{i0} + s_{i1} + \dots + s_{ik})^{-n_i} e^{-b (s_{i0} + s_{i1} + \dots + s_{ik})} s_{i0}^{\alpha T_0-1} \prod_{j=1}^k s_{ij}^{\alpha T_j + n_{ij}-1}.
	\end{align*}
	Applying the change of variables $\beta_i = S_{i0}+\cdots+S_{ik}$ and $W_{ij} = S_{ij}/\beta_i$, for $j=0,\dots,k$, the joint density of $\beta_i \ge 0$ and $\bm{W}_i = (W_{i0},\dots,W_{ik}) \in \Delta^k$, where $\Delta^k$ is the $k$-dimensional unit simplex, is
	\begin{equation*}
		f_{\scriptscriptstyle \beta_i,\bm{W}_i\mid {\bm T}}(z_i,\bm{w}_i) \propto z_i^{\alpha (T_0+T_1+\cdots+T_k)-1} e^{-b z_i} \, w_{i0}^{\alpha T_0-1} \prod_{j=1}^k w_{ij}^{\alpha T_j + n_{ij}-1}.
	\end{equation*}
	Therefore $\beta_i = S_{i0}+\cdots+S_{ik}$ is independent from $\bm{W}_i$ and has gamma distribution with shape $\alpha (T_0+\cdots+T_k)$ and rate $b$.
	Moreover, the quantity $C(\bm{m}; \bm{t})$ in \eqref{eq:norm_constant} is given by
	\begin{align*} 
		C(\bm{m}; & \bm{t}) = \int_{(0,+\infty)^{k+1}} (s_{0} + \dots + s_{k})^{-m_{\bullet}} f_{\scriptscriptstyle \id(t_0 \rho)}(s_{0}) \prod_{j=1}^k s_{j}^{m_{j}}  f_{\scriptscriptstyle \id(t_j \rho)}(s_{j}) \ddr \bm{s} \\
		& = \frac{b^{\alpha \sum_{j=0}^k t_j}}{\Gamma(\alpha t_0) \prod_{j=1}^k \Gamma(\alpha t_j)} \int_{(0,+\infty)} z^{\alpha \sum_{j=0}^k t_j -1} e^{-b z} \ddr z \, \int_{\Delta^k} w_{0}^{\alpha t_0-1} \prod_{j=1}^k w_{j}^{\alpha t_j + m_{j}-1} \ddr \bm{w} \\
		& = \frac{\Gamma(\alpha (t_0+\cdots+t_k))}{\Gamma(\alpha (t_0+\cdots+t_k) + m_{\bullet})} \prod_{j=1}^k \frac{\Gamma(\alpha t_j+m_j)}{\Gamma(\alpha t_j)}.
	\end{align*}
	
	\item[(b)] From the specification of $\rho_0$ in Example~\ref{ex:gamma-gamma}, the random variable $\text{ID}(\rho_0)$ has gamma distribution with shape parameter $\alpha_0$ and rate parameter $b_0$. Therefore, the density of $\bm{T} = (T_0,\dots,T_k)$ is proportional to
	\begin{align*}
		f_{\scriptscriptstyle {\bm T}}(\bm{t}) & \propto \prod_{i=1}^d C(n_{i1},\dots,n_{ik}; \bm{t}) f_{\scriptscriptstyle \id(\rho_0)}(t_0) \prod_{j=1}^k \rho_0(t_j) \\
		& \propto \prod_{i=1}^d \left( \frac{\Gamma(\alpha (t_0+\cdots+t_k))}{\Gamma(\alpha (t_0+\cdots+t_k) + n_i)} \prod_{j=1}^k \frac{\Gamma(\alpha t_j+n_{ij})}{\Gamma(\alpha t_j)} \right) t_0^{\alpha_0-1} \, e^{-b_0 t_0} \prod_{j=1}^k t_j^{-1} \, e^{-b_0 t_j} \\
		& \propto \prod_{i=1}^d \left( \frac{1}{((\alpha(t_0+\cdots+t_k)))_{n_i}} \right) t_0^{\alpha_0-1} \, e^{-b_0 t_0} \prod_{j=1}^k t_j^{-1} \, e^{-b_0 t_j} \left( \prod_{i=1}^d ((\alpha t_j))_{n_{ij}} \right),
	\end{align*}
	where $((\alpha t))_n = \Gamma(\alpha t+n)/\Gamma(\alpha t)$ is the ascending factorial. Since we only need the distribution of $\alpha T = \alpha(T_0 +\dots + T_k)$ to sample the random variables $\beta_1, \dots, \beta_d$, we could apply the change of variables
	\begin{equation*}
		\alpha T = \alpha(T_0 +\dots + T_k), \qquad V_{j} = T_{j}/T, \quad j = 0,\dots,k,
	\end{equation*}
	and obtain the joint density of $\alpha T$ and the vector $\bm{V} = (V_0,\dots,V_k) \in \Delta^k$ of auxiliary latent variables, supported on the $k$-dimensional unit simplex $\Delta^k$:
	\begin{equation*}
		f_{\scriptscriptstyle \alpha T,\bm{V}}(t,\bm{v}) \propto t^{\alpha_0-1} e^{-(b_0/\alpha)\,t} \, \prod_{i=1}^d \frac{1}{((t))_{n_i}} \ v_0^{\alpha_0-1} \prod_{j=1}^k \left( v_j^{-1}\prod_{i=1}^d ((t v_j))_{n_{ij}} \right).
	\end{equation*}
	Therefore, the marginal density of $\alpha T$ is
	\begin{equation*}
		f_{\scriptscriptstyle \alpha T}(t) \propto t^{\alpha_0-1} e^{-(b_0/\alpha)\,t} \, \prod_{i=1}^d \frac{1}{((t))_{n_i}} \int_{\Delta^k} v_0^{\alpha_0-1} \, \prod_{j=1}^k \left( v_j^{-1}\prod_{i=1}^d ((t v_j))_{n_{ij}} \right) \ddr \bm{v}.
	\end{equation*}
	Note that this function is integrable in $t$; indeed, it holds that $((s))_q \sim s$ for $s \to 0$ when $q$ is a positive integer. Therefore, for $t \to 0$,
	\begin{align*}
		f_{\scriptscriptstyle \alpha T}(t) & \sim t^{\alpha_0-1} \prod_{i=1}^d \frac{1}{((t))_{n_i}} \int_{\Delta^k} v_0^{\alpha_0-1} \, \prod_{j=1}^k \left( v_j^{-1}\prod_{i=1}^d ((t v_j))_{n_{ij}} \right) \ddr \bm{v} \\
		& \sim t^{\alpha_0-1} \prod_{i=1}^d \left( t^{-1} \prod_{j=1}^k t^{m_{ij}} \right) \sim t^{\alpha_0-1} \prod_{i=1}^d t^{m_{i \bullet}-1} \sim t^{\alpha_0 + m-d-1},
	\end{align*}
	and since we are assuming $n_i > 0$, then $m_{i \bullet} \ge 1$ and therefore $m \ge d$. See also the proof of Proposition~\ref{prop:gamma_baselatent_exact}.
	
\end{itemize}

\subsection{Proof of Proposition~\ref{prop:gamma_basejumps_exact}}

Recall that the ascending factorial $((s))_q$, for integer $q$, can be written as
\begin{equation} \label{eq:ascending_factorial}
	((s))_{q} = \sum_{h=0}^{q} S(q,h)\,s^{h},
\end{equation}
where $S(q,h)$ are the unsigned Stirling numbers of the first kind, defined through the recursive relation $S(q+1,h) = q\,S(q,h) + S(q,h-1)$, with initial conditions $S(0,0) = 1$ and $S(q,0) = S(0,h) = 0$ for $q > 0$ or $h > 0$. 
Since $S(q,0) = 0$ whenever $q>0$, the summation above can start from $1$ if $q$ is strictly positive.

From Proposition~\ref{prop:gamma_jumps}, the density of $\alpha J_{0j}$, for each $j = 1, \dots, k$, can be rewritten as
\begin{align*}
	f_{\scriptscriptstyle \alpha J_{0j} \mid \bm{U}}(t) & \propto t^{-1} \, e^{-\lambda(\bm{U})\,t} \prod_{i=1}^d \left( \sum_{h_{ij}=m_{ij}}^{n_{ij}} S(n_{ij},h_{ij})\,t^{h_{ij}} \right) \\
	& \propto t^{-1} \, e^{-\lambda(\bm{U})\,t} \sum_{h_j=m_{\bullet j}}^{n_{\bullet j}} \sum_{\substack{h_{1j}+\dots+h_{dj} = h_j \\ m_{ij} \le h_{ij} \le n_{ij}}} \prod_{i=1}^d S(n_{ij},h_{ij}) \, t^{h_j} \\
	& \propto \sum_{h_j=m_{\bullet j}}^{n_{\bullet j}} \Bigg( \sum_{\substack{h_{1j}+\dots+h_{dj} = h_j \\ m_{ij} \le h_{ij} \le n_{ij}}} \prod_{i=1}^d S(n_{ij},h_{ij}) \Bigg) \, t^{h_j-1} e^{-\lambda(\bm{U})\,t}
\end{align*}
where $m_{ij} \in \{0,1\}$ is the indicator for $n_{ij} > 0$, that is, $m_{ij} = \min(1,n_{ij})$, and $m_{\bullet j} = \sum_{i=1}^d m_{ij}$. 
Note that since $n_{\bullet j} > 0$, then $m_{\bullet j} > 0$, and the density of $\alpha J_{0j}$ is properly defined.
Moreover, for each $j = 1, \dots, k$, define the coefficients $S(n_{1j},\dots,n_{dj}; h_j)$ where
\begin{equation} \label{eq:multivariate_stirling_supp}
	S(q_{1},\dots,q_{d}; h) = \sum_{\substack{h_{1}+\dots+h_{d} = h \\ 0 \le h_{i} \le q_{i}}} \prod_{i=1}^d S(q_{i},h_{i}).
\end{equation}
Exploiting the recursive relation for unsigned Stirling numbers of the first kind, we obtain 
\begin{align*}
	S(q_{1},\dots, & q_{\ell}+1, \dots,q_{d}; h) = \\
	& = \sum_{\substack{h_{1}+\dots+h_{d} = h \\ 0 \le h_{i} \le q_{i} \\ 0 \le h_\ell \le q_\ell+1}} S(q_\ell+1,h_\ell) \prod_{\substack{i=1 \\ i \neq \ell}}^d S(q_{i},h_{i}) \\
	& = \sum_{\substack{h_{1}+\dots+h_{d} = h \\ 0 \le h_{i} \le q_{i} \\ 0 \le h_\ell \le q_\ell+1}} q_\ell S(q_\ell,h_\ell) \prod_{\substack{i=1 \\ i \neq \ell}}^d S(q_{i},h_{i}) 
	+ \sum_{\substack{h_{1}+\dots+h_{d} = h \\ 0 \le h_{i} \le q_{i} \\ 0 \le h_\ell \le q_\ell+1}} S(q_\ell,h_\ell-1) \prod_{\substack{i=1 \\ i \neq \ell}}^d S(q_{i},h_{i}) \\
	& = q_\ell \sum_{\substack{h_{1}+\dots+h_{d} = h \\ 0 \le h_{i} \le q_{i} \\ 0 \le h_\ell \le q_\ell}}  S(q_\ell,h_\ell) \prod_{\substack{i=1 \\ i \neq \ell}}^d S(q_{i},h_{i}) 
	+ \sum_{\substack{h_{1}+\dots+h_{d} = h-1 \\ 0 \le h_{i} \le q_{i} \\ 0 \le h_\ell \le q_\ell}} S(q_\ell,h_\ell) \prod_{\substack{i=1 \\ i \neq \ell}}^d S(q_{i},h_{i}) \\
	& = q_\ell \,S(q_{1},\dots,q_{d}; h) + S(q_{1},\dots,q_{d}; h-1).
\end{align*}
Notably, in the third line we have used the fact that $S(q_\ell,q_\ell+1) = 0$ (first term) and applied the change of index $h_\ell \mapsto h_\ell+1$ (second term). Therefore, the coefficients above satisfy the recurrence relation in \eqref{eq:multivariate_stirling} defining the multivariate Stirling numbers.

\begin{remark} \label{remark:computation_stirling}
	A naive approach to the evaluation of the coefficients $S(q_{1},\dots,q_{d}; h)$ from the relation \eqref{eq:multivariate_stirling_supp} would involve $d$ nested cycles, with a computational cost of $\mathcal{O}\big( \prod_{i=1}^dn_{ij}\big)$, for each $j = 1, \dots, k$. However, recognizing the recursive structure \eqref{eq:multivariate_stirling} substantially reduces the 
	cost to quadratic in the number of observations $n_{\bullet j}$.
\end{remark}

\subsection{Proof of Proposition~\ref{prop:gamma_baselatent_exact}}

The density of $\alpha T$ in Proposition~\ref{prop:gamma_latent} can be rewritten using the equality in \eqref{eq:ascending_factorial} and the definition of multivariate Stirling numbers in \eqref{eq:multivariate_stirling_supp}:
\begin{align*}
	f_{\scriptscriptstyle \alpha T}&(t) \propto \\
	& \propto t^{\alpha_0-1} e^{-(b_0/\alpha)\,t} \, \prod_{i=1}^d \frac{1}{((t))_{n_i}} \int_{\Delta^k} \hspace{-.25em} v_0^{\alpha_0-1} \prod_{j=1}^k \left( v_j^{-1}\prod_{i=1}^d \left( \sum_{h_{ij}=m_{ij}}^{n_{ij}} S(n_{ij},h_{ij})\,v_j^{h_{ij}} t^{h_{ij}} \right) \right) \ddr \bm{v} \\
	& \propto t^{\alpha_0-1} e^{-(b_0/\alpha)\,t} \,\prod_{i=1}^d \frac{1}{((t))_{n_i}} \int_{\Delta^k} \hspace{-.25em} v_0^{\alpha_0-1} \prod_{j=1}^k \left( \sum_{h_j=m_{\bullet j}}^{n_{\bullet j}} S(\bm{n}_j; h_j) \,v_j^{h_j-1} t^{h_j} \right) \ddr \bm{v} \\
	& \propto t^{\alpha_0-1} e^{-(b_0/\alpha)\,t} \, \prod_{i=1}^d \frac{1}{((t))_{n_i}} \int_{\Delta^k} \hspace{-.25em} v_0^{\alpha_0-1} \sum_{h=m}^{n} \Bigg(\sum_{\substack{h_{1}+\dots+h_{k} = h \\ m_{\bullet j} \le h_{j} \le n_{\bullet j}}} \prod_{j=1}^k S(\bm{n}_j; h_j) \, v_j^{h_j-1} t^{h_j} \Bigg)\, \ddr \bm{v},
\end{align*}
where $\bm{n}_j = (n_{1j},\dots,n_{dj})$.
The density of $\alpha T$ is then obtained computing the integral over the $k$-dimensional simplex:
\begin{align*}
	f_{\scriptscriptstyle \alpha T}&(t) \propto \\
	& \propto t^{\alpha_0-1} e^{-(b_0/\alpha)\,t} \prod_{i=1}^d \frac{1}{((t))_{n_i}} \Bigg(\sum_{h=m}^{n} t^h \hspace{-.25em} \sum_{\substack{h_{1}+\dots+h_{k} = h \\ m_{\bullet j} \le h_{j} \le n_{\bullet j}}} \prod_{j=1}^k S(\bm{n}_j; h_j) \int_{\Delta^k} \hspace{-.25em} v_0^{\alpha_0-1} \prod_{j=1}^k v_j^{h_j-1} \, \ddr \bm{v}\Bigg) \\
	& \propto t^{\alpha_0-1} e^{-(b_0/\alpha)\,t} \prod_{i=1}^d \frac{1}{((t))_{n_i}} \Bigg(\sum_{h=m}^{n} \frac{t^h}{((\alpha_0))_h} \sum_{\substack{h_{1}+\dots+h_{k} = h \\ m_{\bullet j} \le h_{j} \le n_{\bullet j}}} \prod_{j=1}^k \Gamma(h_j)\,S(\bm{n}_j; h_j)\Bigg) \\
	& \propto t^{\alpha_0-1} e^{-(b_0/\alpha)\,t} \prod_{i=1}^d \frac{1}{((t))_{n_i}} \left(\sum_{h=m}^{n} \frac{c_h }{((\alpha_0))_h}\, t^{h} \right),
\end{align*}
where the coefficients $c_h$ for $h = m,\dots,n$ are defined as
\begin{equation*}
	c_h = \sum_{\substack{h_{1}+\dots+h_{k} = h \\ m_{\bullet j} \le h_{j} \le n_{\bullet j}}} \prod_{j=1}^k \Gamma(h_j)\,S(n_{1j},\dots,n_{dj}; h_j).
\end{equation*}


\section{Posterior sampling for the gamma-gamma hCRV}
\label{app:sampling}

This section provides additional details on the practical implementation of posterior sampling algorithms for the normalized gamma-gamma hCRV described in Section~\ref{sec:gamma-gamma}, and contains numerical illustrations supporting their effectiveness. 
Section~\ref{app:mh_sampling} analyzes the Metropolis-Hastings steps of Algorithm~\ref{alg:mcmc}. Specifically, a random walk Metropolis-Hastings scheme on the log-scale with Gaussian increments is proposed and compared with the approach of \cite{Barrios2013}, based on gamma proposals. The initialization of Algorithm~\ref{alg:exact} is detailed in Sections~\ref{app:coefficients_exact} and~\ref{app:ropt}, which respectively discuss the computation of coefficients $c_h$'s in Proposition~\ref{prop:gamma_baselatent_exact} and the optimization of parameter $r$ within the rejection sampling scheme. 
Section~\ref{app:sampling_random_measure} describes the inverse L\'evy measure algorithm to obtain a truncated sample from the hierarchy of gamma CRMs in Proposition~\ref{prop:gamma_jumps}(a). For this purpose, an efficient procedure to sequentially invert the exponential integral function is outlined in Section~\ref{app:expint}. 
This same simulation strategy remains valid \emph{a priori} to sample from gamma-gamma hCRVs.
Section~\ref{app:normalized_jumps} characterizes the posterior random probabilities $\tilde p_i = \crm_i / \crm_i(\X) \mid \bm{X}_{1:d}$ as conditionally normalized gamma CRMs, i.e.~conditionally Dirichlet processes, and thus provides an alternative approach to directly sample their posterior normalized jumps. This allows for a straightforward comparison with standard samplers for the HDP with gamma prior on the concentration parameter, which target the same posterior distributions (Proposition~\ref{th:hdp_hyper}).
The marginal Gibbs sampler of \cite{Teh2006}, based on the restaurant franchise metaphor, is tailored to our setting in Section~\ref{app:tables}. This section also proposes an alternative collapsed Gibbs sampler for the HDP, which directly samples the number of tables serving each dish, and may be of independent interest. Numerical illustrations of the effectiveness of the proposed algorithms in terms of mixing properties and posterior accuracy are provided in Section~\ref{app:illustrations}.
Finally, Section~\ref{app:timesk} extends the simulation study of Section~\ref{sec:simulations} comparing the different algorithms in terms of execution time per effective sample, as the number $k$ of distinct values increases.

\subsection{Metropolis-Hastings steps in Algorithm~\ref{alg:mcmc}}
\label{app:mh_sampling}

The non-standard steps in the posterior sampling algorithms of Section~\ref{sec:algorithms_gamma} 
are:
\begin{itemize}
	\item[(i)] the marginal sampling of random variable $\alpha T$ in \eqref{eq:joint_baselatent_simplex}, whose joint density with the auxiliary vector $\bm{V}$ is known up to a normalizing constant;
	\item[(ii)] the sampling of random variables $\alpha J_{01}, \dots, \alpha J_{0k}$ in \eqref{eq:basejump}, whose densities are again known up to normalizing constants.
\end{itemize}
A natural approach to obtain samples from non-standard densities is resorting to MCMC schemes. In Algorithm~\ref{alg:mcmc}, we consider a blocked Gibbs sampler with Metropolis-Hastings steps. The full conditional distributions are derived from Propositions~\ref{prop:gamma_jumps} and~\ref{prop:gamma_latent}, namely 
\begin{align*}
	(V_j,V_\ell) \mid \bm{V}^{-j,\ell}, (\alpha T) = t & \, \sim \, v_j^{-1} \,v_\ell^{-1} \prod_{i=1}^d ((t v_j))_{n_{ij}} ((t v_\ell))_{n_{i\ell}} \qquad (j, \ell = 1, \dots, k), \\
	(V_j,V_0) \mid \bm{V}^{-j,0}, (\alpha T) = t & \, \sim \, v_0^{\alpha_0-1} \,v_j^{-1}\prod_{i=1}^d ((t v_j))_{n_{ij}} \qquad (j = 1, \dots, k), \\
	(\alpha T) \mid \bm{V} & \, \sim \, t^{\alpha_0-1} \, e^{-(b_0/\alpha)\,t} \prod_{i=1}^d \frac{1}{((t))_{n_i}} \prod_{j=1}^k \prod_{i=1}^d ((t v_j))_{n_{ij}}, \\
	(\alpha J_{0j}) \mid \bm{U} & \, \sim \, t^{-1} \, e^{-\lambda(\bm{U})\,t} \,\prod_{i=1}^d ((t))_{n_{ij}} \qquad (j = 1, \dots, k),
\end{align*}
where $\bm{V}^{-j,\ell}$ denotes the vector $\bm{V} \in \Delta^k$ with components $V_j$ and $V_\ell$ removed. We note that variables in $\bm{V} \in \Delta^k$ are subject to the constraint $\sum_{j=0}^k V_j = 1$, and one cannot update variables independently;
on the other hand, variables $\alpha T$ and $\alpha J_{0j}$ can take every positive value. Throughout the Gibbs sampling procedure, we propose new values for each variable, and accept or reject the proposal according to the Metropolis-Hastings ratio. The rest of the section outlines some of the possible proposals, focusing on rather straightforward and widely known options. Clearly, more sophisticated alternatives can be considered without undermining the validity of our results.

A simple and effective symmetric proposal for the pair of auxiliary variables $(V_j,V_\ell)$, for each $j,\ell = 0, \dots, k$, is
\begin{equation*}
	v_{j}^* = \varepsilon \, (v_j + v_\ell), \qquad v_{\ell}^* = (1-\varepsilon) \, (v_j + v_\ell),
\end{equation*}
where $v_j$ and $v_\ell$ are the current values of $V_j$ and $V_\ell$ and $\varepsilon \sim \mathcal{U}(0,1)$; see e.g.~\cite{smith2014}. The proposal is accepted with log-probability
\begin{equation*}
	\log(r) = \min \left\{0, \ q_j(v^*_j) - q_j(v_j) + q_\ell(v^*_\ell) - q_\ell(v_\ell) \right\},
\end{equation*}
where $q_0(v) = (\alpha_0 - 1) \, v$, and $q_j(v) = \sum_{i=1}^d \log((t v))_{n_{ij}} - v$, for $j = 1, \dots, k$. At each iteration of the Gibbs sampler, we randomly select $k$ pairs of indexes $(j,\ell)$ and implement the Metropolis-Hastings step detailed above. Note that selecting $k$ pairs is a good compromise between choosing a single pair and scanning each of the $k(k-1)/2$ possible pairs.

For the positive variables $\alpha T$ and $\alpha J_{0j}$'s, we consider random walks with two alternative proposals, which we first describe for a generic positive random variable with density $f(x)$.

\begin{itemize}
	\item[(a)] \emph{Gamma proposal}. Following the approach discussed in \cite{Barrios2013}, we propose a new value $x^*$ from a gamma distribution centered at the current value $x$, 
	\begin{equation*}
		x^* \sim \text{Gamma}(\delta, \delta / x),
	\end{equation*}
	where $\delta > 0$ controls the variance of the proposal. The proposed value $x^*$ is accepted with log-probability
	\begin{align*}
		\log(r) & = \min \big\{0, \, \log f(x^*) - \log f(x) + \log g(x; \delta, \delta/x^*) - \log g(x^*; \delta, \delta/x) \big\} \\
		& = \min \left\{ 0, \, \log f(x^*) - \log f(x) + (2\delta-1)(\log x - \log x^*) + \delta \left(\frac{x^*}{x} - \frac{x}{x^*} \right) \right\},
	\end{align*}
	where $x \mapsto \log g(x; a, b)$ is the log-density of a gamma distribution with shape $a > 0$ and rate $b > 0$. For the practical implementation, \cite{Barrios2013} suggest to restrict to $\delta \ge 1$.
	\item[(b)] \emph{Random walk on log-transform}. For a positive random variable, we can target the density of its log-transform, which is $t \mapsto f(e^t)\,e^t$. In this case, we resort to a random walk on the log-scale, and propose a new value $\log x^*$ from a normal distribution centered at the current value $\log x$,
	\begin{equation*}
		\log x^* \sim \mathcal{N}(\log x, \sigma^2),
	\end{equation*}
	where $\sigma^2 > 0$ controls the variance of the proposal. The proposed value $x^*$ is accepted with log-probability
	\begin{equation*}
		\log(r) = \min \big\{0, \, \log f(x^*) - \log f(x) + \log x^* - \log x \big\}.
	\end{equation*}
\end{itemize}

Interestingly, both proposals can be interpreted as multiplicative perturbations of the current value, with different distributions for the perturbation factor. Indeed, they both propose a new value $x^* = x\,\varepsilon$, where $x$ is the current value and $\varepsilon$ is distributed as either (a) a gamma random variable $\varepsilon \sim \text{Gamma}(\delta, \delta)$ with variance $1/\delta$, or (b) a log-normal random variable $\log \varepsilon \sim \mathcal{N}(0, \sigma^2)$ with variance $\sigma^2$ on the log-scale. Therefore, the acceptance rates can be rewritten as
\begin{align*}
	& {\rm (a)} \qquad \log(r) = \min \left\{ 0, \, \log g(x\,\varepsilon) - \log g(x) + \delta (\varepsilon - 1/\varepsilon) - 2\delta\, \log \varepsilon \right\}, \\
	& {\rm (b)} \qquad \log(r) = \min \left\{ 0, \, \log g(x\,\varepsilon) - \log g(x) \right\},
\end{align*}
where $\log g(x) = \log f(x) + \log x$.
The practical implementation of these approaches requires the evaluation of the logarithm of target density at the current and proposed values, up to normalizing constants. For our purposes, we have
\begin{gather*}
	\log f_{ \scriptscriptstyle \alpha T \mid \bm{V}}(t) = (\alpha_0-1) \log t - (b_0/\alpha)\,t + \sum_{j=1}^k \sum_{i=1}^d \log((tv_j))_{n_{ij}} - \sum_{i=1}^d \log((t))_{n_i}, \\
	\log f_{\scriptscriptstyle \alpha J_{0j} \mid \bm{U}}(t) = \sum_{i=1}^d \log((t))_{n_{ij}} - \lambda(\bm{U}) \, t - \log t \qquad (j = 1, \dots, k).
\end{gather*}
Note that the random variables $\alpha J_{0j}$'s are gamma distributed whenever $n_{1j}, \dots, n_{dj} \le 1$, and can thus be sampled exactly as
\begin{equation} \label{eq:exact_jumps}
	\alpha J_{0j} \mid \bm{U} \sim \text{Gamma}\left( \sum_{i=1}^d n_{ij}, \, \lambda(\bm{U}) \right).
\end{equation}
Integrating this feature into the algorithm leads to a substantial reduction in computational times in many scenarios. A crucial aspect of random walk Metropolis-Hastings algorithms is the choice of a suitable variance parameter for the proposal, in order for the resulting Markov chain to mix efficiently \citep{roberts2001}. For this purpose, we resort to adaptive MCMC techniques, which allow to automatically tune the proposals to optimize their performances. Specifically, we implement a simple but effective updating rule, based on the Robbins-Monro recursion \cite[see e.g.][]{andrieu2008}. The explicit updates for the two proposals discussed above are
\begin{equation*}
	{\rm (a)} \quad \log \delta_{s+1} \leftarrow \log \delta_s - \gamma_s (r_s - r^*), \qquad 
	{\rm (b)} \quad \log \sigma^2_{s+1} \leftarrow \log \sigma^2_s + \gamma_s (r_s - r^*),
\end{equation*}
where $\delta_s$ and $\sigma^2_s$ are the parameters of the proposals at step $s$, $r_s$ is the probability of acceptance at step $s$, $r^* = 0.44$ is the target acceptance rate \cite[see][]{roberts1996} and $(\gamma_s)_{s \ge 1}$ is a sequence of positive and non-increasing numbers. In our implementation, we employ the deterministic sequence $\gamma_s = (10 + s)^{-1/2}$, for $s \ge 1$; see \cite{andrieu2008} for a discussion on the desirable properties of such sequence. This adaptive procedure is employed at each iteration during the burn-in phase, and the values of variance parameters reached at the end of the burn-in are kept fixed for the actual posterior sampling. Our implementation allows for different variance parameters for $\alpha T$ and each of the $\alpha J_{0j}$'s.

We conclude this section with a comparison of the gamma and log-normal proposals, carried out through a simple numerical experiment. We consider $1,000$ synthetic datasets, each consisting of a random number of groups $d \sim \text{Poisson}(5)$, with $n_i = 100$ observations per group. Observations are sampled from a hierarchical Dirichlet process with random concentration parameters $\alpha \sim \text{Gamma}(5)$ and $\alpha_0 \sim \text{Gamma}(3)$. We draw $10,000$ posterior samples for each dataset, with a burn-in of $1,000$ steps, during which variance parameters are tuned.
Table~\ref{apptable:ess} reports the effective sample size (ESS) and acceptance rate for each target variable and proposal. Figures for the $\alpha J_{0j}$'s are the means computed over $j$, excluding those indexes for which the exact sampling in \eqref{eq:exact_jumps} is available.
The log-normal proposal outperforms the gamma proposal in terms of ESS for both $\alpha T$ and the $\alpha J_{0j}$'s; the difference is particularly evident for the latter. Moreover, the adaptive MCMC procedure systematically reaches the target acceptance rate of $0.44$. Clearly, the evidence provided by this experiment is not conclusive in any respect, as we are restricting ourselves to a particular algorithm and sampling scenario. However, the synthetic datasets cover a wide range of possible count matrices $(n_{ij})_{ij}$ that one may encounter in applications. Therefore, we recommend using the log-normal proposal, which we also consider for the simulations in Section~\ref{sec:simulations} of the main manuscript.

\begin{table}[]
	\centering
	\begin{tabular}{llcc}
		variable(s) & proposal & ESS & accept.~rate \\
		\hline
		\multirow{2}{6em}{latent $\alpha T$} & gamma & 992 (432) & 0.440 (0.022) \\
		& log-normal & 1114 (479) & 0.440 (0.021) \\
		\hline
		\multirow{2}{6em}{jumps $\alpha J_{0j}$'s} & gamma & 966 (225) & 0.440 (0.009) \\
		& log-normal & 1779 (232) & 0.440 (0.008) \\
		\hline
	\end{tabular}
	\captionsetup{width=0.9\textwidth,font=small}
	\caption{\small Effective sample size (ESS) and acceptance rate for each target variable and proposal. Results are averaged over $1,000$ synthetic datasets; standard deviations are reported in parentheses. After $1,000$ burn-in steps, $10,000$ posterior samples are drawn for each dataset. The target acceptance rate for the adaptive MCMC scheme is $0.44$ \citep[see][]{roberts1996}.}
	\label{apptable:ess}
\end{table}

\subsection{Computing the coefficients in Proposition~\ref{prop:gamma_baselatent_exact}}
\label{app:coefficients_exact}

In this section, we show that the coefficients $c_h$, for $h = m, \dots, n$, defined in Proposition~\ref{prop:gamma_baselatent_exact}, can be computed through a sequence of discrete convolutions. For each $\ell = 1, \dots, k$, define the vector $\bm{c}(\ell) = (c(\ell,h))_h$ with
\begin{equation*}
	c(\ell; h) = \sum_{\substack{h_1 + \cdots + h_\ell = h \\ m_{\bullet j} \le h_j \le n_{\bullet j}}} \prod_{j=1}^\ell a(j; h_j), \qquad h = \sum_{j=1}^\ell m_{\bullet j}, \dots, \sum_{j=1}^\ell n_{\bullet j},
\end{equation*}
where $a(j; h) = \Gamma(h) \, S(n_{1j},\dots,n_{dj}; h)$, for $h = m_{\bullet j}, \dots, n_{\bullet j}$ and $j = 1, \dots, \ell$.
Thus, the $c_h$'s are such that $c_h = c(k; h)$, that is, coincide with $\bm{c}(k)$. The entries of vector $\bm{c}(\ell)$ can be computed from vector $\bm{c}(\ell-1)$ as follows:
\begin{align*}
	c(\ell; h) & = \sum_{\substack{h_1 + \cdots + h_\ell = h \\ m_{\bullet j} \le h_j \le n_{\bullet j}}} \prod_{j=1}^{\ell} a(j; h_j) = \sum_{h_{\ell} = m_{\bullet \ell}}^{n_{\bullet \ell}} \left( \sum_{\substack{h_1 + \cdots + h_{\ell-1} = h - h_{\ell} \\ m_{\bullet j} \le h_j \le n_{\bullet j}}} \prod_{j=1}^{\ell-1} a(j; h_j) \right) a(\ell; h_\ell) \\
	& = \sum_{h_{\ell} = m_{\bullet \ell}}^{n_{\bullet \ell}} c(\ell-1; h-h_\ell) \, a(\ell; h_\ell).
\end{align*}
In other words, the vector $\bm{c}(\ell)$ is obtained by convolution between vector $\bm{c}(\ell-1)$ and vector $\bm{a}(\ell) = (a(\ell;h),h)$.
The computational cost for this operation is $\mathcal{O}\left( n_{\bullet \ell} \sum_{j=1}^{\ell-1} n_{\bullet j} \right) $.
Therefore, the coefficients $(c_h)_h = \bm{c}(k)$ can be computed through the recursive relation
\begin{equation*}
	\bm{c}(0) = (1), \qquad \bm{c}(\ell) = \bm{c}(\ell-1) * \bm{a}(\ell) \quad (\ell = 1, \dots, k),
\end{equation*}
or, equivalently, $\bm{c}(k) = \bm{a}(1) * \cdots * \bm{a}(k)$, where $*$ is the convolution operator. The total computational cost is $\mathcal{O}\left( \sum_{j<\ell} n_{\bullet j} n_{\bullet \ell} \right)$.

\subsection{Optimal choice of parameter $r$}
\label{app:ropt}

Section~\ref{sec:exact_sampling_gamma} outlines a rejection sampling algorithm for sampling the latent variable $\alpha T$ from its density \eqref{eq:density_baselatent}. Specifically, values are proposed from ${\rm Gamma}(\alpha_0 + r, b_0 / \alpha)$ and accepted with probability proportional to $t^{-r} R(t)$, where $r$ is a real parameter and
\begin{equation*}
	R(t) = \prod_{i=1}^d \frac{1}{((t))_{n_i}} \left( \sum_{h=m}^n \frac{c_h}{((\alpha_0))_h} \, t^h \right).
\end{equation*}
A necessary condition for the rejection sampling scheme is $t^{-r} R(t)$ to be bounded above for $t \ge 0$. Since $R(t)$ is a ratio of polynomials, both having degree $n$ and non-negative coefficients, it is a continuous function for $t > 0$. Moreover, when $q$ is a positive integer, $((s))_q \sim s$ for $s \to 0$ while $((s))_q \sim s^q$ for $s \to \infty$. Hence, $R(t) \sim t^{m-d}$ for $t \to 0$ and $R(t) \sim c_n / ((\alpha_0))_n$ for $t \to \infty$, which implies that $t^{-r} R(t)$ is continuous and bounded for $t \ge 0$ when $0 \le r \le m-d$. Our goal is choosing the value of $r$ within this interval such that the acceptance probability is maximized.

For this purpose, let $t^*(r)$ be the value of $t$ that maximizes $t^{-r} R(t)$ for $t \ge 0$. The overall acceptance probability is
\begin{align*}
	& \mathbb{E}_T \left[ \left(\frac{t^*(r)}{t}\right)^r \frac{R(t)}{R(t^*(r))} \right] =\\
	& = \frac{t^*(r)^r}{R(t^*(r))} \int_0^\infty t^{-r} R(t) \left( \frac{b_0}{\alpha} \right)^{\alpha_0+r} \Gamma(\alpha_0+r)^{-1} t^{\alpha_0+r-1} e^{-(b_0/\alpha)\,t} \ddr t \\
	& = \frac{t^*(r)^r}{R(t^*(r))} \left( \frac{b_0}{\alpha} \right)^{\alpha_0+r} \Gamma(\alpha_0+r)^{-1} \int_0^\infty t^{\alpha_0-1} e^{-(b_0/\alpha)\,t} R(t) \, \ddr t.
\end{align*}
Finding the value of $0 \le r \le m-d$ that maximizes the quantity above is equivalent to finding the value maximizing its logarithm, discarding terms not depending on $r$, that is
\begin{equation*}
	r^* = \arg\max_r \left\{ r \log t^*(r) - \log R(t^*(r)) + (\alpha_0 + r) \log (b_0/\alpha) - \log \Gamma(\alpha_0 + r) \right\}.
\end{equation*}
This maximization problem can be further simplified by restricting to a finite set of potentially maximizing values. 
Indeed, for $0 < r < m-d$, the function $R(t)$ is continuous and differentiable for $t > 0$, and such that $t^{-r} R(t) \to 0$ for both $t \to 0$ and $t \to \infty$. Hence, $t^*(r)$ is a stationary point for $t^{-r} R(t)$, which implies $R'(t^*(r)) \, t^*(r) = r\,R(t^*(r))$. 
Moreover, by the implicit function theorem, $t^*(r)$ is a continuous and differentiable function in $r$. Therefore, the objective function is continuous and differentiable for $0 < r < m-d$, and the set of potentially maximizing points in $(0,m-d)$ may be restricted to the stationary points (if any), satisfying
\begin{equation*}
	\log t^*(r) + \log (b_0/\alpha) - \Psi(\alpha_0 + r) = 0,
\end{equation*}
where $\Psi$ denotes here the digamma function $\Psi(x) = \frac{\ddr}{\ddr x} \log \Gamma(x)$, and we have used that $t^*(r)$ is a stationary point for $t^{-r} R(t)$ and thus $R'(t^*(r)) \, t^*(r) = r\,R(t^*(r))$.
Remarkably, this stationarity condition is also satisfied at the boundary of the maximization set. Indeed, for $r = 0$, the optimal $t^*(r)$ may be either a stationary point or $+\infty$, for which $R'(t^*(r)) = 0$ necessarily holds. On the other hand, if $r = m-d$, the optimal $t^*(r)$ may be either a stationary point or $0$, for which $R(t^*(r)) = 0$ holds.

\subsection{Sampling from the hierarchy of gamma CRMs}
\label{app:sampling_random_measure}

From Proposition~\ref{prop:gamma_jumps}(a), the residual component $\bm{\crm}^*$ of the posterior distribution of $\bm{\crm}$ retains a hierarchical structure, conditionally on latent variables $\bm{U}$:
\begin{align*}
	\crm_1^*,\dots,\crm_d^* \mid \crm_0^*, \bm{U} & \sim \prod_{i=1}^d \CRM\left( \alpha \, s^{-1}\, e^{-b(1 + U_i/b)\,s} \ddr s \otimes \crm_0^*\right), \\
	\crm_0^* \mid \bm{U} & \sim \CRM \left( \alpha_0 \, s^{-1}\, e^{-\alpha\, \lambda(\bm{U})\,s} \ddr s \otimes P_0 \right).
\end{align*}
Therefore, one needs to sample from a hierarchy of gamma CRMs in order to obtain complete samples from the posterior.

At the root of the hierarchy, approximate posterior samples from the rescaled gamma random measure $\alpha \tilde \mu_0^* \mid \bm{U}$ can be obtained through the Ferguson-Klass representation \citep{Ferguson1972}. This amounts to sampling the largest $L$ jumps of an infinitely active random measure in decreasing order, and thus provides its best finite-dimensional approximation: for a fixed truncation level, the approximation error is minimized. A straightforward approach uses the inverse L\'evy measure algorithm \citep{Wolpert1998,walker2000}.
For $\ell = 1, \dots, L$, let $\omega_{0\ell} \ge 0$ be the value solving the equation
\begin{equation*}
	\frac{\xi_\ell}{\alpha_0} = \int_{\omega_{0\ell}}^\infty s^{-1}\, e^{-\alpha\, \lambda(\bm{U})\,s} \ddr s = E_1 ( \alpha\, \lambda(\bm{U}) \,\omega_{0\ell} ) \qquad (\ell = 1, \dots, L),
\end{equation*}
where $E_1$ is the exponential integral function and $\xi_1 < \dots < \xi_L$ a.s. are the first $L$ jump times of a unit rate Poisson process, that is, $\xi_0 = 0$ and the inter-arrival times are $\xi_{\ell} - \xi_{\ell-1} \simiid {\rm Exp}(1)$, for $\ell=1,\dots,{L}$. Then the random measure $\alpha \tilde \mu_0^* \mid \bm{U}$ can be approximated as
\begin{equation*}
	\alpha \tilde \mu_0^* \mid \bm{U} \approx \sum_{\ell=1}^L (\alpha \omega_{0\ell}) \, \delta_{\phi_\ell},
\end{equation*}
where $\phi_1,\dots,\phi_L \simiid P_{0}$ are independent from the Poisson process $(\xi_1, \dots, \xi_L)$. This algorithm requires to sequentially invert the exponential integral function numerically, which is a nontrivial albeit much investigated task. Details about our implementation are provided in the following Section~\ref{app:expint}. Alternative approaches for sampling the largest $L$ jumps of a random measure in decreasing order, and their specifications for the gamma process, are explored in \cite{Campbell2019, Zhang2024}; see also \cite{Rosinski2001}.

The random measures in $\bm{\crm}^*$ are conditionally independent given $\alpha \tilde \mu_0^*$ and $\bm{U}$, and each component $\crm_i^* \mid \tilde \alpha \mu_0^*,\bm{U}$ can be approximately sampled from a gamma CRM having L\'evy intensity
\begin{equation*}
	\ddr \nu^*_i(s,x) = \sum_{\ell=1}^L (\alpha \omega_{0\ell}) \, s^{-1} e^{-b(1+U_i/b)\,s}\,\ddr s \, \ddr \delta_{\phi_\ell}(x).
\end{equation*}
The additive components of the L\'evy measure represent independent summands for $\crm_i^*$ concentrated at different fixed location. Hence, $\crm_i^* \mid \alpha \tilde \mu_0^*,\bm{U}$ can be approximated as
\begin{equation*}
	\crm_i^* \mid \alpha \tilde \mu_0^*,\bm{U} \approx \sum_{\ell=1}^L \omega_{i\ell}\,\delta_{\phi_\ell},
\end{equation*}
where $\omega_{i1}, \dots \omega_{iL}$ are independent random variables with $\omega_{i\ell} \sim \text{Gamma}(\alpha \omega_{0\ell}, b(1+U_i/b))$, for $\ell = 1, \dots, L$.
The sampling procedure to obtain an approximation of $\bm{\crm}$ by truncation of its infinite sequence of jumps is summarized in Figure~\ref{fig:structure_continuous}; the similarities with the sampling algorithms for jumps $\bm{J}$, depicted in Figure~\ref{fig:structure}, are evident. 
Note that the total mass of each random measure $\crm^*_i$, that is, the mass of the posterior random measure $\crm_i \mid \bm{X}_{1:d}$ not assigned to fixed locations, can instead be sampled exactly from a hierarchy of gamma random variables.

\begin{figure}[t]
	\centering
	\includegraphics[width=0.6\linewidth]{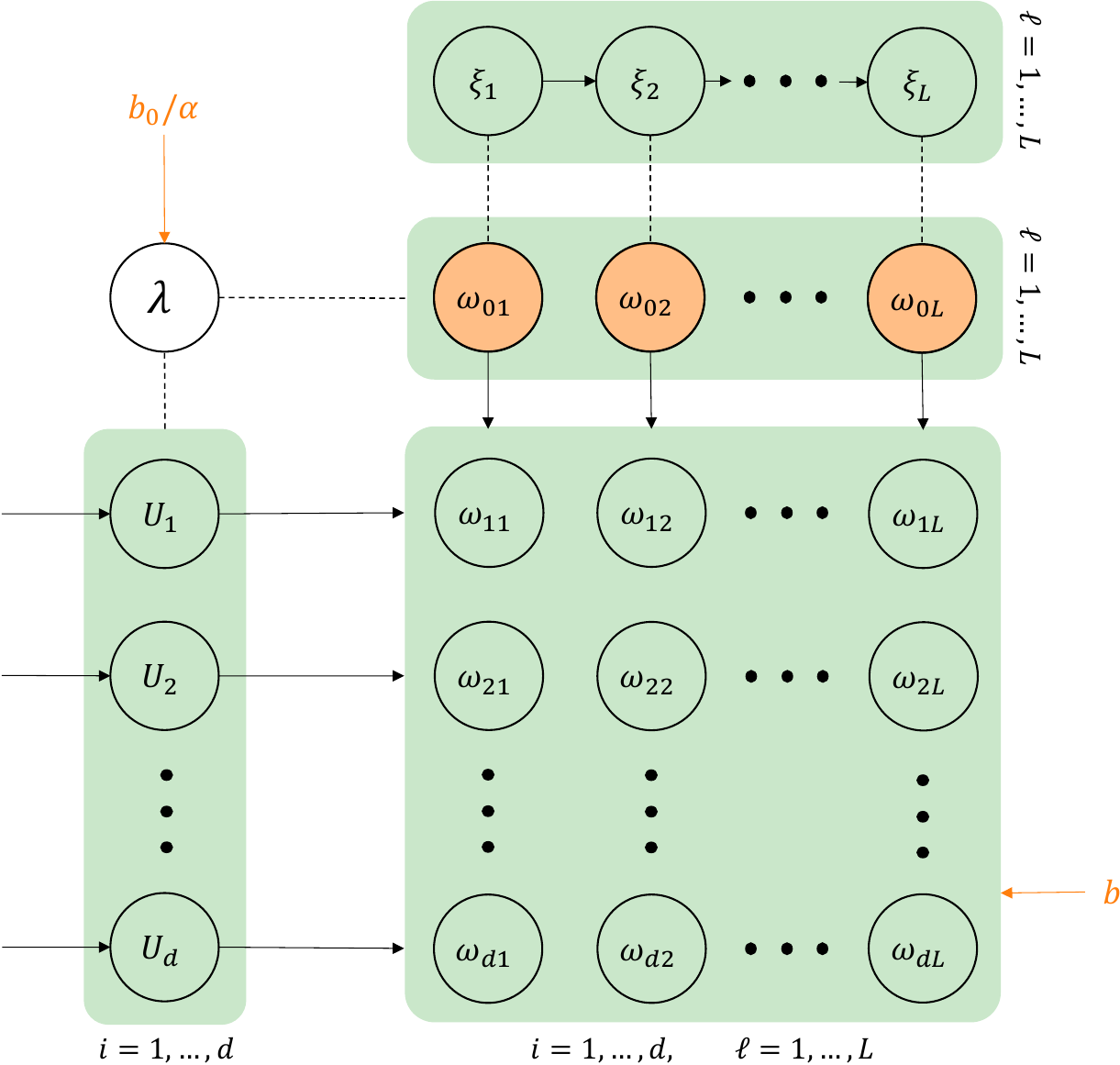}
	\captionsetup{width=0.9\textwidth,font=small}
	\caption{Conditional dependencies between random variables involved in sampling the hierarchy of gamma CRMs. Red circles represent the computational bottlenecks. The sampling scheme for the latent variables $\bm{U}$ is outlined in Section~\ref{sec:posterior_sampling_gamma} and Figure~\ref{fig:structure}. For simplicity, variables are reported up to scaling w.r.t.~model parameters.}
	\label{fig:structure_continuous}
\end{figure}

This same sampling procedure can be also employed to sample \emph{a priori} from the gamma-gamma hCRV, with minimal adjustments. Indeed, the prior construction introduced in Example~\ref{ex:gamma-gamma} coincides with the hierarchy in Proposition~\ref{prop:gamma_jumps}(a) displayed above when $U_i = 0$ for each $i = 1,\dots,d$, and thus $\lambda(\bm{U}) = b_0 / \alpha$. 
Specifically, approximate samples from the rescaled gamma random measure $\alpha \crm_0$ can be obtained via the inverse L\'evy measure algorithm, while increments of each random measure $\crm_i \mid \alpha \crm_0$ can be sampled independently from gamma distributions, as discussed in Remark~\ref{rmk:sampling}.

\subsection{Inverting the exponential integral}
\label{app:expint}

The implementation of the inverse L\'evy measure algorithm of \cite{walker2000} for the gamma CRM requires to invert the tail integrals of its L\'evy density (Section~\ref{app:sampling_random_measure}). This amounts to sequentially invert the exponential integral function $E_1$, defined for $x > 0$, 
\begin{equation*}
	E_1(x) = \int_x^{+\infty} s^{-1} e^{-s} \ddr s.
\end{equation*}
Note that $E_1$ is a strictly decreasing function, and thus is invertible. A convenient approach to find its inverse $E_1^{-1}(y)$, for a given value $y > 0$, is to determine the unique root of the function $x \mapsto E_1(x) - y$, exploiting root-finding algorithms such as Newton's method.
This method requires the evaluation of the derivative $E_1'$ of the exponential integral, which can be computed in closed form, as it equals the opposite of the integrand function, $E_1'(x) = - x^{-1}e^{-x}$. 

A well-known limitation of Newton's method is the possibility to obtain iteration values that fall outside the domain of the function, where its evaluation is not possible. This situation is particularly relevant to inverse L\'evy measure algorithms, as the tail integrals of infinitely divisible L\'evy densities diverge to $+\infty$ as the lower bound of the integration interval goes to $0$. A simple workaround that typically solves this issue is choosing a starting point $x_0$ for Newton's algorithm which falls on the left of the solution, that is $E_1(x_0) \ge y$. This can be achieved by iteratively halving an initial guess.
Remarkably, within the sequential approach required by the algorithm in Section~\ref{app:sampling_random_measure}, the standard choice for the initial guess is the solution at the previous step. However, this falls on the right of the current solution, and thus halving is always necessary.
Convergence of this algorithm is guaranteed whenever the L\'evy density is a decreasing function, and thus its tail integral is decreasing and convex. Figure~\ref{fig:newton} shows an illustration of the different behaviours of Newton's method, depending on the starting point $x_0$.

\begin{figure}[t]
	\centering
	\includegraphics[width=0.8\textwidth]{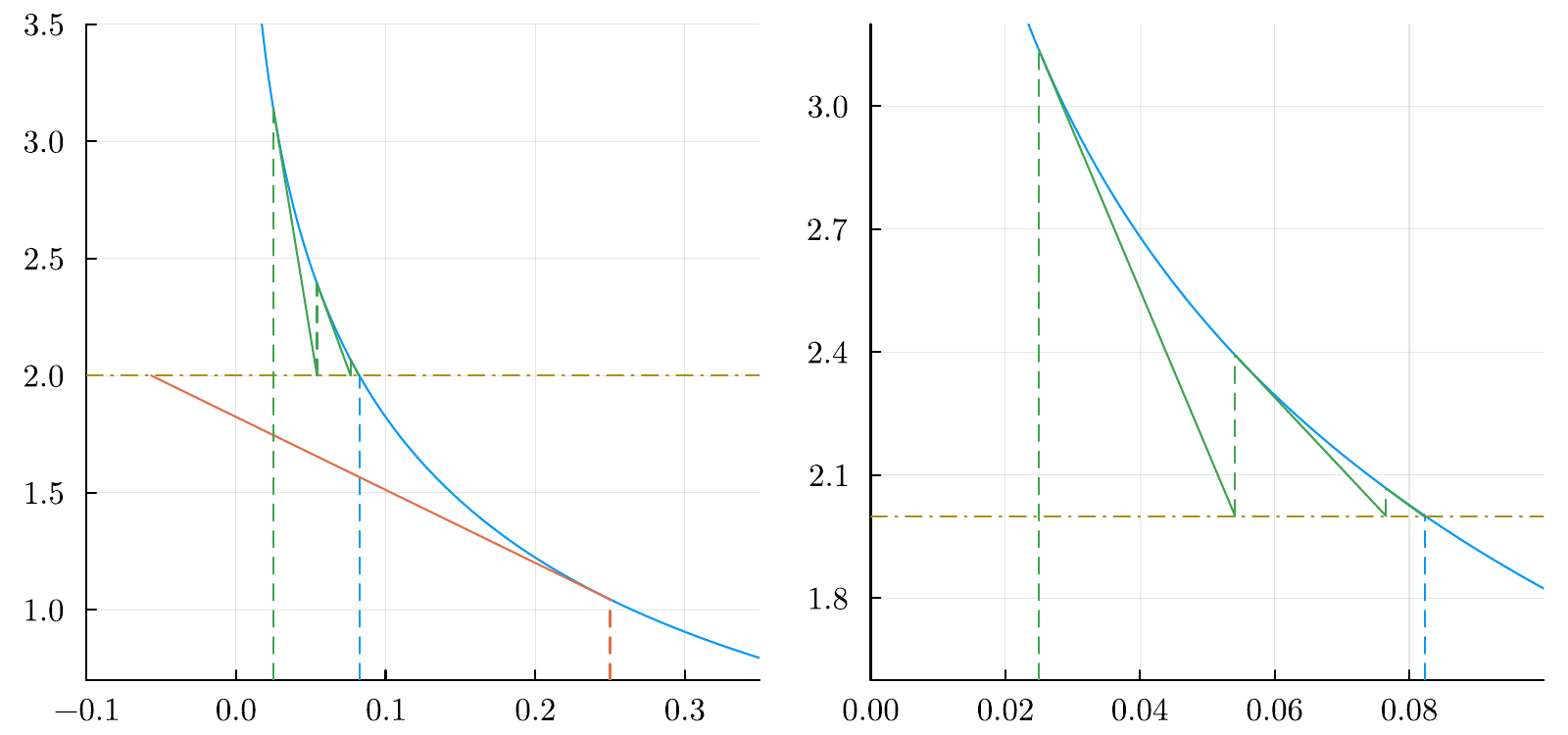}
	\captionsetup{width=0.9\textwidth,font=small}
	\caption{Iterations of Newton's method for solving the equation $E_1(x) = 2.0$, with different starting points. Starting on the right of the solution may lead the algorithm outside the domain of the function (orange); starting on the left guarantees convergence (green). The right panel zooms on converging iterations.}
	\label{fig:newton}
\end{figure}

Instead, we consider an alternative approach that improves the efficiency of Newton's method. Specifically, we redefine the exponential integral through the logarithm of its argument; in other words, we invert the function \cut
\begin{equation*}
	f(z) = E_1(e^z) = \int_{e^z}^{+\infty} s^{-1} e^{-s} \ddr s,
\end{equation*}
whose derivative is given by $f'(z) = E_1'(e^z) \,e^z = - \exp( - e^z)$. This formulation has the advantage of not being restricted to positive values, thus preventing iterations to fall outside the domain. Note that $f$ is decreasing and convex, which guarantees convergence of Newton's method for every starting point $x_0$. 

A decisive advantage of this approach is the asymptotic linearity of function $f$ as $z$ diverges to $-\infty$. Specifically, $f(z) \approx - \gamma - z$ for $z \to -\infty$, where $\gamma$ is the Euler-Mascheroni constant. This property is crucial for speeding up convergence of the numerical scheme, as the rate of convergence of Newton's method is proportional to the second derivative around the solution.
The asymptotic expansion is also useful to improve the numerical stability of function evaluations. We argue that redefining the tail integral as a function on the whole real line, removing the constraint to positive values, may be a general technique to enhance the performances of Newton's method for computing the inverse L\'evy measure, beyond the gamma process case.

\subsection{Posterior distribution of normalized random measures}
\label{app:normalized_jumps}

The posterior random probabilities arising from model \eqref{eq:model} are the normalization of the posterior random measures characterized in Theorem~\ref{th:posterior_crv} and specialized in Proposition~\ref{prop:gamma_jumps} for the gamma-gamma hCRV. 
The simplest approach to obtain posterior samples is thus normalizing the posterior samples from the corresponding random measures. In this section, we show that such posterior random probabilities are distributed as conditionally Dirichlet processes with discrete base measures. Therefore, their probability weights can be alternatively sampled from a Dirichlet distribution. This allows for a straightforward comparison with alternative samplers for the HDP, such as the marginal Gibbs samplers of \cite{Teh2006}; see the following Section~\ref{app:illustrations}.

For this purpose, denote by $\bm{\pi}$ and $\bm{\pi}^*$ the normalized jumps of $\crm_i \mid \bm{X}_{1:d}$, for $i = 1, \dots, d$, 
\begin{equation*}
	\pi_{ij} = \frac{J_{ij}}{\sum_{j=1}^k J_{ij} + \crm^*_i(\X)}, \qquad \pi^*_{i\ell} = \frac{\omega_{i\ell}}{\sum_{j=1}^k J_{ij} + \crm^*_i(\X)}, 
\end{equation*}
where the $J_{ij}$'s are jumps at fixed locations and the $\omega_{i\ell}$'s are the jumps of $\crm_i^*$, arranged in decreasing order; see Section~\ref{app:sampling_random_measure} for details on sampling from $\crm_i^*$. From Theorem~\ref{th:posterior_crv} it follows that the random probabilities $\bm{\tilde P} = \bm{\crm}/\bm{\crm}(\X)$ are distributed, a posteriori, as
\begin{equation}
	\label{eq:normalized_weights}
	\tilde P_i = \frac{\crm_i}{\crm_i(\X)} \bigm\vert \bm{X}_{1:d} \, \stackrel{d}{=} \, \sum_{j=1}^k \pi_{ij} \,\delta_{X^*_j} + \sum_{\ell \ge 1} \pi^*_{i\ell} \,\delta_{Y_\ell} \quad (i = 1, \dots, d).
\end{equation}
For the normalized gamma-gamma hCRV, the conditional distribution of each posterior random probability is a Dirichlet process with discrete base measure.

\begin{proposition}
	Let $\bm{\tilde P}$ be a normalized gamma-gamma hCRV. A posteriori, the random probabilities $\bm{\tilde P} \mid \bm{X}_{1:d}$ are conditionally independent, given variables $\alpha J_{01}, \dots, \alpha J_{0k}$ and the random measure $\crm^*_0$ in Proposition~\ref{prop:gamma_jumps}, and distributed as
	\begin{equation*}
		\tilde P_i = \frac{\crm_i}{\crm_i(\X)} \bigm\vert \bm{X}_{1:d}, \alpha \bm{J}_0, \crm^*_0 \, \stackrel{\mbox{\scriptsize{\rm ind}}}{\sim} \, {\rm DP} \left( \alpha \crm^*_0 + \sum_{j=1}^k (n_{ij} + \alpha J_{0j}) \,\delta_{X^*_j} \right) \qquad (i = 1, \dots, d).
	\end{equation*}
\end{proposition}

\begin{proof}
	By Proposition~\ref{prop:gamma_jumps}, independently for each $i = 1, \dots, d$, the random measure $\crm^*_i \mid \crm^*_0, U_i$ is a gamma CRM with shape $\alpha \crm^*_0(\X)$ and rate $b + U_i$. Moreover, for each $j = 1, \dots, k$, the random jump $J_{ij} \mid \alpha J_{0j}, U_i$ is independently gamma distributed with shape $n_{ij} + \alpha J_{0j}$ and rate $b + U_i$. Therefore, the posterior distribution of $\crm_i \mid \alpha \bm{J}_0, \crm^*_0, U_i$ is that of a gamma CRM, being a superposition of independent gamma processes with same rate. Indeed, its L\'evy intensity is
	\begin{equation*}
		s^{-1} e^{-b(1+U_i/b)} \, \bigg( \alpha \crm^*_0 + \sum_{j=1}^k (n_{ij} + \alpha J_{0j}) \,\delta_{X^*_j} \bigg) \qquad (i = 1, \dots, d).
	\end{equation*}
	The normalization of a gamma CRM is then a Dirichlet process \citep{Ferguson1973}.
\end{proof}

Remarkably, the posterior distribution of $\bm{\tilde P}$ does not depend on the prior rate parameter $b$, as already discussed right after Proposition~\ref{th:hdp_hyper}. Moreover, it is conditionally independent of latent variables $\bm{U}$, given $\alpha \bm{J}_0$ and $\crm_0^*$.

This result is particularly relevant from the algorithmic point of view, as anticipated above.
Indeed, after obtaining samples from variables $\alpha J_{01}, \dots, \alpha J_{0k}$ and an approximation by truncation of $\alpha \crm^*_0$, as described in Section~\ref{app:sampling_random_measure}, the probability weights in \eqref{eq:normalized_weights} can be sampled from a $(k+L)$-dimensional Dirichlet distribution. In particular, for each $i = 1, \dots, d$ and $L \in \mathbb{N}$,
\begin{multline*}
	\big(\pi_{i1}, \dots, \pi_{ik}, \pi^*_{i1}, \dots, \pi^*_{iL} \big) \bigm\vert \bm{X}_{1:d}, \bm{J}_0, \crm^*_0 \\
	\sim \text{Dirichlet}\big(n_{i1} + \alpha J_{01}, \dots, n_{ik} + \alpha J_{0k}, \alpha \omega_{01}, \dots, \alpha \omega_{0L} \big).
\end{multline*}
The conditional independence from $\bm{U}$ allows further parallelization of the posterior sampling schemes.
The resulting structure of conditional dependencies within the complete sampling algorithms for the normalized posterior random measures is summarized in Figure~\ref{fig:structure_normalized}.

\begin{figure}[t]
	\centering
	\includegraphics[width=\linewidth]{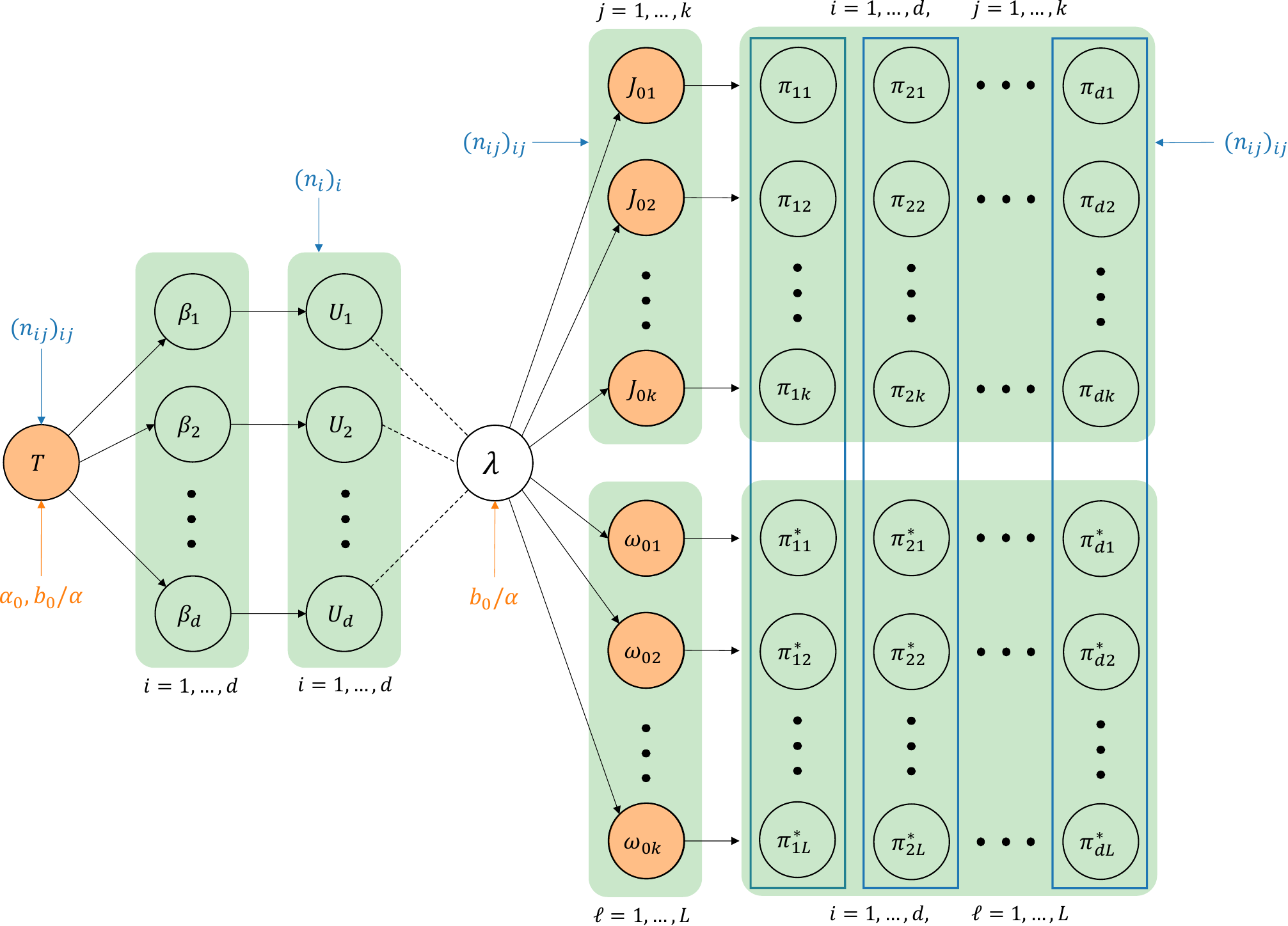}
	\captionsetup{width=0.9\textwidth,font=small}
	\caption{Conditional dependencies between random variables in the sampling algorithms for posterior normalized random measures. Red circles represent the computation bottlenecks; quantities enclosed in blue boxes are sampled from Dirichlet distributions. Variables are reported up to scaling w.r.t.~model parameters.}
	\label{fig:structure_normalized}
\end{figure}

\subsection{Marginal Gibbs samplers for the HDP}
\label{app:tables}

This section provides the details for the derivation of the marginal posterior samplers for the HDP considered in Section~\ref{sec:algorithms_gamma}, namely the CRF-based Gibbs sampler and the collapsed Gibbs sampler. 
The CRF-based Gibbs sampler is based on existing work \citep{Teh2006,Camerlenghi2019}, which we recall here using our notation. 
Proposition~\ref{prop:hdp_tables} establishes a new result, which determines our collapsed Gibbs sampler; this can be seen as an extension of the sampling schemes proposed in \citet{Teh2006} and \citet{baccallado2022}.

The hierarchical Dirichlet process was first described by \cite{Teh2006}, where the restaurant franchise metaphor is introduced. According to this metaphor, customers in each restaurant eat dishes selected from a countably infinite menu that is shared across restaurants. Moreover, in each restaurant, customers eating the same dish may or may not seat at the same table, while customers eating different dishes are necessarily seated at different tables. 
In other words, observations are organized within a nested partition structure. Remarkably, this structure is not specific to the hierarchical Dirichlet process, but characterizes every hierarchical structure of normalized random measures considered in \cite{Camerlenghi2019}; see also \cite{Sankhya2024}. Such partition is typically described through two collections of variables, displaying ties among each collection: for each restaurant $i = 1,\dots,d$, variables $X_{i\ell}$ and $Z_{i\ell}$ denote, respectively, the dish 
and table of the $\ell$-th customer. The distinct values taken by variables $\bm{X}_{1:d} = ({\bm X}_1, \dots, {\bm X}_d)$ are denoted by $X^*_1,\dots,X^*_k$, and represent the distinct dishes from the common menu.
Similarly, for each restaurant $i = 1,\dots,d$, the distinct values taken by the variables $\bm{Z}_{i} = (Z_{i1}, \dots, Z_{in_i})$, are denoted by $Z^*_{ij1},\dots,Z^*_{ijh_{ij}}$ for $j = 1, \dots, k$,
and represent the different tables serving each dish in the $i$-th restaurant. The particular choice of hierarchical random measures determines the joint probability distribution of $\bm{X}_{1:d}$ and $\bm{Z}_{1:d} = ({\bm Z}_1, \dots, {\bm Z}_d)$, namely the dishes and tables allocation, or equivalently the sequence of their predictive distributions. For a hierarchical Dirichlet process with concentration parameters $\alpha, \alpha_0 > 0$, the joint probability of tables and dishes takes the form 
\begin{equation} \label{eq:peppf}
	\Prob (\bm{X}_{1:d}, \bm{Z}_{1:d} ) = \frac{\alpha_0^k}{((\alpha_0))_h} \prod_{j=1}^k \Gamma(h_{\bullet j}) \, \prod_{i=1}^d \left( \frac{\alpha^{h_{i\bullet}}}{((\alpha))_{n_i}} \prod_{j=1}^k \prod_{r=1}^{h_{ij}} \Gamma(q_{ijr}) \right).
\end{equation}
In the expression above, $h_{ij}$ denotes the (random) number of tables in which the $n_{ij}$ customers eating dish $X^*_j$ in restaurant $i$ are partitioned, while $q_{ijr}$, for $r = 1,\dots,h_{ij}$, denotes the (random) number of customers seated at each of those tables.

Contrary to the approach of \cite{Teh2006}, which considers kernel mixtures, our model \eqref{eq:model} assumes that the dishes $\bm{X}_{1:d}$ eaten by each customer are actually observed, and thus fixed. Therefore, one only has to sample the tables $\bm{Z}_{1:d}$, in order to obtain a complete description of the nested partition. Remarkably, the full conditional distribution of each $Z_{i\ell}$, given $\bm{X}_{1:d}$ and the other $\bm{Z}_{1:d}$, can be easily derived from the joint probability \eqref{eq:peppf}. According to the restaurant franchise metaphor, having observed that the $\ell$-th customer from restaurant $i$ eats dish $X^*_j$,  they may either sit at tables $Z^*_{ij1},\dots,Z^*_{ijh_{ij}}$ or open a new table. More precisely, given $X_{i\ell} = X^*_j$, then \cut
\begin{align*}
	\Prob(Z_{i\ell} = Z^*_{ijr} \mid \cdots) & \propto \frac{q_{ijr}^{-\ell}}{\alpha + n_i}, \qquad r = 1, \dots, h_{ij}, \\
	\Prob(Z_{i\ell} = {\rm 'new'} \mid \cdots) & \propto \frac{\alpha}{\alpha + n_i} \, \frac{h_{\bullet j}^{-\ell}}{\alpha_0 + h},
\end{align*}
where $q_{ijr}^{-\ell}$ and $h_{\bullet j}^{-\ell}$ denote, respectively, the number of customers at table $Z^*_{ijr}$ and the number of tables serving dish $X^*_j$ across all restaurants, once customer $\ell$ is removed from the partition structure. As a result, a straightforward Gibbs sampler consists in sequentially resampling each $Z_{i\ell}$, for $i = 1, \dots, d$ and $\ell = 1, \dots, n_i$, from its conditional distribution. In Section~\ref{sec:algorithms_gamma}, we refer to this sampling scheme as CRF-based Gibbs sampler. Note that, since one sequentially updates $n$ latent variables, the computational cost is at least linear in the number of observations. 

The posterior distribution of the hierarchical Dirichlet process \eqref{def:hdp}, given the allocations of the observations within the nested partition, preserves the hierarchical structure \citep[e.g.,][]{Teh2010}, 
\begin{equation*}
	\tilde P_i \mid {\bm X}_{1:d}, {\bm Z}_{1:d}, \tilde P_0 \sim {\rm DP}\left( \alpha \tilde P_0 + \sum_{j=1}^k n_{ij} \delta_{X^*_j} \right); \qquad \tilde P_0 \sim {\rm DP}\left( \alpha_0 P_0 + \sum_{j=1}^k h_{\bullet j} \delta_{X^*_j} \right).
\end{equation*}
Interestingly, these distributions depend only on the numbers of customers $n_{ij}$'s for each restaurant and dish, and on the number of tables $h_{\bullet j}$'s serving each dish across all restaurants. In our setting, the counts $(n_{ij})_{ij}$ are observed, and thus the posterior distribution is fully characterized by sampling the $h_{\bullet j}$'s from their conditional distribution. 

\begin{proposition} \label{prop:hdp_tables}
	Let $\bm{\tilde P}$ be a hierarchical Dirichlet process, characterized by the probability distribution \eqref{eq:peppf}. Conditionally on the counts $(n_{ij})_{ij}$, the probability of the discrete variables $h_j = h_{\bullet j}$ for $j = 1, \dots, k$ is 
	\begin{equation*}
		\Prob(h_1,\dots,h_k) \propto \frac{\alpha^h}{((\alpha_0))_h} \prod_{j=1}^k \Gamma(h_j) \,  S(n_{1j}, \dots, n_{dj}; h_j),
	\end{equation*}
	where coefficients $S(q_1,\dots,q_d; h)$ are the generalized Stirling numbers defined in \eqref{eq:multivariate_stirling}.
\end{proposition}

\begin{proof}
	For a hierarchical Dirichlet process with concentration parameters $\alpha, \alpha_0 > 0$, the joint probability of $\bm{X}_{1:d}$ and the tables counts $(h_{ij})_{ij}$ is
	\begin{equation*}
		\frac{\alpha_0^k}{((\alpha_0))_h} \prod_{j=1}^k \Gamma(h_{\bullet j}) \, \prod_{i=1}^d \left( \frac{\alpha^{h_{i\bullet}}}{((\alpha))_{n_i}} \prod_{j=1}^k S(n_{ij}; h_{ij}) \right).
	\end{equation*}
	This expression can be obtained from \eqref{eq:peppf} by summing over the set of unordered partitions of $\set{1,\dots,n_{ij}}$ into $h_{ij}$ non-empty cycles, for each restaurant $i = 1, \dots, d$ and dish $j = 1, \dots, k$.
	The same result is derived in \citet[][Example~3]{Camerlenghi2019}, where the probability above is stated with a further marginalization over tables $(h_{ij})_{ij}$.
	Furthermore, for each dish $j$, we sum over the set of values $(h_{ij})_i$ such that $h_{1j} + \cdots h_{dj} = h_j$ and $0 \le h_{ij} \le n_{ij}$, for $i = 1, \dots, d$. Exploiting the definition of multivariate Stirling numbers \eqref{eq:multivariate_stirling_supp}, the probability distribution of $\bm{X}_{1:d}$ and the $h_j$'s is
	\begin{equation*}
		\frac{\alpha_0^k\, \alpha^h}{((\alpha_0))_h} \prod_{i=1}^d \frac{1}{((\alpha))_{n_i}} \prod_{j=1}^k \Gamma(h_{j}) \, S(n_{1j}, \dots, n_{dj}; h_{j}).
	\end{equation*}
	The conditional distribution of the $h_j$'s is proportional to the expression above.
\end{proof}

The direct sampling of the number of tables serving each dish in each restaurant was first suggested in \cite{Teh2006}, although conditionally on the random probability $\tilde P_0$. Instead, \cite{baccallado2022} adopt a marginal approach and rely on the conditional distribution of the $h_{ij}$'s for the hierarchical Pitman-Yor process, as derived in \citet[][Theorem 4]{Camerlenghi2019}. The further marginalization over restaurants in Proposition~\ref{prop:hdp_tables}
is particularly suitable for HDP, since its posterior structure depends only on the $h_{\bullet j}$'s; moreover, it conveniently reduces the dimension of the sampling space. As a result, one easily devises a Gibbs sampler for the $h_{\bullet j}$'s based on their joint distribution. In Section~\ref{sec:algorithms_gamma}, we refer to this sampling approach as collapsed Gibbs sampler. 

Finally, for a full comparison with the normalized gamma-gamma hCRV, we often consider a gamma prior on the concentration parameter $\alpha \sim {\rm Gamma}(\alpha_0, \beta)$; see Proposition~\ref{th:hdp_hyper} for the distributional equivalence result. The update of $\alpha$ within both Gibbs samplers outlined above is performed via a Metropolis-Hastings step, targeting the full conditional distribution
\begin{equation*}
	\alpha \mid {\bm X}_{1:d}, {\bm Z}_{1:d} \sim f(s) \propto s^{\alpha_0 + h - 1}\,e^{-\beta\,s} \prod_{i=1}^d \frac{1}{((s))_{n_i}}.
\end{equation*}
An alternative sampling scheme based on an augmentation with auxiliary beta and Bernoulli variables is detailed in the Appendix of \cite{Teh2006}. The Julia implementation of the CRF-based and the collapsed Gibbs samplers is available at \href{https://github.com/claudiodelsole/hCRV.jl}{github.com/claudio\\delsole/hCRV.jl}, where we also provide an interface to allow their integration within the R environment.

\subsection{Numerical illustrations}
\label{app:illustrations}

This section contains numerical illustrations 
of the posterior sampling algorithms introduced in Section~\ref{sec:algorithms_gamma}.
Specifically, we consider the MCMC sampler in Algorithm~\ref{alg:mcmc} with symmetric random-walk Metropolis-Hastings steps on the log-scale (mcmc), detailed in Section~\ref{app:mh_sampling}, and the exact sampler in Algorithm~\ref{alg:exact} (exact). As a reference to evaluate their effectiveness, we take the marginal Gibbs sampler of \cite{Teh2006} for the HDP, based on the restaurant franchise metaphor, with gamma prior on the concentration parameter; see Section~\ref{app:tables} above for details.  These algorithms target the same posterior distribution for the random probabilities (Proposition~\ref{th:hdp_hyper}), hence their posterior estimates should coincide.
We consider $d = 4$ groups of observations, each of size $n_i = 50$, sampled from independent Poisson distributions with means $2, 3, 4$ and $5$. The number of distinct values in the simulated dataset is $k = 11$. Model parameters are fixed at $\alpha_0 = \alpha = 1$ and $b_0 = b = 1$. We draw $10,000$ posterior samples for each algorithm; for the MCMC schemes, we consider a burn-in of $1,000$ steps.

\begin{figure}[t]
	\centering
	\includegraphics[width=\linewidth]{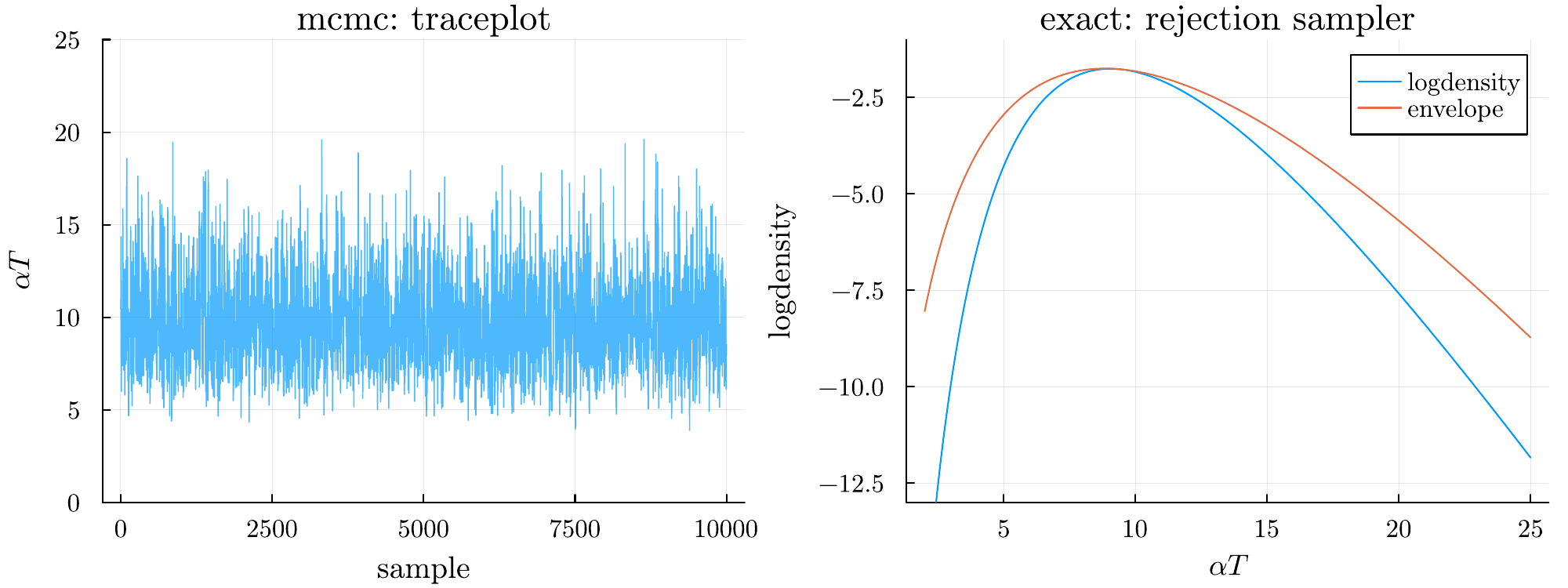}
	\includegraphics[width=.5\linewidth]{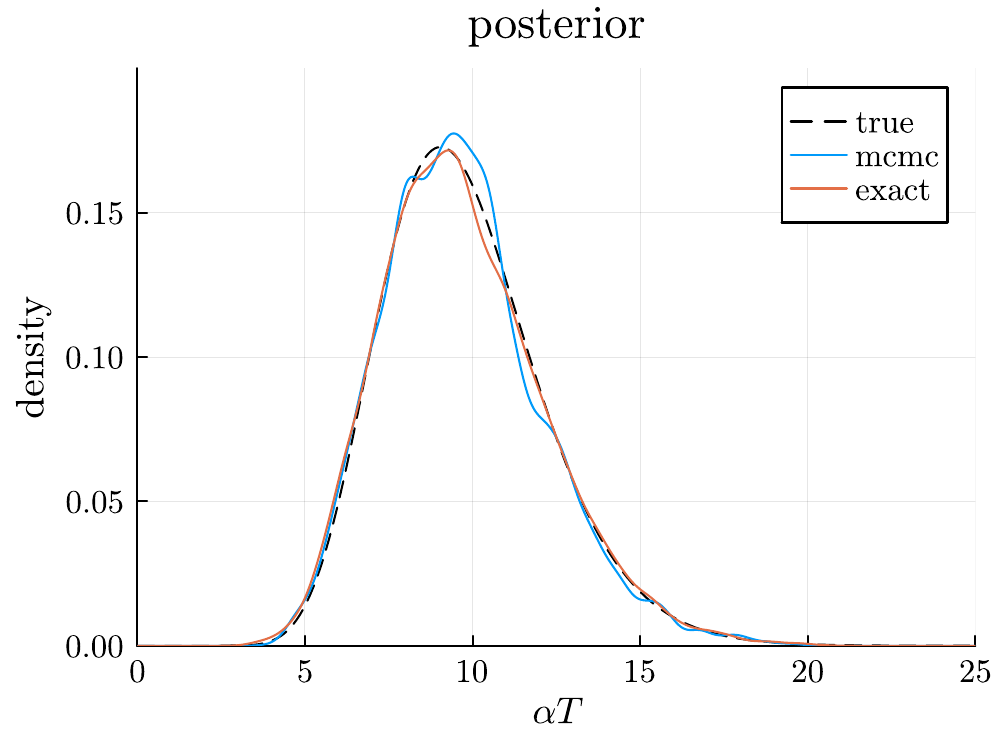}
	\captionsetup{width=0.9\textwidth,font=small}
	\caption{Diagnostics and posterior distribution for the random variable $\alpha T$. Top left: traceplot for the MCMC algorithms. Top right: envelope for the rejection sampler. Bottom: estimates of the posterior distribution via Gaussian kernel smoothing, compared with the true density.}
	\label{fig:latent}
\end{figure}

Figure~\ref{fig:latent} provides insights into the posterior sampling of random variable $\alpha T$. The MCMC algorithm shows good mixing, with an effective sample size of $1996$. The envelope for the rejection sampling in the exact algorithm is tight, and entails an acceptance rate of $0.765$; the optimal value for $r$ is $8.84$ (see Section~\ref{app:ropt}).
Instead, Figure~\ref{fig:jumps} displays the MCMC traceplot and posterior distributions for some of the $J_{0j}$'s. Specifically, we focus on the latent jumps at values $X^*_5 = 4$ and $X^*_8 = 7$, with counts $n_{\bullet 5} = 36$  and $n_{\bullet 8} = 7$. Again, the MCMC algorithms show good mixing, with effective sample sizes above $1,600$ for all $J_{0j}$'s.
Finally, the posterior distributions for the jumps $J_{ij}$ at values $X^*_5 = 4$ and $X^*_8 = 7$ for the $d = 4$ groups are displayed in Figure~\ref{fig:post}. The effective sample sizes are above $3,300$ for all $J_{ij}$'s.

\begin{figure}[t]
	\centering
	\includegraphics[width=\linewidth]{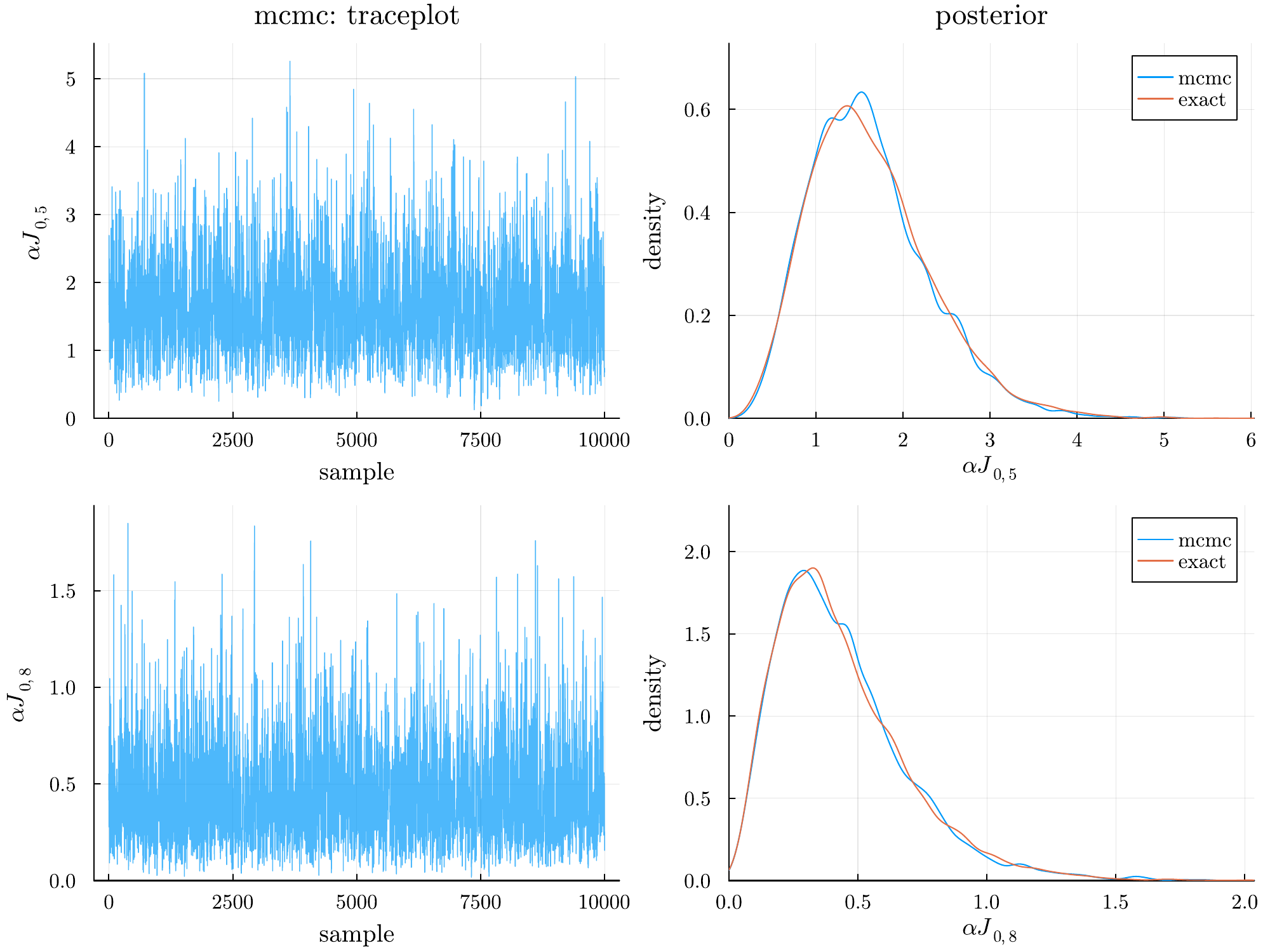}
	\captionsetup{width=0.9\textwidth,font=small}
	\caption{Diagnostics and posterior distributions for the latent jumps variables $J_{0j}$ at values $X^*_5 = 4$ and $X^*_8 = 7$. Left: traceplots for the MCMC algorithm. Right: estimates of the posterior distributions via Gaussian kernel smoothing.}
	\label{fig:jumps}
\end{figure}

\begin{figure}
	\centering
	\makebox[\textwidth]{\includegraphics[width=1.25\linewidth]{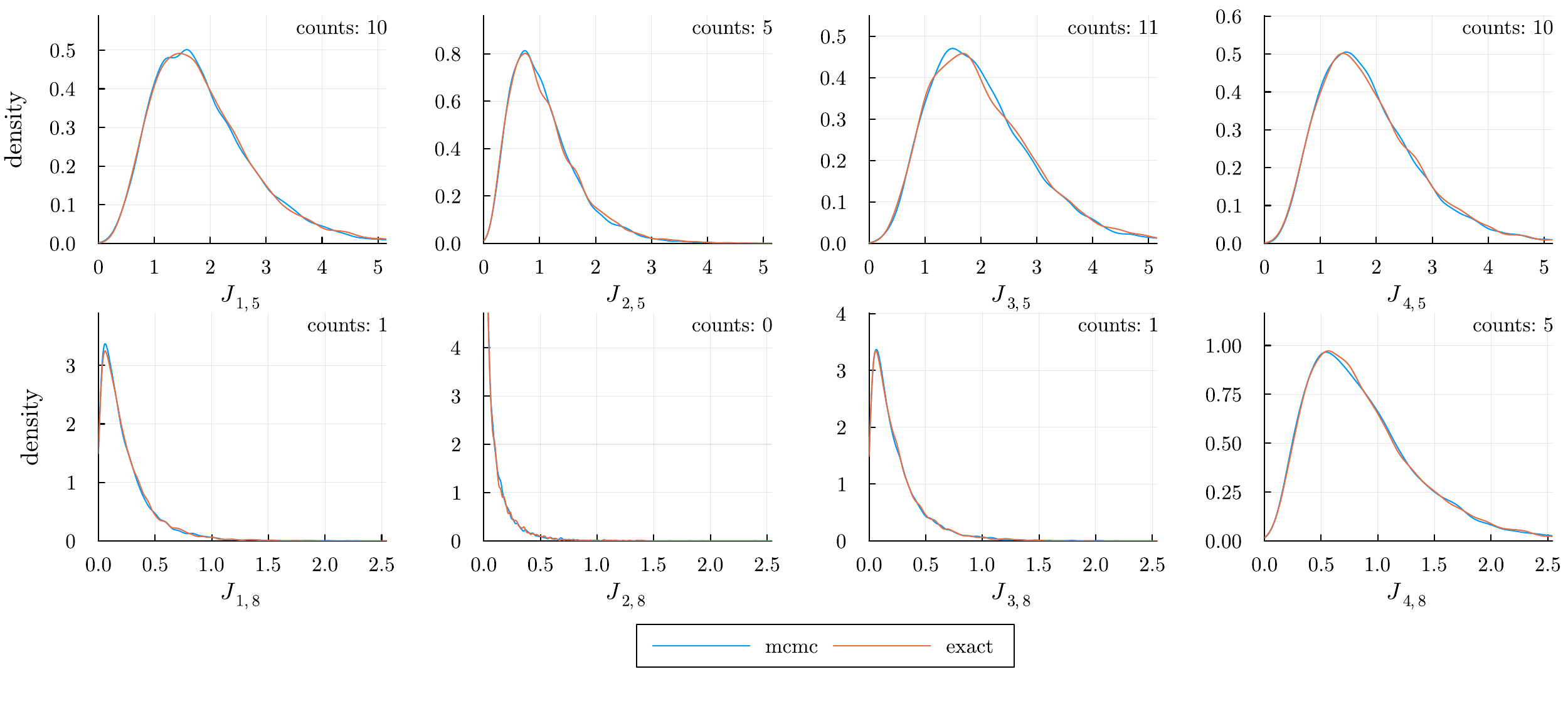}}
	\captionsetup{width=\textwidth,font=small}
	\caption{Posterior distributions for the jumps $J_{ij}$ at values $X^*_5 = 4$ and $X^*_8 = 7$, for the $d = 4$ groups, obtained via Gaussian kernel smoothing.}
	\label{fig:post}
\end{figure}

\begin{figure}
	\centering
	\makebox[\textwidth]{\includegraphics[width=1.25\linewidth]{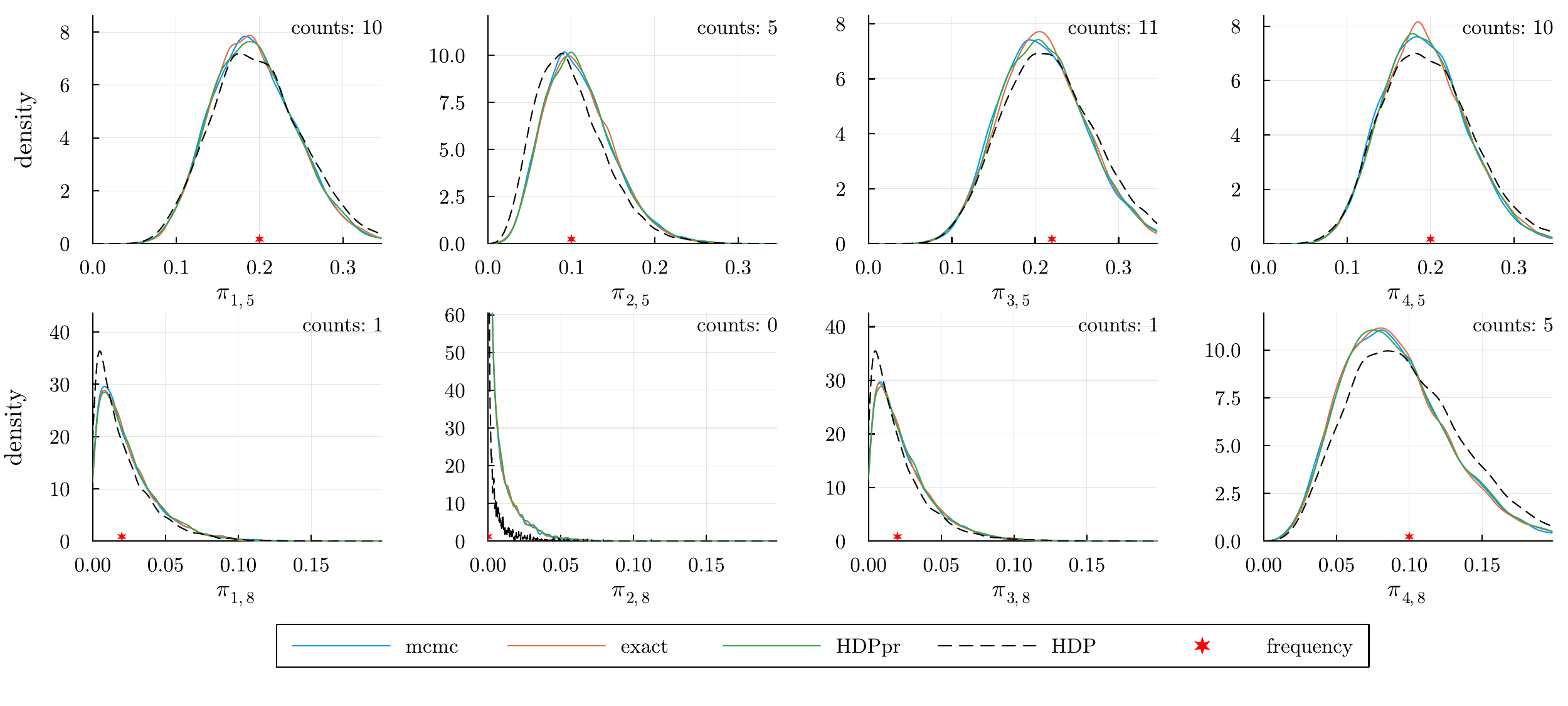}}
	\captionsetup{width=\textwidth,font=small}
	\caption{Posterior distributions for the probability weights $\pi_{ij}$ at values $X^*_5 = 4$ and $X^*_8 = 7$, for the $d = 4$ groups. The HDP model (without prior) targets different posterior distributions.}
	\label{fig:post_normalized}
\end{figure}

Furthermore, we conduct a visual comparison with the marginal Gibbs sampler for the HDP in terms of the accuracy of posterior distributions. Since this algorithm outputs posterior probability weights, we sample from the posterior random probabilities according to the procedure outlined in~\ref{app:normalized_jumps}. Considering the same experimental setting described above, we compare the MCMC sampler, the exact sampler, and the marginal Gibbs sampler for HDP, with or without gamma prior on the concentration parameter (HDP and HDPpr, respectively).
The posterior distributions for the probability weights $\pi_{ij}$ at the same values $X^*_5 = 4$ and $X^*_8 = 7$ are displayed in Figure~\ref{fig:post_normalized}. The effective sample sizes consistently exceed $4,000$. As expected, estimates of the posterior distributions for the HDP model without prior visibly differ from those obtained with the other algorithms, which instead target the same posterior distributions. This difference may be likewise observed in the right plot of Figure~\ref{fig:comparison}.

\subsection{Comparison for increasing number of distinct values}
\label{app:timesk}

Section~\ref{sec:simulations} of the main manuscript compares the different algorithms for posterior sampling through simulation studies. The comparison is performed in terms of computational time per effective sample, as the number of groups $d$ or the number of observations per group $n_i$ increase. In this section, we conduct a further simulation study fixing both the number of groups $d = 10$ and the number of observations per group $n_{i} = 25$, while allowing the number of distinct values $k$ to vary. To enhance the variability of $k$ across simulated datasets, observations are sampled from a hierarchical Dirichlet process with random concentration parameters $\alpha \sim \text{Gamma}(5)$ and $\alpha_0 \sim \text{Gamma}(3)$. Consistently with Section~\ref{sec:simulations}, results are averaged over $100$ simulated datasets for each value of $k$, and execution times are plotted whenever at least $75$ experiments are completed without errors. Remarkably, numerical overflows in the computation of the multivariate Stirling numbers appear more frequently for smaller values of $k$; indeed, since the number of observations $n = 250$ is fixed, a smaller $k$ entails a larger number of observations for each distinct value. As displayed in Figure~\ref{fig:times_k}, the execution time for the MCMC algorithm grows linearly with $k$: this is in line with the fact that the state of the Markov chain has dimension $2k+1$. The exact sampler and collapsed Gibbs sampler also display a nearly linear dependence on the number of distinct values, with a smaller slope compared with the MCMC approach. In contrast, the execution time for the CRF-based Gibbs sampler (not shown) is essentially constant. 

\begin{figure}
	\centering
	\includegraphics[width=0.5\textwidth]{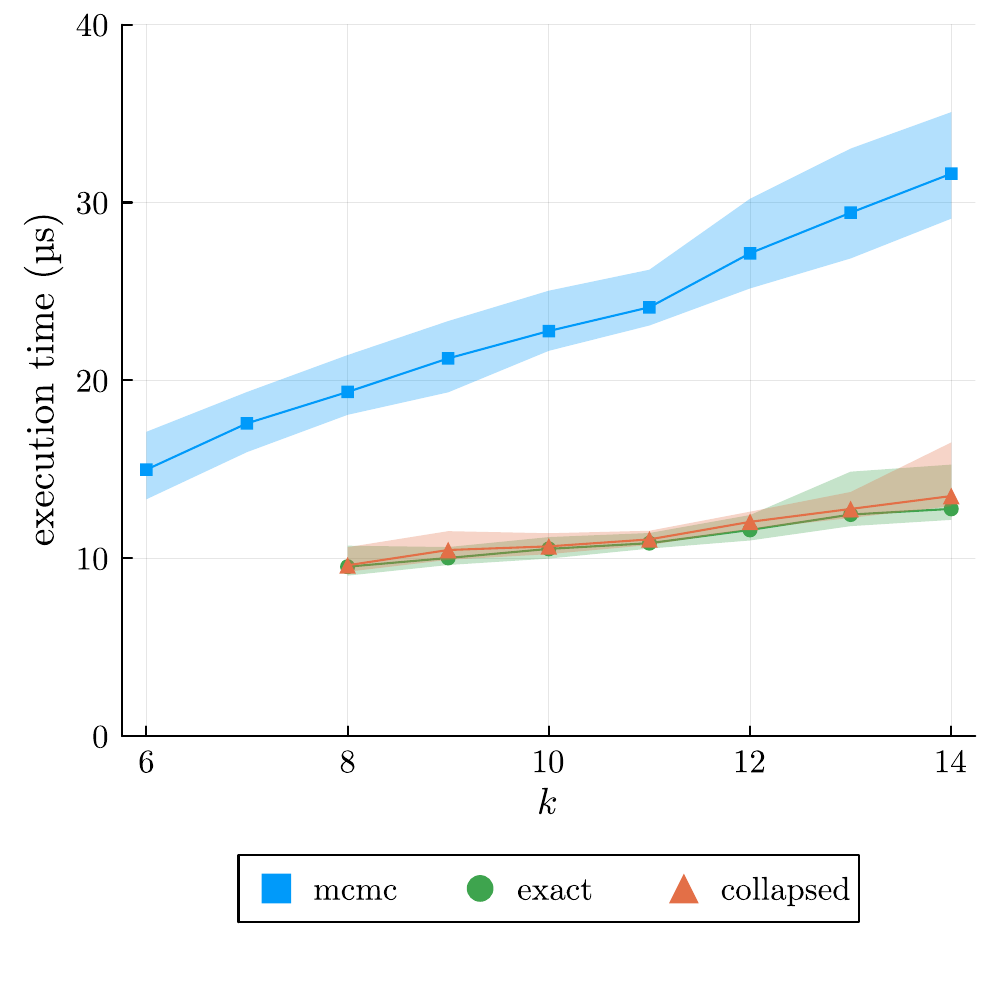}
	\captionsetup{width=0.9\textwidth,font=small}
	\caption{Execution times per effective sample for the different algorithms, with increasing number of distinct values $k$. The number of groups $d = 10$ and observations per group $n_{i} = 25$ are fixed. Results are averaged over  $100$ simulated datasets per experimental setting. Solid curves represent median values, with shaded areas between the first and third quartiles. Times for the CRF-based sampler are around $120 {\rm \mu s}$.}
	\label{fig:times_k}
\end{figure}


\section{Eliciting the dependence structure}
\label{sec:elicitation}

Normalized hierarchical CRVs induce dependence between the marginal random probabilities, which in turn regulates the borrowing of information across different groups. Hierarchical models are a natural way to induce positive dependence, especially in a Bayesian setting. However, as highlighted in \cite{Sankhya2024}, the elicitation of the dependence is more difficult with respect to other proposals because model parameters typically affect both the marginal distribution and the dependence across groups. For this reason, \cite{Sankhya2024} propose two kinds of (weak) flexibility for partially exchangeable models: (i) for every $\rho \in [0,1]$, there exists a specification of the parameters such that $\corr(\tilde P_i(A), \tilde P_j(A)) = \rho$, where $A$ is a fixed set s.t. $E(\tilde P_i(A)) \neq 0,1$; (ii) for every $\rho \in [0,1]$ and for every fixed value of the marginal mean $E(\tilde P_i(A))$ and variance $\var(\tilde P_i(A))$, there exists a specification of the parameters s.t. $\corr(\tilde P_i(A), \tilde P_j(A)) = \rho$. The expressions in Example~\ref{ex:moments_gamma} show that the normalized gamma-gamma hCRV achieves the flexibility of second kind. In this section, we investigate the role of parameter values in the borrowing of information, for fixed values of the marginal mean and variance.

Let $\bm{\crm}$ be a gamma-gamma hCRV as in Example~\ref{ex:gamma-gamma} and let $A$ be a Borel set s.t. $P_0(A) \neq 0,1$. From Example~\ref{ex:moments_gamma}, it is apparent that correlation is not affected by the mean $E(\tilde P_i(A)) = P_0(A)$ but is deeply related to the variance. Indeed, one has
\begin{align*}
	\corr ( \tilde P_i(A),  \tilde P_j(A)) & = \left(1 + \frac{\alpha_0}{\alpha} e^{1/\alpha} E_{\alpha_0}\big(1/\alpha\big)\right)^{-1}, \\
	\var ( \tilde P_i(A) ) & = \left(1 + \frac{\alpha_0}{\alpha} e^{1/\alpha} E_{\alpha_0}\big(1/\alpha \big)\right) \frac{ P_0(A)(1-P_0(A))}{1+\alpha_0}.
\end{align*}
We refer to Table~\ref{table:limits} for some limiting behaviours, which are 
compared with those of the HDP. 
Since $\alpha$ and $\alpha_0$ impact both the variance and the dependence structure, it may be difficult to elicit them in practice. From the point of view of the practitioner, it is certainly more intuitive to elicit mean, variance and correlation independently. Hence, one can choose $\sigma^2, \rho \in (0,1)$ such that $\var(\tilde P_i(A)) = \sigma^2 \, P_0(A)(1-P_0(A))$ and $ \corr(\tilde P_i(A), \tilde P_j(A)) = \rho$, and then find the corresponding values of $\alpha$ and $\alpha_0$ by solving the system of non-linear equations in Section~\ref{sec:comparison}, namely
\begin{equation*}
	\rho(1+\alpha_0/\alpha \, e^{1/\alpha} E_{\alpha_0}(1/\alpha)) - 1 = 0, \qquad \sigma^2 (1 + \alpha_0) - 1/\rho = 0,
\end{equation*}
Figure~\ref{fig:elicitation} displays the values of $\alpha,\alpha_0$ that correspond to a range of values for $\sigma^2$ and $\rho$, along with the corresponding analysis for the HDP.

\begin{table}
	\centering
	\begin{tabular}{ ccccc } 
		parameters \quad & \multicolumn{2}{c}{ fixed $\alpha>0$} & \multicolumn{2}{c}{fixed $\alpha_0>0$} \\ 
		& $\alpha_0 \to 0$ & $\alpha_0 \to +\infty$ & $\alpha \to 0$ & $\alpha \to +\infty$ \\ 
		\hline 
		$\sigma^2 = \sigma^2(\alpha, \alpha_0)$ \quad & $1$  & $0$   & $1$  & $1$ \\
		$\rho = \rho(\alpha, \alpha_0)$ \quad & $1$   & $\alpha/(1+\alpha)$   & $1/(1+\alpha_0)$  & $1$ \\
		\hline
	\end{tabular}
	
	\vspace{.5em}
	\begin{tabular}{ ccccc } 
		parameters \quad & \multicolumn{2}{c}{fixed $\alpha>0$} & \multicolumn{2}{c}{fixed $\alpha_0>0$} \\ 
		& $\alpha_0 \to 0$ & $\alpha_0 \to +\infty$ & $\alpha \to 0$ & $\alpha \to +\infty$ \\ 
		\hline
		$\sigma^2 = \sigma^2(\alpha,\alpha_0)$ \quad & $1$  & $1/(1+\alpha)$   & $1$  & $1/(1+\alpha_0)$ \\
		$\rho = \rho(\alpha,\alpha_0)$ \quad & $1$ & $0$   & $1/(1+\alpha_0)$  & $1$ \\
		\hline
	\end{tabular}
	
	\captionsetup{width=0.9\textwidth,font=small}
	\caption{\small Limiting behaviours of variance and correlation parameters for the normalized gamma-gamma hCRV (top) and HDP (bottom), as functions of their respective parameters $\alpha,\alpha_0>0$, with $\var(\tilde P_i(A)) = \sigma^2 P_0(A)(1 - P_0(A))$ and $\corr(\tilde P_i(A), \tilde P_j(A)) = \rho$.}
	\label{table:limits}
\end{table}


\begin{figure}
	\centering
	\includegraphics[width=\linewidth]{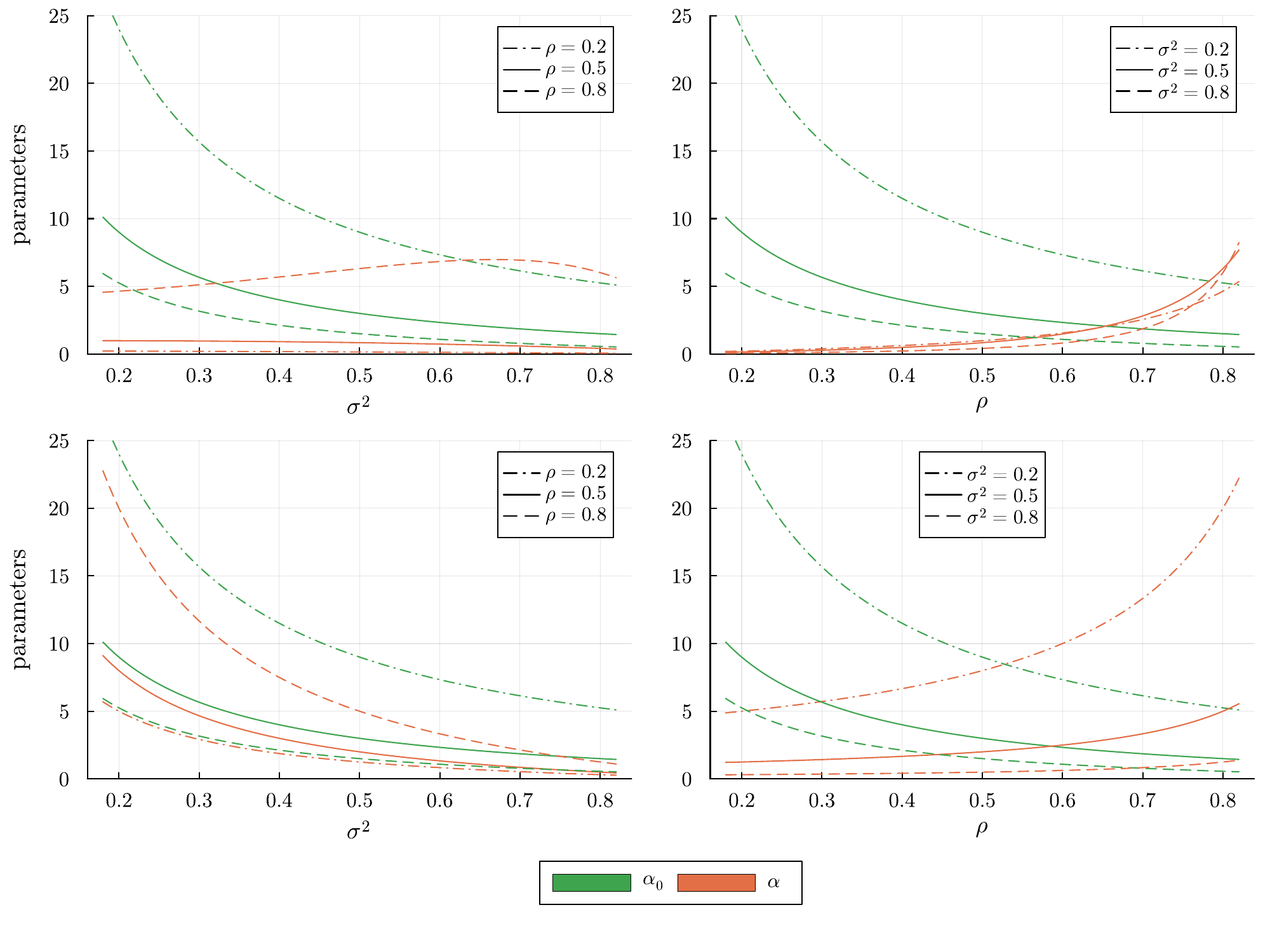}
	\captionsetup{width=0.91\textwidth,font=small}
	\caption{Values of parameters $\alpha$ and $\alpha_0$ for the normalized gamma-gamma hCRV (top) and HDP (bottom) corresponding to fixed values of variance $\sigma^2$ and correlation $\rho$.}
	\label{fig:elicitation}
\end{figure}

To visualize the effect of the borrowing, we consider the same simulated dataset as in Section~\ref{sec:comparison}, namely $d = 3$ groups of independent Poisson observations, each of size $n_{i}=10$, with means equal to $2$, $3$ and $4$.  Figure~\ref{fig:borrowing} displays the expected values of the posterior random means $E \big( \int x \,  \ddr \tilde P_i(x) \bigm\vert \bm{X}\big)$, which coincide with the means of the predictive distributions for the three groups, as we vary the prior pairwise correlation coefficient $\rho \in (0,1)$ and the prior variance through $\sigma^2\in (0,1)$, keeping the base measure $P_0 = N(10,1)$ fixed. Posterior samples are obtained using the exact algorithm. For a fixed variance $\sigma^2$, a higher correlation $\rho$ induces more borrowing, and thus the posterior means are closer to each other. Depending on $\sigma^2$, the estimates are closer to the prior (which is pushing them towards higher values, being centered at $10$), or to their empirical means. 
On the other hand, for a fixed correlation $\rho$, a higher variance $\sigma^2$ reduces the weight of the prior, and thus pushes the estimates towards their empirical means. Interestingly, lower values of $\rho$ can induce estimates that are closer to each other than higher values of $\rho$, depending on the value of $\sigma^2$. Indeed, low values of $\sigma^2$ will force estimates to be closer to the prior, and thus also closer to each other; e.g., compare $\rho = \sigma^2 = 0.1$ vs.~$\rho = 0.5$ and $\sigma^2 = 0.9$. This is the effect of the shrinkage, which can sometimes be difficult to distinguish from the borrowing of information, especially if the prior mean is close to the grand mean of the observations.

\begin{figure}
	\centering
	\includegraphics[width=\linewidth]{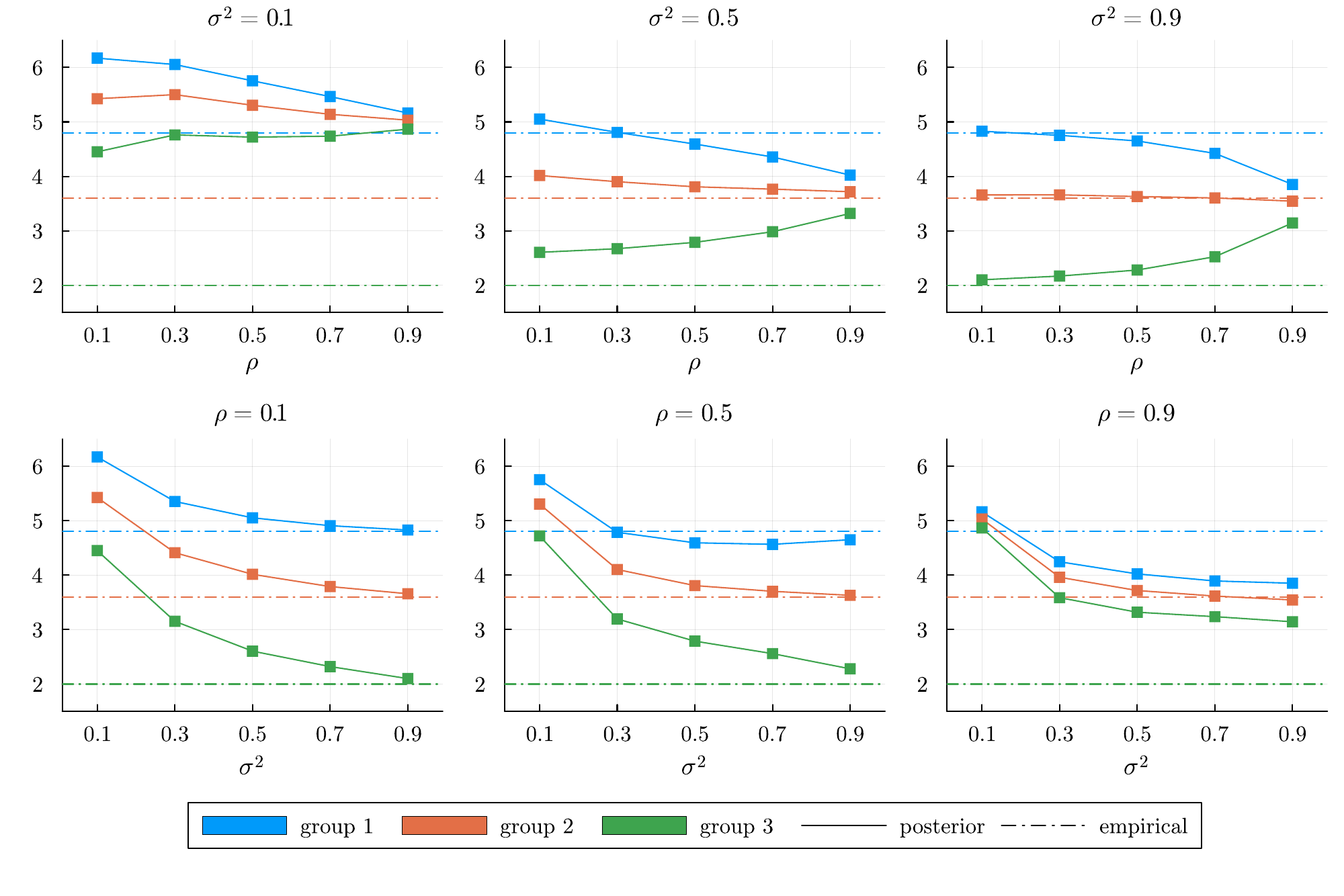}
	\captionsetup{width=0.9\textwidth,font=small}
	\caption{Posterior expected random means of the normalized gamma-gamma hCRV model, for three groups of independent Poisson observations with means equal to $2$, $3$ and $4$, each of size $n_{i}=10$, and prior mean $P_0 = N(10,1)$. Top: fixed variance and increasing correlation. Bottom: fixed correlation and increasing variance.} 
	\label{fig:borrowing}
\end{figure}

\end{document}